\documentclass[10pt]{amsart}

\usepackage{amssymb,amsmath}
\usepackage{amsthm,mathabx}
\usepackage{qsymbols}
\usepackage{graphicx}
\usepackage{xcolor}
\usepackage{enumitem}

\usepackage[utf8]{inputenc}
\usepackage[T1]{fontenc}
\usepackage{mathtools}
\usepackage{hyperref}

\newtheorem*{intro_thm}{Theorem}
\newtheorem{thm}{Theorem}[section]
\newtheorem{lem}[thm]{Lemma}
\newtheorem{prop}[thm]{Proposition}
\newtheorem{cor}[thm]{Corollary}
\newtheorem{conj}[thm]{Conjecture}

\theoremstyle{definition}
\newtheorem{defi}[thm]{Definition}

\theoremstyle{remark}
\newtheorem{rem}[thm]{Remark}

\newcommand{\R}{\mathbb{R}}

\renewcommand{\P}{\mathbb{P}}
\newcommand{\E}{\mathbb{E}}
\newcommand{\ep}{\varepsilon}
\renewcommand{\L}{\mathcal{L}}
\newcommand{\M}{\mathcal{M}}
\newcommand{\F}{\mathcal{F}}
\newcommand{\f}{\widehat{f}}

\newcommand{\Sd}{\mathbb{S}}
\newcommand{\mc}{\mathcal}

\newcommand{\Span}{\operatorname{Span}}
\newcommand{\conv}{\operatorname{Conv}}

\newcommand{\dom}{\operatorname{dom}}
\newcommand{\inte}{\operatorname{int}}
\newcommand{\supp}{\operatorname{supp}}
\newcommand{\bary}{\operatorname{bar}}
\newcommand{\san}{\operatorname{San}}

\DeclarePairedDelimiter{\norm}{\lVert}{\rVert} 
\DeclarePairedDelimiter{\abs}{\lvert}{\rvert}
\DeclarePairedDelimiter{\sbra}{[}{]}
\DeclarePairedDelimiter{\pare}{(}{)}

\makeatletter

\let\oldabs\abs
\def\abs{\@ifstar{\oldabs}{\oldabs*}}
\let\oldnorm\norm
\def\norm{\@ifstar{\oldnorm}{\oldnorm*}}
\let\oldsbra\sbra
\def\sbra{\@ifstar{\oldsbra}{\oldsbra*}}
\let\oldpare\pare
\def\pare{\@ifstar{\oldpare}{\oldpare*}}
\makeatother

\begin{document}
\title{Transport-entropy forms of direct and converse Blaschke-Santal\'o inequalities}

\begin{abstract}
We explore alternative functional or transport-entropy  formulations of the Blaschke-Santal\'o inequality and of its conjectured counterpart due to Mahler. In particular, we obtain new direct and reverse Blaschke-Santal\'o inequalities for $s$-concave functions. We also obtain new sharp symmetrized transport-entropy inequalities for a large class of spherically invariant probability measures, including the uniform measure on the unit Euclidean sphere and generalized Cauchy and Barenblatt distributions. Finally, we show that the Mahler's conjecture is equivalent to some reinforced log-Sobolev type inequality on the sphere.
\end{abstract}

\author{Matthieu Fradelizi}
\address{LAMA, Univ Gustave Eiffel, Univ Paris Est Creteil, CNRS, F-77447 Marne-la-Vall\'ee, France.}
\email{matthieu.fradelizi@univ-eiffel.fr}

\author{Nathael Gozlan}
\address{Université Paris Cité, CNRS, MAP5, F-75006 Paris}
\email{nathael.gozlan@u-paris.fr}

\author{Shay Sadovsky}
\address{School of Mathematical Sciences, Tel Aviv University, Tel Aviv 69978, Israel}
\email{shaysadovsky@mail.tau.ac.il}

\author{Simon Zugmeyer}
\address{UMPA UMR5669, Lyon, France}
\email{simon.zugmeyer@ens-lyon.fr}



\date{\today}

\thanks{The first and second named author are supported by a grant of the Simone and Cino Del Duca foundation. The third named author was partially supported by the Chateaubriand training research fellowship program, and is also grateful to the Azrieli foundation for the award of an Azrieli fellowship. The fourth named author has benefited from a post doctoral position funded by the Simone and Cino Del Duca foundation. This research has been conducted within the FP2M federation (CNRS FR 2036).}
\maketitle

\tableofcontents

\section{Introduction}
The classical Blaschke-Santal\'o inequality \cite{blaschke_book,san49} states that if $K \subset \R^n$ is a convex body, then there exists $z\in\R^n$ such that
\begin{equation}\label{eq:BS-intro}
|K||(K-z)^\circ|\le |B_2^n|^2,
\end{equation}
where the polar of a set $A \subset \R^n$ is defined by $A^\circ = \{y\in\R^n; \langle x, y\rangle\le 1, \forall x\in A\}$, and $B_2^n$ denotes the Euclidean unit ball of $\R^n.$ Equality holds in \eqref{eq:BS-intro} if and only if $K$ is an ellipsoid. Moreover, if one of the convex bodies $K$ or $K^\circ$ has its barycenter at $0$ (which is for instance the case for centrally symmetric convex bodies), then \eqref{eq:BS-intro} holds with $z=0$.

The inequality \eqref{eq:BS-intro} admits a functional version, first proved by Ball \cite{ball86} in the case of even functions, and then extended to arbitrary functions by Artstein-Avidan, Klartag and Milman \cite{AAKM05}:  for any function $\varphi: \R^n\to\R\cup\{+\infty\}$ there exists $z\in\R^n$ such that 
\begin{equation}\label{eq:BSfunction-intro}
\int e^{-\varphi}\,dx\int e^{-(\varphi_z)^*}\,dx\le (2\pi)^n,
\end{equation}
where, $\varphi_z(x)=\varphi(x+z)$, $x\in \R^n$, and the Fenchel-Legendre transform of a function $f:\R^n\to \R\cup\{+\infty\}$ is defined by
\[
f^*(y)=\sup_{x\in \R^n} \{\langle x, y\rangle - \varphi(x)\},\qquad y\in \R^n.
\]
Lehec \cite{lehec3} gave another proof of inequality \eqref{eq:BSfunction-intro} and showed that, if $\int xe^{-\varphi(x)}\,dx=0$, then  \eqref{eq:BSfunction-intro} holds with $z=0$. One sees that \eqref{eq:BSfunction-intro} gives back \eqref{eq:BS-intro} by taking $\varphi = \frac{\|\,\cdot\,\|_K^2}{2}$.

Recently, a sharp form of Talagrand transport-entropy inequality for the Gaussian standard measure $\gamma$ on $\R^n$  has been deduced  from \eqref{eq:BSfunction-intro} by Fathi \cite{fat18}. More precisely, for all probability measures $\nu_1,\nu_2$ on $\R^d$ with $\nu_2$ centered, it holds 
\begin{equation}\label{eq:Talagrand-intro}
W_2^2(\nu_1,\nu_2) \leq 2 H(\nu_1 |\gamma) + 2H(\nu_2 | \gamma),
\end{equation}
where $W_2$ denotes the usual quadratic Wasserstein distance (with respect to the usual Euclidean norm $|\,\cdot\,|$ on $\R^n$), defined by 
\[
W_2^2(\nu_1,\nu_2) = \inf \E[|X_1-X_2|^2],
\]
where the infimum runs over random vectors satisfying $X_1 \sim \nu_1$ and $X_2 \sim \nu_2$, and $H(\,\cdot\, | \mu)$ denotes the relative entropy functional with respect to some measure $\mu$ on $\R^n$, and is defined by 
\[
H(\nu | \mu) = \int \log \frac{d\nu}{d\mu}\,d\mu,
\]
whenever $\nu$ is absolutely continuous with respect to $\mu$, and $+\infty$ if this is not the case.
Choosing $\nu_2 = \gamma$, Inequality \eqref{eq:Talagrand-intro} immediately gives back the following classical  inequality obtained by Talagrand in \cite{Tal96}: for all probability measures $\nu_1$ on $\R^n$
\begin{equation}\label{eq:Talagrand2-intro}
W_2^2(\nu_1,\gamma) \leq 2 H(\nu_1 |\gamma).
\end{equation}
Without centering assumptions on $\nu_2$, the following inequality can be easily deduced from \eqref{eq:Talagrand2-intro}: for all probability measures $\nu_1,\nu_2$ on $\R^n$,
\begin{equation}\label{eq:Talagrand3-intro}
W_2^2(\nu_1,\nu_2) \leq 4 H(\nu_1 |\gamma) + 4H(\nu_2 | \gamma).
\end{equation}
Interestingly, Inequalities \eqref{eq:Talagrand-intro},  \eqref{eq:Talagrand2-intro} and \eqref{eq:Talagrand3-intro} are all sharp. We refer to \cite{Ledoux} or \cite{GL10} for applications of transport-entropy inequalities to the concentration  of measure phenomenon. 

The first main objective of this paper is to extend the preceding results to other model probability spaces than the Gaussian space $(\R^n,|\,\cdot\,|,\gamma)$. For that purpose, we will rely on a more general functional version of the Blaschke-Santal\'o inequality that we shall now present. The functional inequality \eqref{eq:BSfunction-intro} is actually a particular case of the following result first proved by Ball \cite{ball86} for even functions, then by the first named author and Meyer \cite{fm07} for log-concave functions and finally extended by Lehec \cite{Lehec1} to arbitrary measurable functions: if $f: \R^n \to \R_+$ is integrable, then there exists a point $z \in \R^n$ such that for any measurable function $g:\R^n \to \R_+$ satisfying 
\[
f(x+z)g(y) \leq \rho(\langle x,y \rangle)^2 ,\qquad \forall x,y\in \R^n \text{ such that } \langle x,y \rangle >0,
\]
it holds 
\begin{equation}\label{eq:BSfunction-gen-intro}
\int f(x)\,dx \int g(y)\,dy \leq \left(\int \rho(|x|^2)\,dx\right)^2,
\end{equation}
where  $\rho : \R_+ \to \R_+$ is some weight function such that $\int \rho(|x|^2)\,dx<+\infty$. As first proved by Ball \cite{ball86}, if $f$ is even, then $z$ can be chosen to be $0$. Inequality \eqref{eq:BSfunction-intro} corresponds to the weight function $\rho_0(t) = e^{-t/2}$, $t\geq0$. 

In the spirit of Fathi's version of Talagrand's inequality \eqref{eq:Talagrand-intro}, we show in Theorem \ref{thm:Talagrand-noneven-mu} that the general functional version of the Blaschke-Santal\'o inequality \eqref{eq:BSfunction-gen-intro} implies sharp transport-entropy inequalities for a class of spherically invariant probability measures that contains the standard Gaussian as a particular case. More precisely, we prove the following result in Theorem~\ref{thm:Talagrand-noneven-mu}: 
\begin{intro_thm}
    If $\rho :\R_+\to(0,\infty)$ is a continuous non-increasing function such that $\int \rho(|x|^2)\,dx<+\infty$, and $t\mapsto -\log\rho(e^t)$ is convex on $\R$, then the probability measure 
    \[  
    \mu_\rho(dx)= \frac{\rho(\abs x^2)}{\int \rho(\abs y^2)\,dy}\,dx
    \]
    satisfies the following inequality: for all $\nu_1,\nu_2\in \mc P(\R^n)$ with $\nu_1$ and $\nu_2$ symmetric, 
      \begin{equation}\label{eq:Talagrand-gen-intro}
     \mathcal{T}_{\omega_\rho}(\nu_1,\nu_2) \leq    H(\nu_1|  \mu_\rho) + H(\nu_2|  \mu_\rho),
    \end{equation}
      where  
      \[
      \mathcal{T}_{\omega_\rho}(\nu_1,\nu_2) = \inf_{X_1 \sim \nu_1,X_2 \sim \nu_2} \E\left[ \omega_\rho(X_1,X_2)\right]
    \] 
    is the optimal transport cost associated to the cost function $\omega_\rho$ defined by
    \[
        \omega_\rho(x,y)=\left\{\begin{array}{ll}  \log\pare{\frac{\rho(x\cdot y)^2}{\rho(\abs x^2)\rho(\abs y^2)}}
    &  \text{ if } x\cdot y \geq 0  \\ +\infty & \text{ otherwise}  \end{array}\right. ,\qquad x,y \in \R^n.
    \]
\end{intro_thm}
In the result above, and in all the paper, a probability measure $\mu$ on $\R^n$ will be called symmetric if it is invariant under the map $\R^n \to \R^n : x \mapsto -x$.

The proof of this result relies on a classical duality argument due to Bobkov and G\"otze \cite{BG99}. Since Inequality \eqref{eq:Talagrand-gen-intro} holds only for symmetric probability measures, it can be considered as some transport-entropy version of Ball's functional Blaschke-Santal\'o inequality for even functions. Linearizing Inequality \eqref{eq:Talagrand-gen-intro} around $\mu_\rho$ gives back a sharp Brascamp-Lieb type inequality recently used by Cordero-Erausquin and Rotem \cite{CER} in their study of the $(B)$ conjecture and the Gardner-Zvavitch conjecture for rotationally invariant probability measures. More precisely, we get the following in Theorem~\ref{th:sharp_poincare}: 
\begin{intro_thm}\label{th:sharp_poincare}
  Let $\rho:\R_+\to\R_+$ such that $t\mapsto v_\rho(t) = -\log\rho(e^t)$ is convex and increasing. Define the measure $\mu_\rho$ in the same way as in the previous theorem. Then, for all $f\in\mc C^\infty_c(\R^n)$ even and such that $\int f\,d\mu_\rho=0$,
  \begin{equation}\label{eq:weighted_poincare_intro}
  \int f^2\,d\mu_\rho \leq \frac{1}{2}\int H_\rho^{-1}\nabla f\cdot\nabla f\,d\mu_\rho,
  \end{equation}
  where $H_\rho$ is the positive matrix given by
  \[
  \frac12 H_\rho(y) = \frac{1}{\abs y^2}\sbra{\pare{I_n-\frac{y\otimes y}{\abs y^2}}v_\rho'(t) + \frac{y\otimes y}{\abs y^2}v_\rho''(t)} 
  \]
  where, for simplicity, we used the notation $t=2\log \abs y$.
\end{intro_thm}
Since \eqref{eq:weighted_poincare_intro} admits equality cases, this shows in particular that Inequality \eqref{eq:Talagrand-gen-intro} is sharp.

In comparison to Fathi's inequality \eqref{eq:Talagrand-intro}, it seems natural to ask if \eqref{eq:Talagrand-gen-intro} can be extended to more general couples of probability measures, as for instance couples of the form $(\nu_1,\nu_2)$ with $\nu_1$ arbitrary and $\nu_2$ centered with respect to $\mu_\rho$. A closely related question is whether, for a given weight function $\rho$, the functional Blaschke-Santal\'o inequality \eqref{eq:BSfunction-gen-intro} is true with $z=0$ whenever $f$ has its barycenter at $0$, as proved by Lehec  \cite{lehec3}  in the particular case of the weight function $\rho_0$ defined above. As we will now explain, the answer to these questions actually depends on the weight function $\rho$. Consider the class of weight functions $(\rho_s)_{s\in \R}$, defined for $s\neq 0$ by
\[
\rho_s(t)=(1-s t)_+^\frac{1}{2s},\qquad t \geq0.
\]
The associated probability measures are the following:
\begin{itemize}
\item For $s >0$, we will denote 
\begin{equation}\label{eq:gamma_s}
    d\gamma_{s}(x):=\mu_{\rho_s}(dx) = \frac{1}{Z_s} \left[1 - s |x|^2\right]_+^{1/(2 s)} \,dx,
\end{equation}
which is a particular case of the so-called \emph{Barenblatt profiles.} Note that $\gamma_s \to \gamma$ as $s \to 0$ (in the sense of pointwise convergence of densities for instance).
\item For $\beta>n/2$, we will denote
\[
d\mu_\beta(x) = \frac{1}{Z_\beta(1+\abs x^2)^\beta}\,dx,
\]
which is a \emph{Cauchy type distribution} and corresponds to (a dilation of)$\mu_{\rho_s}$ with $s=-1/(2\beta)$.
\end{itemize}
Let us first present our main contributions in the range $s>0$.
As we shall see in Theorem \ref{thm:Barenblatt}, the following is true.
\begin{intro_thm}
    Let $s>0$. Consider the probability $\gamma_s$ defined in~\eqref{eq:gamma_s}. Then, for any $\nu_1,\nu_2$ with compact support included in the open Euclidean ball $B_s$ centered at the origin and of radius $1/\sqrt{s}$, and with $\nu_2$ centered,
\begin{equation}\label{eq:Barenblatt-intro}
\mathcal{T}_{k_s} (\nu_1,\nu_2) \leq H(\nu_1 |  \gamma_s) + H(\nu_2 |  \gamma_s),
\end{equation}
where $k_s:B_s \times B_s \to \R$
\[
k_s(x,y) =  \frac{1}{s}\log \left( \frac{1- s x\cdot y}{(1- s |x|^2)^{1/2}(1- s|y|^2)^{1/2}}\right),\qquad x,y \in B_{s}.
\]
\end{intro_thm}
This result is analogous to Fathi's result \eqref{eq:Talagrand-intro} in the Gaussian case and gives back \eqref{eq:Talagrand-intro} by sending $s \to 0$.
One can show that \eqref{eq:Barenblatt-intro} (see Remark \ref{rem:centering} for explanations) also implies the following version of the functional Blaschke-Santal\'o inequality: for all continuous $f:\R^n\to\R_+$ and $g:\R^n \to \R_+$ with supports in $B_s$ and such that $\bary(f):= \frac{\int xf(x)\,dx}{\int f(y)\,dy}=0$ and  
\begin{equation}\label{eq:fgrho_s}
f(x)g(y) \leq \rho_s(\langle x,y \rangle)^2 ,\qquad \forall x,y\in B_s,
\end{equation}
it holds 
\[
\int f(x)\,dx \int g(y)\,dy \leq \left(\int \rho_s(|x|^2)\,dx\right)^2.
\]
This generalizes the Blaschke-Santal\'o inequality under a centering condition obtained by Lehec in \cite{lehec3}  for the weight $\rho_0$ (which corresponds to the limit case $s\to0$). As we will see with Theorem \ref{BS-fun-gen}, one can go actually a step further: 

\begin{intro_thm}
    If $f:\R^n\to\R_+$ is integrable and such that $0\in \inte(\conv(\supp(f)))$, then it holds
    \[
    \int f(x)\,dx\int \L_s f(y)\,dy\le \left(\int \rho_s(|x|^2)\,dx\right)^2    (1-s\langle \san_s(\L_s(f)),\bary(f)\rangle)^{n+1+\frac{1}{s}},
    \]
    where 
    \[
    \L_s f(y)=\inf_{x\in\R^n}\frac{\left(1-s\langle x,y\rangle\right)_+^\frac{1}{s}}{f(x)},\quad \hbox{for $s\neq 0$,}
    \]
    the infimum being taken on $\{x\in\R^n; f(x)>0\}$ and $\san_s(g)$ denotes the $s$-Santal\'o point of $g$ whose definition is given in Lemma \ref{formula-f-z}.
\end{intro_thm}
The proof of this theorem relies on the fact that the integral of $\L_s(f_z)$ with respect to Lebesgue, where $f_z(x)=f(z+x)$, $x\in \R^n$, can be expressed as the integral of $\L_s(f)$ under some weighted measure. The same type of arguments can actually be used at the level of the Blaschke-Santal\'o inequality for sets. 
In particular, we show the following in Theorem \ref{BS-set-gen}: 
\begin{intro_thm}
    If $K$ is a compact set such that $|K|>0$ and $0\in \inte(\conv(K))$, then 
\[
|K||K^\circ|\le|B_2^n|^2(1-\langle \san(K^\circ),\bary(K)\rangle)^{n+1},
\]
with equality if and only if $K$ is a centered ellipsoid, where $\san(K^\circ)$ is defined in Section \ref{sec:BScompact sets}.
In particular, if $\bary(K):=\frac{\int_K x\,dx}{|K|}=0$ then $|K||K^\circ|\le|B_2^n|^2$.
\end{intro_thm}
The above centered inequality seems to be new, even for convex bodies, while the case where $\bary(K)=0$ extends a result by Lutwak \cite{L91}, also reproved differently by Lehec \cite{lehec3}, who both obtained the same inequality but under the additional assumption that $K$ is starshaped.

Let us now turn to the range $s<0$. Applying Inequality \eqref{eq:Talagrand-gen-intro} with the weight function $t\mapsto(1+t)^{-\beta}$ and $\beta >n/2$, yields
  \begin{equation}\label{eq:Cauchy-intro}
 \beta \mathcal{T}_{\omega}(\nu_1,\nu_2) \leq    H(\nu_1|  \mu_\beta) + H(\nu_2|  \mu_\beta),
  \end{equation}
  where the optimal transport cost $ \mathcal{T}_{\omega}$ is defined with respect to the cost function $\omega$ given by
\[
    \omega(x,y)=\left\{\begin{array}{ll} -2\log\pare{\frac{1 +x\cdot y}{\sqrt{1+|x|^2}\sqrt{1+|y|^2}}}
&  \text{ if } x\cdot y >0  \\ +\infty & \text{ otherwise}  \end{array}\right. ,\qquad x,y \in \R^n.
\]
The fact that the cost function $\omega$ can take the value $+\infty$ makes inequality \eqref{eq:Cauchy-intro} for Cauchy type distributions more rigid than its counterpart \eqref{eq:Barenblatt-intro} for Barenblatt type distributions. Namely, it is not possible to extend \eqref{eq:Cauchy-intro} to couples of probability measures $(\nu_1,\nu_2)$ with $\nu_1$ arbitrary and $\nu_2$ symmetric. See Remark \ref{rem:non-Fathi} for more details. For the particular value $\beta = (n+1)/2$, it turns out that the canonical geometric framework for \eqref{eq:Cauchy-intro} is the unit sphere $\Sd^n \subset \R^{n+1}$ equipped with the uniform probability measure, denoted by $\sigma$. In Theorem~\ref{thm:Kolesnikovsym}, we establish the following.
\begin{intro_thm}
    Let  $\alpha : \Sd^n\times \Sd^n \to \R_+\cup\{+\infty\}$ be the cost function defined by 
\begin{equation}\label{eq:alpha-intro}
\alpha(u,v) = \left\{\begin{array}{ll}  \log \left(\frac{1}{u\cdot v}\right)  & \text{if } u\cdot v >0    \\ + \infty  & \text{otherwise}  \end{array}\right.,\qquad u,v \in \Sd^n
\end{equation}
and denote by $\mathcal{T}_\alpha$ the corresponding transport cost on $\mathcal{P}(\Sd^n)$. 
Then, for all probability measures $\nu_1,\nu_2$ on  $\Sd^n$ which are invariant under the maps $\Sd^n \to \Sd^n:u\mapsto -u$ and $\Sd^n\to \Sd^n :u \mapsto (u_1,\ldots,u_n,-u_{n+1})$, it holds 
\begin{equation}\label{eq:uniform1-intro}
(n+1)\mathcal{T}_\alpha(\nu_1,\nu_2) \leq H(\nu_1|\sigma)+H(\nu_2|\sigma).
\end{equation}
\end{intro_thm}
This result is deduced from \eqref{eq:Cauchy-intro} using the fact that the standard Cauchy distribution $\mu_{(n+1)/2}$ is the image of $\sigma_+$, the uniform probability measure on the upper half sphere $\Sd^n_+$, under the so-called \emph{gnomonic transformation}:
\[
\Sd^n_+ \to \R^n : u \mapsto \left(\frac{u_1}{u_{n+1}}, \frac{u_2}{u_{n+1}},\ldots, \frac{u_n}{u_{n+1}} \right).
\]
The cost function $\alpha$ defined above has been introduced by Oliker \cite{Olik07} (see also \cite{Ber16} and \cite{Kol20}) in connection with the so-called Aleksandrov problem in convex geometry. Recently, Kolesnikov \cite{Kol20} proved the following inequality involving the transport cost $\mathcal{T}_\alpha$: for any symmetric probability measure $\nu$ on $\Sd^n$ (that is, invariant under the map $\Sd^n \to \Sd^n:u\mapsto -u$), it holds
\begin{equation}\label{eq:Kolesnikov-intro}
(n+1)\mathcal{T}_\alpha(\nu,\sigma) \leq H(\nu|\sigma).
\end{equation}
Thus \eqref{eq:uniform1-intro} already improves \eqref{eq:Kolesnikov-intro} for a special class of distributions. One can actually improve \eqref{eq:Kolesnikov-intro} further. We show in Theorem \ref{thm:Kolesnikovsym}, by a direct proof using the Blaschke-Santal\'o inequality written in polar coordinates together with the dual Kantorovich type formula for $\mathcal{T}_\alpha$, that \eqref{eq:uniform1-intro} holds under the sole assumption that $\nu_1$ and $\nu_2$ are symmetric. We refer to the end of Section \ref{sec:Cauchy} for additional comments about the sharpness of this improvement of Kolesnikov inequality \eqref{eq:Kolesnikov-intro}.

The second main objective of this paper is to propose a transport-entropy framework for reverse Blaschke-Santal\'o inequalities. Let us recall that Mahler \cite{Mah39a} conjectured that for any centrally symmetric convex body $K$ the following lower bound holds: 
\begin{equation}\label{eq:Mahler-intro}
|K||K^\circ|\ge \frac{4^n}{n!},
\end{equation}
with equality for example if $K$ is a cube. Mahler established this inequality in dimension $2$ \cite{Mah39b}, while the conjecture for centrally symmetric bodies was established by Iriyeh and Shibata in dimension~$3$ \cite{is20} (see also \cite{fhmrz21}). The conjecture was proved for particular families of convex bodies like unconditional convex bodies \cite{sr81, mey86}, zonoids \cite{R86,GMR88}, bodies having symmetries \cite{bf13, is22}. Bourgain and Milman \cite{bm87} (see also \cite{kup08, naz12, blo14, gpv14, ber20a, ber20b}) established an asymptotic form of Mahler conjecture by proving that there exists a constant $\kappa>0$ such that 
\[
|K||K^\circ|\ge \frac{\kappa^n}{n!},
\] 
for any centrally symmetric convex body $K$.
Like the classical Blaschke-Santal\'o inequality, the Mahler conjecture admits an equivalent functional form introduced by Klartag-Milman \cite{km05} and the first named author and Meyer \cite{fm07, fm08a}: as shown in \cite{fm08a}, the inequality \eqref{eq:Mahler-intro} holds for all $n\geq 1$ if and only if the inequality
\begin{equation}\label{eq:Mahler-functions-intro}
\int e^{-f}\,dx \int e^{-f^*}\,dx \geq 4^n
\end{equation}
holds for all $n\geq1$ and all even, convex functions $f:\R^n \to \R\cup\{+\infty\}$ such that $\int e^{-f}\,dx >0$ and $\int e^{-f^*}\,dx>0$. Moreover, if \eqref{eq:Mahler-functions-intro} holds for a given $n$, then \eqref{eq:Mahler-intro} also holds for this $n$. For unconditional functions, \eqref{eq:Mahler-functions-intro} holds true for all $n\geq1.$ Denote by 
\[
c_n^{S} = \inf |K||K^\circ|\qquad \text{and}\qquad c_n^{F}=\inf \int e^{-f}\,dx \int e^{-f^*}\,dx,
\]
where $S$ stands for \emph{sets}, $F$ for \emph{functions} and the infima run respectively over all centrally symmetric convex bodies $K$ and all even, lower semicontinuous and convex functions $f:\R^n \to \R\cup\{+\infty\}$ such that $\int e^{-f}\,dx >0$ and $\int e^{-f^*}\,dx>0$. Then, as explained above, for any $n\geq 1$ it always holds
\[
c_n^S \geq \frac{c_n^F}{n!},
\] 
while the converse relation between $c_n^{F}$ and $(c_m^{S})_{m\geq n}$ is more intricate.

In \cite{Gozlan21}, the second named author obtained an equivalent formulation of the Mahler conjecture involving transport, entropy and Fisher information in the Gauss space $(\R^n,|\,\cdot\,|,\gamma)$.
More precisely, according to \cite[Theorem 1.3]{Gozlan21}, for any $n\geq 1$, the constant $c_n^F$ is the best constant $c>0$ (that is the greatest) in the inequality 
\begin{equation}\label{eq:LSI-moment-intro}
H(\eta_1 |\gamma)+ H(\eta_2 |\gamma) + \frac{1}{2}W_2^2(\nu_1,\nu_2) \leq \frac{1}{2} I(\eta_1| \gamma)+ \frac{1}{2} I(\eta_2| \gamma)+  \log \left(\frac{(2\pi)^n}{c}\right),
\end{equation}
where $\eta_1=e^{-V_1}\,dx,\eta_2=e^{-V_2}\,dx$ are arbitrary symmetric log-concave probability measures on $\R^n$ with full support and, for $i=1,2$, $\nu_i$ is the so-called moment measure of $\eta_i$ defined by 
\[
\nu_i = \nabla (V_i)_\# \eta_i
\]
and $I(\eta_i| \gamma)$ is the relative Fisher information of $\eta_i$ with respect to $\gamma$ defined by
\[
I(\eta_i| \gamma)= \int |\nabla V_i(x)-x|^2 e^{-V_i(x)}\,dx.
\]
Moreover, \eqref{eq:LSI-moment-intro} holds true with the constant $c= 4^n$ if $\eta_1,\eta_2$ are further assumed to be unconditional. The class of probability measures $\eta(dx)=e^{-V}\,dx$ for which \eqref{eq:LSI-moment-intro} holds can be slightly extended to those having an essentially continuous potential $V$, which means that the convex potential $V$ explodes at almost every points of the boundary of the support of $\eta$ (we refer to \cite{Gozlan21} or \cite{CEK15} for a precise definition).
When the $W_2$ distance between the moment measures of $\eta_1$ and $\eta_2$ is large enough, inequality \eqref{eq:LSI-moment-intro} thus improves the classical log-Sobolev inequality for the standard Gaussian measure $\gamma$ due to Gross \cite{Gro75}
\begin{equation}\label{eq:LSI-Gross-intro}
H(\eta |\gamma) \leq \frac{1}{2} I(\eta| \gamma),
\end{equation}
which holds for all probability measures $\eta$ with a sufficiently smooth density. In the unconditional case, this improvement is sharp in the sense that, one can easily construct sequences of probability measures $\eta_1^k,\eta_2^k$ for which the difference between the two sides of \eqref{eq:LSI-moment-intro} (with $c=4^n$) goes to $0$ as $k \to \infty$. Note however that each side goes individually to $+\infty$. There is, in particular, no equality case in \eqref{eq:LSI-moment-intro} (we refer to \cite[Remarks 3.9, 3.10 and 3.11]{Gozlan21} for this question).
The proof of \cite[Theorem 1.3]{Gozlan21} relies on the following two ingredients:
\begin{itemize}
\item The characterization of moment measures given by Cordero-Erausquin and Klartag \cite{CEK15}, according to which a probability measure $\nu$ is the moment measure of some log-concave probability measure $\eta_o$ with an essentially continuous potential if and only if $\nu$ is centered and its support is not contained in a hyperplane. The probability $\eta_o$ is then unique up to translations.
\item The following variational characterization highlighted by Santambrogio in \cite{San16} (see also \cite{FGJ17}): if $\nu$ is centered, has a finite moment of order $2$ and its support is not contained in a hyperplane, then the probability measure $\eta_o$ is up to translations the unique minimizer of the functional 
\[
\eta \mapsto  H(\eta | \gamma)- \frac{1}{2}W_2^2(\nu,\eta)
\]
over the set of probability measures having a finite moment of order $2$.
\end{itemize}

In the present paper, we provide a similar transport-entropy formulation of the (conjectured) reverse Blaschke-Santal\'o inequality where the space $\R^n$ is replaced by the sphere $\Sd^n$, the standard Gaussian measure $\gamma$ by the uniform probability measure $\sigma$ on $\Sd^n$, the $W_2$ distance by the transport cost $\mathcal{T}_\alpha$ associated to the cost function $\alpha$ defined in \eqref{eq:alpha-intro}, and where finally the notion of moment measure is replaced by the notion of \emph{cone measure}. If $C\subset \R^{n+1}$ is a centrally symmetric convex body of volume $1$, the cone measure of $C$ is the probability measure $\nu_C$ on $\Sd^n$ defined by
\[
\nu_C = \mathrm{Law} \left(N_C \left(\rho_C(X)X\right)\right),
\]
where $X$ is a random vector uniformly distributed over $C$, $\rho_C$ is the radial function of $C$ and $N_C: \partial C \to \Sd^n$ is the Gauss map. 
Equivalently, $\nu_C$ is also the pushforward of the probability measure $\eta_C$ on $\Sd^n$ defined by
\begin{equation}\label{eq:eta_C-intro}
d\eta_C(x) = |B_2^{n+1}| \rho_C^{n+1}(x)\,d\sigma(x)
\end{equation}
under the map $\Sd^n \to \Sd^n:x \mapsto N_C(x\rho_C(x))$, a construction which is reminiscent of the definition of moment measures.
The set of cone measures has been characterized by  B\"or\"oczky, Lutwak, Yang and Zhang in \cite{BLYZ13}. They proved that a symmetric probability measure $\nu$ on $\Sd^n$ is the cone measure of some centrally symmetric convex body $C$ if and only if it satisfies the so-called \emph{subspace concentration condition}, which is recalled in Section \ref{sec:cone}. 
To associate a set $C$ to a probability measure $\nu$ having good properties, the main step in the method proposed in \cite{BLYZ13} consists in solving a certain optimization problem over the set of support functions. As noticed by Kolesnikov \cite{Kol20}, this minimization problem can be recasted as follows: given a probability measure $\nu$ on $\Sd^n$, minimize the function
$F_\nu$ defined by
\[
F_\nu(\eta) = \frac{1}{n+1} H(\eta|\sigma) - \mathcal{T}_\alpha (\nu,\eta)
\]
over the set of symmetric probability measures on $\Sd^n$. More precisely, if $\nu$ satisfies the \emph{strict subspace concentration inequality} (which is stronger than the subspace concentration condition), then the functional $F_\nu$ admits at least one minimizer $\eta^*$ which is of the form $\eta^*=\eta_C$ above for some centrally symmetric convex body $C$ of volume $1$, and $\nu$ is the cone measure of $C$. 
A notable difference between cone and moment measures, is that there is in general no uniqueness of $C$. This uniqueness question is related to the log-Minkowski conjecture, a major open problem in convex geometry introduced in \cite{BLYZ12}, which can be restated as follows: if $C$ is a centrally symmetric convex body with unit volume, then $\eta_C$ minimizes $F_{\nu_C}$.

Assuming the log-Minkowski conjecture is true, we obtain in Theorem \ref{thm:LSIequivMahler} the following result:
\begin{intro_thm}
     If the log-Minkowski conjecture is true, then the constant $c_{n+1}^S$ is the best constant $c>0$ (that is the greatest) in the inequality
    \begin{align}\label{eq:Log-M-intro}
     &H(\eta_{C_1} | \sigma)+ H(\eta_{C_2} | \sigma) + (n+1)\mathcal{T}_\alpha(\nu_{C_1},\nu_{C_2})\\
    &\notag\quad\leq \log \left(\frac{|B_2^{n+1}|^2}{c}\right) + \frac{n+1}{2}\int \log\left(1+ \frac{|\nabla_{\Sd^n}V_1|^2}{(n+1)^2} \right)e^{-V_{1}}\,d\sigma 
    + \frac{n+1}{2}\int \log\left(1+ \frac{|\nabla_{\Sd^n}V_2|^2}{(n+1)^2} \right)e^{-V_{2}}\,d\sigma,
    \end{align}
    where $C_1,C_2 \subset \R^{n+1}$ are arbitrary centrally symmetric convex bodies with unit volume and, for $i=1,2$, $d\eta_{C_i} = |B_2^{n+1}| \rho_{C_i}^{n+1}\,d\sigma := e^{-V_i}\,d\sigma$ and $\nu_{C_i}$ is the cone measure of $C_i$.
\end{intro_thm}
Since a version of the log-Minkowski conjecture is true in the unconditional case \cite{Sar15}, we show in Theorem \ref{thm:LSIimproved-unconditional} that \eqref{eq:Log-M-intro} is true with the constant $c=4^{n+1}/(n+1)!$ when $C_1,C_2$ are assumed to be unconditional. In this unconditional setting,  contrary to what happens in the Gaussian case, Inequality \eqref{eq:Log-M-intro} admits equality cases which are given by couples of unconditional convex bodies $(C,C^\circ)$ minimizing the volume product and properly normalized to be of volume $1$. Without assuming the log-Minkowski conjecture, some weaker statement remains valid, see Theorem \ref{thm:LSIimproved}. 
Contrary to the Gaussian case, we do not know whether the underlying log-Sobolev inequality 
\[
H(e^{-V}\,dx | \sigma) \leq \frac{a}{2}\int \log\left(1+ \frac{|\nabla_{\Sd^n}V|^2}{a^2} \right)e^{-V}\,d\sigma 
\]
holds true for all regular enough potentials $V$, and some constant $a>0$.
We refer to Section \ref{sec:reverse-sphere} for additional remarks and open questions about these improved log-Sobolev inequalities on the sphere.

\section{Blaschke-Santal\'o's inequality for compact sets and $s$-concave functions and functional forms of Mahler's conjecture}\label{sec:BScompact sets}

In the first subsection, we extend to arbitrary compact sets the result of Lutwak \cite{L91} and Lehec \cite{lehec3} stating that the  Blaschke-Santal\'o inequality holds for starshaped set with barycenter at the origin. In the second subsection, we generalize this to the Blaschke-Santal\'o inequality for $s$-concave functions, for $s\ge0$. In fact, for sets as well as for functions, we prove an inequality valid also if the barycenter is not at the origin. In the third subsection, we establish functional forms of Mahler's conjecture for unconditional $s$-concave functions, $s>-1/n$. In the case $s<0$, the situation is more involved because the set of $s$-concave functions is not preserved under $\L_s$-duality.

\subsection{Blaschke-Santal\'o inequality for compact sets}

For any set $A$ in $\R^n$ we define its polar by $A^\circ=\{y\in\R^n; \langle x, y\rangle\le 1, \forall x\in A\}$. Then, one has $A^\circ=(\conv(A,0))^\circ$, thus the set $A^\circ$ is a closed convex set containing the origin and, from the bipolar theorem, one has  $(A^\circ)^\circ=\overline{\conv(A,0)}$.
The classical Blaschke-Santal\'o \cite{blaschke_book,san49} inequality asserts that, for any convex body $K$ in $\R^n$, one has 
\[
\min_{z\in\inte(K)}|K||(K-z)^\circ|\le |B_2^n|^2,
\]
with equality if and only if $K$ is an ellipsoid. For any convex body $K$, we define its support function $h_K(y)=\sup_{x\in K}\langle x,y\rangle$, for $y\in\R^n$. If moreover $K$ contains the origin, we define its radial function by $\rho_K(u)=\sup\{t; tu\in K\}$, for $u\in S^{n-1}$ and one has $\rho_{K^\circ}(u)=h_K(u)^{-1}$, for all $u\in S^{n-1}$. 
For any $z$ in the interior of a convex body $K$ and any $y\in\R^n$, one has 
\[
h_{K-z}(y)=\sup_{x\in K}\langle x-z,y\rangle=h_K(y)-\langle z,y\rangle.
\]
Integrating in polar coordinates, we get
\begin{equation}\label{formulasphere}
|(K-z)^\circ|=\int_{S^{n-1}}\int_0^{\rho_{(K-z)^\circ}(u)}r^{n-1}\,drd\sigma(u)=\frac{1}{n}\int_{S^{n-1}}\rho_{(K-z)^\circ}(u)^nd\sigma(u)=\frac{1}{n}\int_{S^{n-1}}\frac{d\sigma(u)}{(h_K(u)-\langle z,u\rangle)^n}.
\end{equation}
This formula shows that the map $z\mapsto |(K-z)^\circ|$ is strictly convex. Moreover, it is not difficult to see that $|(K-z)^\circ|$ tends to infinity when $z\to\partial K$. It follows that the minimum $\min_z|(K-z)^\circ|$ is reached at a unique point $\san(K)$ called the Santal\'o point of $K$, which is in the interior of $K$. It follows that Blaschke-Santal\'o theorem may be reformulated as follows: for any convex body $K$ such that $\san(K)=0$ one has $|K||K^\circ|\le |B_2^n|^2$,
with equality if and only if $K$ is a centered ellipsoid. We say that a measurable set $K$ with finite and positive volume is centered if its center of mass $\bary(K)$ defined by 
\[
\bary(K)=\int_K \frac{x\,dx}{|K|}
\]
is at the origin. Since $\san(K)$ is the unique critical point of the function $z\mapsto |(K-z)^\circ|$, we get that $z=\san(K)$ if and only if $\nabla|(K-z)^\circ|=0$. By differentiating \eqref{formulasphere} and integrating in spherical coordinates, we get
\[
\nabla|(K-z)^\circ|=\int_{S^{n-1}}\frac{ud\sigma(u)}{(h_K(u)-\langle z,u\rangle)^{n+1}}=(n+1)\int_{(K-z)^\circ}x\,dx=(n+1)|(K-z)^\circ|\bary((K-z)^\circ).
\]
It follows that the Santal\'o point $\san(K)$ of $K$ is also the unique point $z$ such that $\bary((K-z)^\circ)=0$. One deduces from this property that $\san((K-\bary(K))^\circ)=0$ and that $\san(K)=0$ if and only if $\bary(K^\circ)=0$. Thus, the following third reformulation of Blaschke-Santal\'o inequality holds: for any convex body $K$ such that $\bary(K)=0$, one has $|K||K^\circ|\le |B_2^n|^2$,
with equality if and only if $K$ is an ellipsoid.
Lutwak noticed this in \cite{L91} and extended it to the case of compact starshaped bodies. A compact set $A$ is called starshaped with respect to the origin if for any $a\in A$ the segment $\{ta; t\in[0,1]\}$ is contained in $A$. In his Theorem 3.15 in \cite{L91}, Lutwak proved that if $A$ is starshaped with respect to the origin and has barycenter at the origin then $|A||A^\circ|\le |B_2^n|^2$, with equality if and only if $A$ is a centered ellipsoid. This result was also reproved by Lehec \cite{lehec3} who deduced it from a version of this theorem for log-concave functions. In the following theorem, we extend Lutwak's theorem to any compact set with a different proof. 

\begin{thm}\label{BS-set-gen}
Let $K$ be a compact set such that $|K|>0$ and $0\in \inte(\conv(K))$. Then 
\begin{equation}\label{eq:bs-sets}
|K||K^\circ|\le|B_2^n|^2(1-\langle \san(K^\circ),\bary(K)\rangle)^{n+1},
\end{equation}
with equality if and only if $K$ is a centered ellipsoid.
In particular, if $\bary(K)=0$ then $|K||K^\circ|\le|B_2^n|^2$, with equality if and only if $K$ is a centered ellipsoid.
\end{thm}

\begin{rem}\label{rem:santalo-sets-new}
 Formula \eqref{eq:bs-sets} seems to be new even in the case of convex sets. 
\end{rem}
\begin{rem}\label{rem:santalo-sets-sign} If $K$ is convex, since $\bary(K)\in K$  and $\san(K^\circ)\in K^\circ$, one has $\langle \san(K^\circ),\bary(K)\rangle\le1$, but it follows from the proof that actually $\langle \san(K^\circ),\bary(K)\rangle\le0$, see remark \ref{rem:santalo-sign}.
\end{rem}
\begin{rem}\label{rem:santalo-sets-other} Another formulation of the Blaschke-Santal\'o inequality for compact sets follows directly from the case of convex sets but with a less natural polarity point: given a compact set $A$, choosing $z=\san(\conv(A))$ and applying the classical inequality to $\conv(A)$, we get $(A-z)^\circ=(\conv(A)-z)^\circ$ and we deduce that
\[
|A||(A-z)^\circ|\le |\conv(A)||(\conv(A)-z)^\circ|\le|B_2^n|^2.
\]
\end{rem}
Before proving this theorem we first give a lemma which is very classical in projective geometry.

\begin{lem}\label{homog}
For $z\neq0$, we denote the open halfspace $H_z=\{y;1+\langle y,z\rangle>0\}$ and the map $F_z:H_z\to \R^n$ is defined for any $y\in H_z$ by
\[
F_z(y)=\frac{y}{1+\langle y,z\rangle}.
\]
Then \\
(i) The map $F_z$ is a bijection from $H_z$ onto $H_{-z}$ whose reciprocal is $F_{-z}$ and the Jacobian determinant of $F_z$ is $J_z(y):=(1+\langle y,z\rangle)^{-(n+1)}$.\\
(ii) For any compact set $K$ in $\R^n$ such that $0,z\in \inte(\conv(K))$, we have $(K-z)^\circ=F_{-z}(K^\circ)$ and 
\begin{equation}\label{formula-polar}
|(K-z)^\circ|=\int_{K^\circ}\frac{\,dx}{(1-\langle z,x\rangle)^{n+1}}.
\end{equation}
\end{lem}
Notice that formula (\ref{formula-polar}) is classical and can be found for example in Meyer and Werner in \cite[Lemma 3]{MW98} who proved it by using (\ref{formulasphere}) and a change of variable. We give here another proof which we shall extend to the functional case in the next section.

\begin{proof}
(i) From the definition of $F_z$, it is immediate that $F_z(H_z)\subset H_{-z}$ and that $F_{-z}(F_z(y))=y$, for all $y\in H_z$. It follows that $F_z$ is a bijection from $H_z$ onto $H_{-z}$ whose reciprocal is $F_{-z}$. The computation of the Jacobian matrix of $F_z$ is direct and gives 
\[
Jac(F_z)(y)=\frac{1}{1+\langle y,z\rangle}\left(I_n-\frac{yz^T}{1+\langle y,z\rangle}\right).
\]
Using the following Sylvester's identity, $\det(I_p+AB)=\det(I_q+BA)$ for any matrix $A\in M_{p,q}$ and $B\in M_{q,p}$, we conclude that the Jacobian determinant of $F_z$ is 
\[
J_z(y)=\det(Jac(F_z)(y))=\frac{1}{(1+\langle y,z\rangle)^n}\left(1-\frac{\langle y,z\rangle}{1+\langle y,z\rangle}\right)=\frac{1}{(1+\langle y,z\rangle)^{n+1}}.
\]
(ii)
One has 
\[
(K-z)^\circ=\{y;\ \langle y,x-z\rangle\le1,\ \forall x\in K\}=\{y;\ \langle y,x\rangle\le1+\langle y,z\rangle,\ \forall x\in K\}.
\]
Since $0\in\inte(\conv(K))$, for any $y\in(K-z)^\circ$ one has $\langle y,-z\rangle<1$ hence $1+\langle y,z\rangle>0$, thus
\[
(K-z)^\circ=\left\{y;\ \langle \frac{y}{1+\langle y,z\rangle},x\rangle\le1,\ \forall x\in K\right\}=\{y;\ F_z(y)\in K^\circ\}=F_{-z}(K^\circ).
\]
The last equality follows from the fact that $K^\circ\subset H_{-z}$ which deduces from the hypothesis $z\in\inte(\conv(K))$. Formula (\ref{formula-polar}) follows by using a change of variable and the formula for the Jacobian from (i).
\end{proof}

Now we give the proof of Theorem \ref{BS-set-gen}.
\begin{proof}[Proof of Theorem~\ref{BS-set-gen}]
Let $K$ be a compact set such that $0<|K|<+\infty$ and $0\in \inte(\conv(K))$. Then $K^\circ$ is a convex body to which we apply the classical Blaschke-Santal\'o's inequality: for $z=\san(K^\circ)$ one has 
\[
|K^\circ||(K^\circ-z)^\circ|\le |B_2^n|^2,
\]
with equality if and only if $K^\circ$ is an ellipsoid.
Since $0\in \inte(K^\circ)$ and $z\in\inte(K^\circ)$ we may apply formula (\ref{formula-polar}) to $K^\circ$ and we get
\[
|(K^\circ-z)^\circ|=\int_{K^{\circ\circ}}\frac{\,dx}{(1-\langle z,x\rangle)^{n+1}}.
\]
Using that $K\subset K^{\circ\circ}$ and applying Jensen's inequality to the function $\varphi(x)=(1-\langle z,x\rangle)^{-(n+1)}$, which is convex on $K$, we deduce that 
\begin{equation}\label{eq:polar-center}
|(K^\circ-z)^\circ|\ge\int_{K}\frac{\,dx}{(1-\langle z,x\rangle)^{n+1}}\ge\frac{|K|}{(1-\langle \san(K^\circ),\bary(K)\rangle)^{n+1}}.
\end{equation}
This concludes the proof of the inequality. If there is equality in this inequality, then, from the equality case in Blaschke-Santal\'o's inequality, we deduce that $K^\circ$ is an ellipsoid. Moreover, from the equality case in Jensen's inequality, it follows that $\san(K^\circ)=0$, thus $\bary(K)=0$. Finally, one has $|K|=|K^{\circ\circ}|$ which implies that $|\conv(K)\setminus K|=0$. Since $K$ is compact, it follows that $K=K^{\circ\circ}$. We thus conclude that $K$ is a centered ellipsoid.  
\end{proof}

\begin{rem}\label{rem:santalo-sign}
 Proof of Remark \ref{rem:santalo-sets-sign}: if $K$ is convex then, using that, in formula \eqref{eq:polar-center}, one has $z=\san(K^\circ)$, it follows from the definition of the Santal\'o point that $|(K^\circ-z)^\circ|\le |K^{\circ\circ}|=|K|$. Thus, we conclude that $\langle \san(K^\circ),\bary(K)\rangle\le0$. Notice that, applied to $K^\circ$, it gives also $\langle \san(K),\bary(K^\circ)\rangle\le0$.
 \end{rem}

\subsection{Blaschke-Santal\'o inequality for the $s$-concave duality}
The following general form of the functional Blaschke-Santal\'o inequality was proved by Ball \cite{ball86} in the even case, by the first named author and Meyer \cite{fm07} in the log-concave case and by Lehec \cite{Lehec1} in the general case. 
\begin{thm}\label{thm:Lehec}
Let $f: \R^n \to \R_+$ be integrable. Then there exists $z \in \R^n$ such that whenever $g:\R^n \to \R_+$ is a measurable function satisfying $f(x+z)g(y) \leq \rho(\langle x,y \rangle)^2$ for all $x,y\in \R^n$ such that $\langle x,y \rangle >0$ for some weight function $\rho : \R_+ \to \R_+$ such that $\int \rho(|x|^2)\,dx<+\infty$, it holds 
\[
\int f(x)\,dx \int g(y)\,dy \leq \left(\int \rho(|x|^2)\,dx\right)^2.
\]
Moreover, the point $z$ can be selected in the interior of the convex hull of the support of the measure with density $f$. In the case where $f$ is even, then $z$ can be chosen to be $0$.
\end{thm}
The fact that $z$ can be chosen in the convex hull of the support of $\nu_f(dx) = f(x)\,dx$ follows from Lehec's construction of $z$ as the center of a Yao-Yao partition for $\nu_f$ (see \cite[Theorem 9]{Lehec1}) and from Proposition 5 of \cite{Lehec2} which implies that the center of any such partition must belong to the convex hull of the support of $\nu_f$. In the following, we shall denote $f_z=f(z+\cdot).$

For $s\in\R$ and $g:\R^n\to\R_+$ non identically zero, we define its $s$-concave dual function $\L_s g:\R^n\to\R_+$ in the following way: for every $y\in\R^n$ 
\[
\L_s g(y)=\inf_{x\in\R^n}\frac{\left(1-s\langle x,y\rangle\right)_+^\frac{1}{s}}{g(x)},\quad \hbox{for $s\neq 0$,}
\]
where the infimum is taken on $\{x\in\R^n; g(x)>0\}$. For $s=0$, we set 
\[
\L_0 g(y)=\inf_{x\in\R^n}\frac{e^{-\langle x,y\rangle}}{g(x)}.
\]
Notice that the $s$-dual (even of a non $s$-concave function) is $s$-concave and that the $0$-dual is very much related to the Legendre transform since for any function $\varphi:\R^n\to\R\cup\{+\infty\}$ one has $\L_0(e^{-\varphi})=e^{-\varphi^*}$, where $\varphi^*$ is the Legendre transform of $\varphi$ defined by $\varphi^*(y)=\sup_x(\langle x,y\rangle-\varphi(x))$. 

This class was previously studied by Artstein-Avidan and Milman \cite{AAM08} where they proved that $\L_s$ is essentially the only order reversing transformation on $s$-concave functions. They also show that this duality is the usual polarity transform on the epigraphs of the functions for $s=1$.

Applied to the function $\rho_s(t)=(1-s t)_+^\frac{1}{2s}$, for $s\neq0$ and $\rho_0(t)=e^{-t/2}$, Theorem \ref{thm:Lehec} implies that for any integrable function $f:\R^n\to\R_+$, there exists $z$ such that for any $s>-1/n$, 
\begin{equation}\label{FM-rho-gamma}
\int_{\R^n}f(x)\,dx\int_{\R^n}\L_s (f_z)(y)\,dy\le\left(\int_{\R^n}\rho_s(|x|^2)\,dx\right)^2:=c_s,
\end{equation}
where a direct explicit computation gives that $c_0=(2\pi)^n$ and
\[
c_s= \left(\frac{\pi}{s}\right)^n\left(\frac{\Gamma\left(1+\frac{1}{2s}\right)}{\Gamma\left(1+\frac{1}{2s}+\frac{n}{2}\right)}\right)^2\ \ \hbox{for $s>0$ and }\quad c_s= \left(\frac{\pi}{|s|}\right)^n\left(\frac{\Gamma\left(\frac{1}{2|s|}-\frac{n}{2}\right)}{\Gamma\left(\frac{1}{2|s|}\right)}\right)^2\ \ \hbox{for $-\frac{1}{n}<s<0$.}
\]
Inequality \eqref{FM-rho-gamma} was established earlier in the case where $\frac{1}{s}$ is an integer and $s=0$ by Artstein-Avidan, Klartag and Milman \cite{AAKM05}. For $s<0$, inequality \eqref{FM-rho-gamma} was proved by Rotem in \cite{Rot14}. In particular, for $s=0$, this gives back the Blaschke-Santal\'o inequality for the Legendre transform established in \cite{AAKM05} which states that for any function $\varphi: \R^n\to\R\cup\{+\infty\}$ there exists $z\in\R^n$ such that 
\[
\int e^{-\varphi}\int e^{-(\varphi_z)^*}\le (2\pi)^n.
\]
This theorem was reproved by Lehec \cite{lehec3} who also established that if the barycenter of $e^{-\varphi}$ defined $\bary(e^{-\varphi})=\int xe^{-\varphi(x)}\,dx/\int e^{-\varphi}$ is at the origin then one may choose $z=0$, that is 
\[
\int e^{-\varphi}\int e^{-\varphi^*}\le (2\pi)^n.
\]
 We extend this theorem to the $s$-duality for any $s\ge0$. First we define the barycenter of $f$ to be $\bary(f)=\int xf(x)\,dx/\int f$. 
As in the case of sets we first state a lemma. Recall that $F_z(y)=\frac{y}{1+\langle y,z\rangle}$.

\begin{lem}\label{formula-f-z}
Let $s\ge0$ and $f:\R^n\to\R_+$ be a measurable function such that $f(0)>0$. 
\begin{enumerate}
    \item Then for every $z, y\in\R^n$ one has 
$\L_s(f_z)(y)=(1+s\langle z,y\rangle)_+^\frac{1}{s}\L_s f(F_{s z}(y))$, for $s>0$ and $\L_0(f_z)(y)=e^{\langle z,y\rangle}\L_0f(y)$.
\item Moreover if $f(z)>0$, then $\{x; \L_s f(x)>0\}\subset H_{-s z}=\{x; 1-s\langle z,x\rangle>0\}$ and, for $s>0$,
\begin{equation}\label{formula-gamma-polar}
\int\L_s(f_z)=\int\frac{\L_s f(x)}{(1-s\langle z,x\rangle)^{n+1+\frac{1}{s}}}\,dx.
\end{equation}
\item If $f$ is bounded and $\L_s f$ is integrable then the function $S(z):=\int\L_s(f_z)$ is strictly convex and admits a unique minimum at a point $\san_s(f)$ that we call the $s$-Santal\'o point of $f$ and which is in the interior of $\conv(\supp(f))$.
\end{enumerate}
\end{lem}

\begin{proof}
(1) For $s=0$, the relation is clear. Let us assume that $s>0$. From the definition one has 
\[
\L_s(f_z)(y)=\inf_x\frac{(1-s\langle x,y\rangle)_+^\frac{1}{s}}{f(x+z)}=\inf_x\frac{(1+s\langle z,y\rangle-s\langle x,y\rangle)_+^\frac{1}{s}}{f(x)}.
\]
Since the infimum runs on the set $\{x; f(x)>0\}$ and since $f(0)>0$ one deduces that 
\[
\L_s(f_z)(y)\le \frac{(1+s\langle z,y\rangle)_+^\frac{1}{s}}{f(0)}.
\]
Hence $\L_s(f_z)(y)=0$ if $1+s\langle z,y\rangle\le0$. 
Moreover, for $y\in H_{s z}$, one has
\[
\L_s(f_z)(y)= (1+s\langle z,y\rangle)^\frac{1}{s}\L_s f\left(\frac{y}{1+s\langle z,y\rangle}\right)=(1+s\langle z,y\rangle)^\frac{1}{s}\L_s f(F_{s z}(y)).
\]
(2) In the same way, from the definition of $\L_s$, if $f(z)>0$ then for all $y$,  $\L_s(f)(y)\le\frac{(1-s\langle z,y\rangle)_+^\frac{1}{s}}{f(z)}$. Thus if $\L_s(f)(y)>0$ then $1-s\langle z,y\rangle>0$ which means that $y\in H_{-s z}$.
Thus, using the change of variable $y=F_{-s z}(x)$, for $y\in H_{s z}$ and the fact that $(1+s\langle z,y\rangle)(1-s\langle z,x\rangle)=1$, we get 
\[
\int\L_s(f_z)(y)\,dy=\int_{H_{s z}}(1+s\langle z,y\rangle)^\frac{1}{s}\L_s f(F_{s z}(y))\,dy=\int_{H_{-s z}}\frac{\L_s f(x)}{(1-s\langle z,x\rangle)^{n+1+\frac{1}{s}}}\,dx.
\]
(3) The convexity is a direct consequence of formula \eqref{formula-gamma-polar}. The boundedness of $f$ implies that $\L_s f(0)>0$ and so $0$ is in the interior of the support of $\L_s f$. The existence of a unique minimizer was recently proved by Ivanov and  Werner in \cite{iw22}. They assumed for their proof that $f$ is $s$-concave but using that $\L_s\L_s\L_s f_z=\L_s f_z$, we can actually assume that $f$ is $s$-concave.
Moreover, it is clear that $\supp(\L_s f_z)=(\supp(f_z))^\circ$ so if $z$ is not in the interior of $\conv(\supp(f))$ then $0$ is not in the interior of $\conv(\supp(f_z))$ and $\supp(\L_s f_z)=(\supp(f_z))^\circ$ is unbounded, which implies that $\int \L_s f_z=+\infty$.
\end{proof}

Using the preceding lemma, we can now prove the following theorem. 
\begin{thm}\label{BS-fun-gen}
Let $s\ge0$ and $f:\R^n\to\R_+$ be integrable such that $\int f>0$ and $0\in \inte(\conv(\supp(f)))$. Then 
\[
\int f\int \L_s f\le c_s (1-s\langle \san_s(\L_s(f)),\bary(f)\rangle)^{n+1+\frac{1}{s}} \ \  \hbox{for $s>0$ and }\ \ \int f\int \L_0 f\le (2\pi)^ne^{-\langle \san_0(\L_0(f)),\bary(f)\rangle}
\]
In particular, if $\bary(f)=0$ then $\int f\int \L_s f\le c_s$.
\end{thm}

\begin{proof}[Proof of Theorem~\ref{BS-fun-gen}]
The proof to this theorem is similar to that of Theorem~\ref{BS-set-gen}. Fix a function $f$ (without any concavity assumption), such that $0\in \inte(\conv(\supp(f)))$ and $0 < \int_{\R^n} f < +\infty$. Then, from \eqref{FM-rho-gamma} applied to $\L_s f$, one has, for $z=\san_s(\L_s f)$, 
\[
\int_{\R^n}\L_s f(x)\,dx\int_{\R^n}\L_s ((\L_s f)_z)(y)\,dy\le c_s.
\]
Since $\L_s f(z)>0$, applying (2) of Lemma \ref{formula-f-z}, we deduce that 
\[
\int_{\R^n}\L_s ((\L_s f)_z)(y)\,dy=\int\frac{\L_s\L_s f(x)}{(1-s\langle z,x\rangle)^{n+1+\frac{1}{s}}}\,dx.
\]
Using that $\L_s\L_s f(x)\ge f(x)$ and Jensen's inequality, we get
\[
\int_{\R^n}\L_s ((\L_s f)_z)(y)\,dy\ge\int\frac{f(x)}{(1-s\langle z,x\rangle)^{n+1+\frac{1}{s}}}\,dx\ge \frac{\int f(x)\,dx}{(1-s\langle \san_s(\L_s(f)),\bary(f)\rangle)^{n+1+\frac{1}{s}}} 
\]
which concludes the proof of the theorem.
\end{proof}
\subsection{Duality and Mahler conjecture for $s$-concave even functions, when $s>-1/n$.}

In this section, we consider the extension of Mahler's conjecture \cite{Mah39a, Mah39b} to $s$-concave even functions. For $s>0$, the conjecture holds for unconditional functions and  follows from theorems of Saint Raymond \cite{sr81} and Reisner \cite{R87}. In the case $s=0$, the inequality was proved in \cite{fm08a, fm08b} and the equality case in \cite{fgmr10}. For $s<0$, the situation is more involved because the set of $s$-concave functions is not preserved by $\L_s$-duality. For $s<0$, we first characterize the class of $s$-concave integrable functions which is globally stable under the $\L_s$ duality. Then, we prove Mahler's conjecture for the functions in this class which are unconditional, using the same theorems of Saint Raymond and Reisner. 
Recall that a function $g:\R^n\to\R$ is unconditional if
$g(x_1,\dots, x_n)=g(|x_1|,\dots,|x_n|)$, for any $(x_1,\dots,x_n)\in\R^n$. And a set $K$ is unconditional if ${\bf 1}_K$ is unconditional. Let us first recall the original Mahler's conjecture for centrally symmetric convex bodies.

\begin{conj}\label{mahler} Let $K$ be a centrally symmetric convex body in $\R^n$. Then 
$$
|K||K^\circ|\ge \frac{4^n}{n!},
$$ 
with equality of and only if $K$ is a Hanner polytope.
\end{conj}
Hanner polytopes are succession of $\ell_1$ or $\ell_\infty$ sums of segments and include in particular the cube $B_\infty^n=[-1,1]^n$ and its polar $B_1^n=\{x=(x_1,\dots,x_n)\in\R^n;\sum_{i=1}^n|x_i|\le1\}$. 
Saint Raymond \cite{sr81} established Mahler's conjecture for unconditional convex bodies. He even prove the following more general statement, whose equality case is due to Reisner \cite{R87}.

\begin{thm}[Saint Raymond \protect{\cite[Theorem~21]{sr81}} and Reisner \protect{\cite[Theorem 1 and Remark~2]{R87}}]
\label{thm:SR}
Let $K\subset\R^n$ be an unconditional convex body and let $m_1,\dots,m_n >0$. Then
$$
\int_K\prod_{i=1}^n m_i|x_i|^{m_i-1}\,dx\int_{K^\circ}\prod_{i=1}^n m_i|x_i|^{m_i-1}\,dx \ge \frac{4^n\prod_{i=1}^n\Gamma(m_i+1)}{\Gamma(m_1+\cdots+m_n+1)},
$$
with equality if and only if $K$ is a Hanner polytope.
\end{thm}
We prove the following version of Mahler conjecture for unconditional $s$-concave functions, $s\ge0$.

\begin{thm}\label{thm:s-concave-SR}
Let $s\ge0$ and $g:\R^n\to\R_+$ be an $s$-concave unconditional function.
Then
\[
P_s(g)=\int_{\R^n} g\int_{\R^n}\L_s g\ge\frac{4^n}{(1+s)\cdots(1+ns)},
\]
with equality if and only if there exists a partition $\{1,\dots,n\}=I_1\cup I_2$ and two Hanner polytopes $K_1\subset F_1$ and $K_2\subset F_2$, where $F_j=\Span\{e_i, i\in I_j\}$, for $j=1,2$  such that for any $x_1\in F_1$ and $x_2\in F_2$, for $x=x_1+x_2$, one has $g(x)=(1-\|x_1\|_{K_1})^\frac{1}{s}_+{\bf 1}_{K_2(x_2)}$, for $s>0$ and $g(x)=e^{-\|x_1\|_{K_1}}{\bf 1}_{K_2(x_2)}$, for $s=0$.
\end{thm}

First, notice that if $s=0$ then  $\L_0(e^{-\varphi})= e^{-\varphi^*}$, as was previously noted. Hence, the result reduces to the reverse-Blaschke-Santal\'o inequality for unconditional log-concave functions, due to \cite{fm08a, fm08b} and the equality case was proved in \cite{fgmr10}.

The proof for $s>0$ will follow the same methods as in Artstein-Avidan, Klartag and Milman \cite{AAKM05}.

\begin{proof}[Proof for $s>0$.]

Recall that the function $g$ is $s$-concave if and only if $g^s$ is concave on its support. Let $f=g^s$ be concave on its support and denote $m=1/s>0$. Then one has $g=f^m$ and so
\[
\L_s g(y)=\inf_{x\in\R^n}\frac{\left(1-s\langle x,y\rangle\right)_+^\frac{1}{s}}{g(x)}=\left(\inf_{x\in\R^n}\frac{\left(1-s\langle x,y\rangle\right)_+}{g^s(x)}\right)^\frac{1}{s}=\left(\L_1 \left(g^s\right)(s y)\right)^\frac{1}{s}=
\left(\L_1 f\left(\frac{y}{m}\right)\right)^m.
\]
Hence 
\[
P_s(g)=\int g\int\L_s g=m^n\int f^m\int (\L_1 f)^m.
\]
For any function $f$ we define the set $K(f)=\{(x,s)\in\R^n\times\R; |s|\le f(x)\}$. The set $K(f)$ is convex in $\R^{n+1}$ if and only if $f$ is concave and using Fubini one has for every $m>0$
\[
\int_{\R^n}f^m(x)\,dx=\int_{\R^n}\int_0^{f(x)}mt^{m-1}\,dtdx=\frac{m}{2}\int_{K(f)}|t|^{m-1}\,dtdx. 
\]
Moreover, 
\begin{align*}
K(f)^\circ&=\{(y,t)\in\R^n\times\R; \langle x,y\rangle+\langle s,t\rangle\le1,\ \forall(y,t)\in\R^n\times\R\ \hbox{such that}\ |s|\le f(x)\}\\
&=\left\{(y,t)\in\R^n\times\R; |t|\le\frac{(1-\langle x,y\rangle)_+}{f(x)}, \forall x\in\{f>0\}\right\}= K(\L_1 f).
\end{align*}
From this formula we deduce that if $f$ is concave on its support and if $0$ is in the support of $f$ then 
$K(\L_1\L_1 f)=K(\L_1 f)^\circ=(K(f)^\circ)^\circ=K(f)$ and it follows that $\L_1\L_1 f=f$. Moreover we get
\[
\int f^m\int (\L f)^m=\frac{m^2}{4}\int_{K(f)}|t|^{m-1}\,dtdx\int_{K(f)^\circ}|t|^{m-1}\,dtdx.
\]
Lastly, notice that if $f$ is unconditional then $K(f)$ is unconditional, thus, from Theorem \ref{thm:SR} of Saint Raymond and Reisner, we conclude these quantities are minimized among unconditionnal convex sets if and only if $K(f)$ is a Hanner polytope in $\R^{n+1}$. It is not difficult  to see that this happens if and only if there exists a partition $\{1,\dots,n\}=I_1\cup I_2$ and two Hanner polytopes $K_1\subset F_1$ and $K_2\subset F_2$, where $F_j=\Span\{e_i, i\in I_j\}$, for $j=1,2$  such that for any $x_1\in F_1$ and $x_2\in F_2$, for $x=x_1+x_2$, one has $f(x)=(1-\|x_1\|_{K_1})_+{\bf 1}_{K_2(x_2)}$.
\end{proof}

\begin{proof}[Case $-1/n<s< 0$]
For $s<0$, the function $g$ is $s$-concave if and only if $g^s$ is convex. Let $f=g^s$ and $m=-1/s$. Then one has $g=f^{-m}$, $m>n$ and 
\[
\L_s g(y)=\inf_{x\in\R^n}\frac{\left(1-s\langle x,y\rangle\right)_+^\frac{1}{s}}{g(x)}
=\left(\inf_{x\in\R^n}\frac{(1-s\langle x,y\rangle)_+^{-1}}{g(x)^{|s|}}\right)^\frac{1}{|s|}
=\left(\L_{-1}(f^{-1})\left(\frac{y}{m}\right)\right)^{m}.
\]
For simplification, we introduce the following notation: for any $f:\R^n\to(0,+\infty)$ convex 
such that $f(tx)\to+\infty$ when $t\to+\infty$ for any $x\neq0$, one denotes, for $y\in\R^n$,
\[
\M f(y)=(\L_{-1}(f^{-1})(y))^{-1}=\sup_x\frac{1+\langle x,y\rangle}{f(x)}.
\]
Using this notation and a change of variables, we get that 
\begin{equation}\label{eq:Ps-when-negative}
 \int_{\R^n} g\int_{\R^n}\L_s g=m^n\int_{\R^n} \frac{1}{f^m}\int_{\R^n} \frac{1}{(\M f)^{m}}.
\end{equation}
 Variants of this transform have been considered by Rotem \cite{Rot14} and a reflection of $\M$ was also considered in \cite{AASW23}.
Indeed, the latter showed that the image class of $\M$ is the set of all functions who's epigraph $K$ is a convex set for which $\lambda K \subseteq K$ for all $\lambda \ge 1$, see \cite[Section 3]{AASW23}. This class of sets is called \emph{pseudo-cones}, studied in depth by Xu, Li and Leng \cite{XLL23} and Schneider \cite{sch23}. They define the \emph{copolar} of a set by
\[K^* = \{y\in \R^n:  \forall x \in K \ \  \langle x,y\rangle \le -1\}.\]
One may check that for a function $f$,
\[\text{epi}(\M f) = \text{epi}(-f)^*.\]

The following proposition establishes a few basic properties of $\M f$.
\begin{prop}\label{prop:Mf} 
Let $f:\R^n\to(0,+\infty)\cup\{+\infty\}$ be convex 
such that $\lim_{t\to+\infty}f(tx)=+\infty$, for any $x\neq0$ and $f\not\equiv+\infty$. Define, for $y\in\R^n$,
\[
\M f(y)=\sup_x\frac{1+\langle x,y\rangle}{f(x)}.
\]
Then $\M f:\R^n\to (0,+\infty)$ is convex and lower semi-continuous. Moreover, for every $y\neq0$, the function $t\mapsto\frac{\M f(ty)}{t}$ is non-increasing on $(0,+\infty)$ and there exists $a(y)>0$ and $b(y)\ge0$ such that $\lim_{t\to+\infty}\M f(ty)-(a(y)t+b(y))=0$.
\end{prop}
\begin{proof}
Since $f\not\equiv+\infty$, there exists $x_0\in\R^n$ such that $f(x_0)<+\infty$. Hence $\M f(y)\ge \frac{1+\langle x_0,y\rangle}{f(x_0)}>0$.
Since $\M f$ is the supremum of affine functions, it is convex and lower semi-continuous. For any $y\in\R^n$ and $t>0$ one has 
\[
\frac{\M f(ty)}{t}=\sup_x\frac{\frac{1}{t}+\langle x,y\rangle}{f(x)},
\]
hence the function $t\mapsto\frac{\M f(ty)}{t}$ is non-increasing on $(0,+\infty)$. Moreover one has for $t\ge T>0$
\[
\sup_x\frac{\langle x,y\rangle}{f(x)}\le \frac{\M f(ty)}{t}=\sup_x\frac{\frac{1}{t}+\langle x,y\rangle}{f(x)}\le\frac{1}{T}+\sup_x\frac{\langle x,y\rangle}{f(x)}.
\]
Hence $\lim_{t\to+\infty}\frac{\M f(ty)}{t}=\sup_x\frac{\langle x,y\rangle}{f(x)}:=a(y)>0$. Thus for any $s>0$ the function $t\mapsto\frac{\M f(ty)-\M f(sy)}{t-s}$ is non-decreasing on $(s,+\infty)$ and converges to $a(y)$ when $t\to +\infty$. It follows that 
\[
\frac{\M f(ty)-\M f(sy)}{t-s}\le a(y)
\] 
which implies that $t\mapsto \M f(ty)- ta(y)$ is non-increasing. Since $\M f(ty)\ge ta(y)$ we conclude that there exists $b(y)\ge0$ such that $\M f(ty)-(ta(y)+b(y))\to0$ when $t\to+\infty$, which implies that the function $\M f$ has asymptotes in every directions.
\end{proof}
\begin{rem}\label{rk:var-rate}
The argument above shows that for a convex function $f:\R^n\to(0,+\infty)\cup\{+\infty\}$ and any $x\neq0$, the function $t\mapsto\frac{f(tx)}{t}$ is non-increasing on $(0,+\infty)$ if and only if there exists $a(x)>0$ and $b(x)\ge0$ such that $\lim_{t\to+\infty} f(tx)-(a(x)t+b(x))=0$.
\end{rem}

To any convex function $f:\R^n\to(0,+\infty)\cup\{+\infty\}$  such that $\lim_{t\to+\infty}f(tx)=+\infty$, for any $x\neq0$ we attach $\f:\R^{n+1}\to\R_+\cup\{+\infty\}$ defined for $(x,s)\in\R^n\times\R$ by 
\[
\f(x,s)=|s|f(x/|s|)\quad\hbox{for}\ s\neq0\quad \hbox{and}\quad \f(x,0)=\lim_{s\to0}|s|f(x/|s|).
\]
Notice that this limit always exists in $\R_+\cup\{+\infty\}$ because the convexity of $f$ implies that the function $t\mapsto f(tx)/t$ is quasi-convex on $(0,+\infty)$. Hence $t\mapsto f(tx)/t$ is either non-increasing and non-negative on $(0,+\infty)$, or it is first non-increasing and then non-decreasing. Notice that $\f(x,0)=\lim_{t\to+\infty}f(tx)/t\in(0,+\infty)\cup\{+\infty\}$ and this limit is finite if and only if $f$ has an asymptote in the direction $x$, in which case it is the slope of this asymptote. Notice also that $\f$ is positively homogeneous: $\f(\lambda x,\lambda s)=\lambda \f(x,s)$, for every $\lambda \ge0$. We denote the domain of the convex function $f$ by $\dom(f)=\{x\in\R^n; f(x)<+\infty\}$. Then $\dom(\f)=\{(x,s)\in\R^n\times\R; x\in |s|\dom(f)\}$. Moreover, we also define  
\[
C(f)=\{(x,s)\in\R^n\times\R; \f(x,s)\le1\}.
\]
The function $\f$ is called the perspective function of $f$ and $C(f)$ is called the perspective body of $f$. For further properties of this functional transform, see \cite{AAFM12}.

\begin{defi}\label{rk:var-rate}
We denote by $\F$  the set of convex lower semi-continuous functions $f:\R^n\to(0,+\infty)$ such that, for any $x\neq0$, one has $\lim_{t\to+\infty}f(tx)=+\infty$ and the function $t\mapsto\frac{f(tx)}{t}$ is non-increasing on $(0,+\infty)$. Proposition \ref{prop:Mf} establishes in particular that $\M(\F)\subset\F$. Theorem \ref{thm:C(f)} below implies that for any $f\in\F$ one has $f=\M\M f$ and thus $\M(\F)=\F$.
\end{defi}

The following theorem gathers the important observations regarding $C(f)$.
\begin{thm}\label{thm:C(f)}
Let $f:\R^n\to(0,+\infty)\cup\{+\infty\}$ be convex and lower semi-continuous such that $\lim_{t\to+\infty}f(tx)=+\infty$, for any $x\neq0$ and whose domain $\dom(f)$ has non empty interior. Let $C(f)=\{(x,s)\in\R^n\times\R; \f(x,s)\le1\}$. Let $C(f)_+=C(f)\cap\{(x,s)\in\R^n\times\R; s\ge0\}$ and $C(f)_{-}=C(f)\cap\{(x,s)\in\R^n\times\R; s\le0\}$. Then \\
(i) $C(f)_+$ is a convex body containing $0$ on its boundary, $C(f)_+=\overline{\{(x,s)\in\R^n\times(0,+\infty);sf(x/s)\le1\}}$ and $C(f)_{-}$ is its symmetric image with respect to the hyperplane $\{s=0\}$.\\
(ii) $C(f)$ is convex if and only if $f\in\F$, \em{i.e.} $t\mapsto f(tx)/t$ is non-increasing on $(0,+\infty)$, for every $x$.\\
(iii) if, moreover, $f\in\F$, \em{i.e.} if $t\mapsto f(tx)/t$ is non-increasing on $(0,+\infty)$, for every $x$,  then the function $\f$ is a gauge on $\R^{n+1}$ whose unit ball is the convex body $C(f)$ and $f$ is the restriction of this gauge to the affine hyperplane $\{s=1\}$: for any $x\in\R^n$, one has $f(x)=\|(x,1)\|_{C(f)}$.\\
(iv) $C(f)^\circ=C(\M f)$ and for any $f\in\F$, one has $\M\M f=f$.\\
(v) For any $m>0$ one has $\int_{\R^n} f^{-(m+n)}=\frac{m+n}{2}\int_{C(f)}|s|^{m-1}\,dsdx$.
\end{thm}

\begin{proof}
(i) Let us prove that $\f$ is convex on $\R^n\times\R_+$.
Let $(x_1,s_1), (x_2,s_2)\in\R^n\times(0,+\infty)$ and $\lambda\in[0,1]$. Then, 
\begin{align*}
\f((1-\lambda)x_1+\lambda x_2,(1-\lambda)s_1+\lambda s_2) &=((1-\lambda)s_1+\lambda s_2) f\left(\frac{(1-\lambda)x_1+\lambda x_2}{(1-\lambda)s_1+\lambda s_2}\right) \\
&= ((1-\lambda)s_1+\lambda s_2) f\left(\frac{(1-\lambda)s_1\frac{x_1}{s_1}+\lambda s_2\frac{x_2}{s_2}}{(1-\lambda)s_1+\lambda s_2}\right)\\
&\le (1-\lambda)s_1f\left(\frac{x_1}{s_1}\right)+\lambda s_2f\left(\frac{x_2}{s_2}\right)\\
&= (1-\lambda)\f(x_1,s_1)+\lambda \f(x_2,s_2).
\end{align*}
It follows that $\f$ is convex on $\R^n\times(0,+\infty)$. Since $\f$ is defined on $\R^n\times\{0\}$ by taking a limit, it follows that $\f$ is convex on $\R^n\times\R_+$. Thus $C(f)_+=\{(x,s)\in\R^n\times\R_+; \f(x,s)\le1\}$ is convex. Moreover, $\f$ is lower semi-continuous on $\R^n$, hence $C(f)_+$ is closed. Moreover, since $\f(x,s)\in(0,+\infty)\cup\{+\infty\}$ and $\f$ is positively homogeneous, one has $\lim_{\lambda\to+\infty}\f(\lambda x,\lambda s)=+\infty$, for every $x\in\R^n$ and $s\ge0$. Hence $C(f)_+$ is bounded. Moreover, since $\dom(f)$ has non empty interior and 
\[
\dom(\f)=\{(x,s)\in\R^n\times\R; x\in |s|\dom(f)\}\supset \conv(0,\dom(f)\times\{1\}),
\]
we deduce that $\dom(\f)$ has also non empty interior. From Baire's theorem, there exists $M>0$ such that $K:=\{(x,s)\in\R^n\times\R_+; \f(x,s)\le M\}$ has non empty interior. Thus, by homogeneity, $K/M\subset C(f)_+$, which implies that $C(f)_+$ has non-empty interior. Therefore it is a convex body. Since $\f(0,0)=0$, one has $(0,0)\in C(f)_+$, thus $(0,0)$ is in the boundary of $C(f)_+$.  
The fact that $C(f)_{-}$ is the symmetric image of $C(f)_+$ with respect to the hyperplane $\{s=0\}$ is clear.\\
(ii) If $t\mapsto f(tx)/t$ is non-increasing on $(0,+\infty)$, for every $x$, let $(x_1,s_1), (x_2,s_2)\in\R^n\times\R$ and $\lambda\in[0,1]$. Assume first that $s_1, s_2, (1-\lambda)s_1+\lambda s_2\in\R^*$. 
Then, using that $|(1-\lambda)s_1+\lambda s_2|\le(1-\lambda)|s_1|+\lambda |s_2|$ and the fact that $t\mapsto f(tx)/t$ is non-increasing on $(0,+\infty)$, one has 
\begin{align*}
\f((1-\lambda)x_1+\lambda x_2,(1-\lambda)s_1+\lambda s_2) &=|(1-\lambda)s_1+\lambda s_2| f\left(\frac{(1-\lambda)x_1+\lambda x_2}{|(1-\lambda)s_1+\lambda s_2|}\right) \\
&\le ((1-\lambda)|s_1|+\lambda |s_2|) f\left(\frac{(1-\lambda)|s_1|\frac{x_1}{|s_1|}+\lambda |s_2|\frac{x_2}{|s_2|}}{(1-\lambda)|s_1|+\lambda |s_2|}\right)\\
&\le (1-\lambda)|s_1|f\left(\frac{x_1}{|s_1|}\right)+\lambda|s_2|f\left(\frac{x_2}{|s_2|}\right)\\
&= (1-\lambda)\f(x_1,s_1)+\lambda \f(x_2,s_2).
\end{align*}
For $s_1=0$ or $s_2=0$ or $(1-\lambda)s_1+\lambda s_2=0$  the result follows by taking the limit. It follows that $C(f)$ is convex. \\
If there exists $x\neq0$ such that the function  $t\mapsto f(tx)/t$ is not non-increasing on $(0,+\infty)$, then, by convexity this function is first decreasing then increasing on $(0,+\infty)$. Thus, for $s>0$, the function $s\mapsto \f(x,s)=sf(x/s)$ is also first decreasing then increasing on $(0,+\infty)$. By convexity, $\f$ is continuous on $\{(x,s);\f(x,s)<+\infty\}$, thus there exists $s_0>0$ such that $\f(x,s_0)=\inf_{s>0}\f(x,s):=m_0>0$. It follows that $\f(x,0)>\f(x,s_0)$. Hence, by homogeneity and by symmetry of $\f$, we deduce that $(x/m_0,\pm s_0/m_0)\in C(f)$, but $(x/m_0,0)\notin C(f)$, which proves that $C(f)$ is not convex. \\
(iii) If $t\mapsto f(tx)/t$ is non-increasing on $(0,+\infty)$, for every $x\neq0$, then, from (i) and (ii), $C(f)$ is a convex body which contains the origin and is symmetric with respect to the hyperplane $\{s=0\}$. Moreover, by homogeneity of $\f$, its gauge $\|\cdot\|_{C(f)}$ satisfies, for every $(x,s)\in\R^n\times\R$, 
\[
\|(x,s)\|_{C(f)}=\inf\{\lambda>0; (x,s)\in\lambda C(f)\}=\inf\{\lambda>0; f(x,s)\le\lambda\}=f(x,s).
\]
Thus, the function $\f$ is a gauge on $\R^{n+1}$ whose unit ball is the convex body $C(f)$. Moreover, for $s=1$, we get, for any $x\in\R^n$, $\|(x,1)\|_{C(f)}=\f(x,1)=f(x)$.\\
(iv) One has
\begin{align*}
(C(f))^\circ &= \left\{(y,t)\in\R^n\times\R;\ \langle x,y\rangle+st\le1, \forall (x,s),\ |s|f\left(\frac{x}{|s|}\right)\le1\right\}\\
&= \left\{(y,t)\in\R^n\times\R;\ |s|\langle z,y\rangle+st\le1, \forall (z,s),\ |s|f(z)\le1\right\}\\
&= \left\{(y,t)\in\R^n\times\R;\ \langle z,y\rangle+|t|\le f(z), \forall z\right\}\\
&= \left\{(y,t)\in\R^n\times\R;\ \sup_z\frac{\langle z,y\rangle+|t|}{f(z)}\le1\right\}\\
&= \overline{\left\{(y,t)\in\R^n\times\R;\ |t|\sup_z\frac{\langle z,\frac{y}{|t|}\rangle+1}{f(z)}\le1\right\}}\\
&= \overline{\left\{(y,t)\in\R^n\times\R;\ |t|\M f\left(\frac{y}{|t|}\right)\le1\right\}}\\
&=C(\M f)
\end{align*}
Assume that  $f\in\F$, then, from (i) and (ii), $C(f)$ is a convex body containing the origin and one has $C(\M \M f)=C(\M f)^\circ=(C(f)^\circ)^\circ=C(f)$. From (iii), we deduce that $\M\M f(x)=\|(x,1)\|_{C(\M \M f)}=\|(x,1)\|_{C(f)}=f(x)$.\\
(v)
Using Fubini, one has 
\begin{align*}
\int_{C(f)}|s|^{m-1}\,dsdx &=2\int_0^{+\infty}s^{m-1}|\{x\in\R^n; (x,s)\in C(f)\}|\,ds\\
&=2\int_0^{+\infty}s^{m-1}|\{x\in\R^n; sf(x/s)\le1\}|\,ds\\
&=2\int_0^{+\infty}s^{m-1}|\{sz\in\R^n; sf(z)\le1\}|\,ds\\
&=2\int_0^{+\infty}s^{n+m-1}|\{z\in\R^n; sf(z)\le1\}|\,ds
\end{align*}
Using the change of variable $s=1/t$ and Fubini we get
\begin{align*}
\int_{C(f)}|s|^{m-1}\,dsdx &=2\int_0^{+\infty}t^{-n-m-1}|\{z\in\R^n; f(z)\le t\}|\,dt\\
&=\frac{2}{m+n}\int_{\R^n}f(z)^{-n-m}\,dz.
\end{align*}
\end{proof}

We can now state and prove that unconditional functions in $\F$ satisfy a kind of Mahler conjecture.

\begin{thm}\label{mahler-for-F}
Let $f\in\F$, \em{i.e.} let $f:\R^n\to(0,+\infty)$ such that for any $x$, the function $t\mapsto\frac{f(tx)}{t}$ is non-increasing on $(0,+\infty)$. Assume moreover that $f$ is unconditional. Then for any $m>n$, one has 
\[
\int_{\R^n} f^{-m} \int_{\R^n} (\M f)^{-m}\ge \frac{4^n}{(m-1)\cdots(m-n)},
\]
with equality if and only if there exists a Hanner polytope $K$ in $\R^{n+1}$ such that for every $x\in\R^n$, $f(x)=\|(x,1)\|_K$.
\end{thm}


Recalling that $m=-1/s>n$, denote by $r=m-n>0$. Then, according to Theorem \ref{thm:C(f)} (v), one has 
$$\int_{\R^n} f^{-m}=\int_{\R^n} f^{-(r+n)}=\frac{m}{2}\int_{C(f)} |s|^{r-1}\,dsdx,$$
and similarly,
$$\int_{\R^n} \M f^{-m} = \frac{m}{2}\int_{C(\M f)} |s|^{r-1}\,dsdx = \frac{m}{2}\int_{C(f)^\circ} |s|^{r-1}\,dsdx.$$

Hence,
\[\int f^{-m} \int (\M f)^{-m} = \left(\frac{m}{2}\right)^2 \int_{C(f)} |s|^{r-1}\,dsdx\int_{C(f)^\circ} |s|^{r-1}\,dsdx \]
and from Theorem \ref{thm:SR} due to Saint Raymond and Reisner, the right hand side is minimized among unconditionnal convex sets if and only if $C(f)$  is a Hanner polytope $K$, which means that $f(x)=\|(x,1)\|_K$, for every $x\in\R^n$. 
\end{proof}

We can now state the following direct consequences of Theorems \ref{thm:C(f)} and \ref{mahler-for-F} for $s$-concave functions, when $s<0$. We denote by $\mathcal{C}^s$ the set of $s$-concave functions $g:\R^n\to (0,+\infty)$, which are lower semi-continuous and such that, for any $x\neq0$, one has $\lim_{t\to+\infty}g(tx)=0$ and the function $t\mapsto t^{-\frac{1}{s}}g(tx)$ is non-decreasing. 

\begin{cor} Let $s<0$. Then $\L_s(\mathcal{C}^s)=\mathcal{C}^s$. Moreover, for every $g\in\mathcal{C}^s$, one has $\L_s\L_s(g)=g$ and if $g$ is unconditional then
\[
P_s(g)=\int_{\R^n} g\int_{\R^n}\L_s g\ge\frac{4^n}{(1+s)\cdots(1+ns)},
\]
with equality if and only if there exists a Hanner polytope $K$ in $\R^{n+1}$ such that for every $x\in\R^n$, $f(x)=\|(x,1)\|_K^\frac{1}{s}$.
\end{cor}

\color{black}
\section{Transport-entropy forms of Blaschke-Santal\'o inequality}\label{sec:TE}
Given a measurable cost function $c:\R^n\times \R^n\to \R\cup\{+\infty\}$, bounded from below, the optimal transport cost between two probability measures $\nu_1,\nu_2 \in \mathcal{P}(\R^n)$ is defined as follows
\[
\mathcal{T}_c(\nu_1,\nu_2) =\inf\left\{ \int c(x,y) \,d\pi(x,y):\ \pi\in \mathcal{P}(\R^n\times\R^n),\ \pi(\R^n\times \cdot)=\nu_1(\cdot),\ \pi(\cdot\times\R^n)=\nu_2(\cdot) \right\},
\]
where $\mathcal{P}(\R^n)$ (resp. $\mathcal{P}(\R^n\times \R^n)$) denotes the set of all Borel probability measures on $\R^n$ (resp. $\R^n\times \R^n$).

Relative entropy is another classical functional on $\mathcal{P}(\R^n)$ that we shall now recall. Whenever $m$ is some measure on $\R^n$ (not necessarily of mass $1$) and $d\nu = f dm \in \mathcal{P}(\R^n)$, the relative entropy of $\nu$ with respect to $m$ is defined by 
\[
H(\nu|m)=\int f\log f\,dm,
\]
as soon as the right-hand side makes sense (that is to say $f \log^+ f$ or $f \log^- f$ is $m$-integrable). In particular, when $m$ is a probability measure, $H(\nu|m)$ always makes sense in $\R_+ \cup\{+\infty\}$.

Comparing optimal transport costs to relative entropy is the purpose of the family of transport-entropy inequalities introduced by Marton \cite{Mar86,Mar96a,Mar96b} and Talagrand \cite{Tal96} in the nineties. We refer to the survey \cite{GL10} for a presentation of this class of inequalities and their applications in the concentration of measure phenomenon. One of the most classical example of such an inequality is the so-called Talagrand's transport inequality for the standard Gaussian measure. It reads as follows:
\[ 
W_2^2(\nu,\gamma) \le 2H(\nu|\gamma),\qquad \forall \nu \in \mathcal{P}(\R^n),
\]
where $\gamma$ is the standard Gaussian probability measure on $\R^n$, and $W_2^2(\nu,\gamma)$ is the squared Wasserstein distance, which is equal to $\mathcal{T}_c(\nu,\gamma)$ for $c(x,y)=|x-y|^2$, $x,y\in \R^n$. This inequality is optimal with equality obtained when $\nu$ is a translation of $\gamma.$
Using the triangle inequality for $W_2$, it is easily seen that the following variant involving two probability measures also holds
\[ W_2^2(\nu_1,\nu_2) \le 4H(\nu_1|\gamma)  +4H(\nu_2|\gamma),\qquad \forall \nu_1,\nu_2 \in \mathcal{P}(\R^n).\]
This inequality is still optimal with equality achieved when $\nu_1$ and $\nu_2$ are two standard Gaussian with opposite means.
Recently, a symmetrized version of this inequality was obtained by Fathi \cite{fat18}, namely 
\begin{equation}\label{eq:Fathi} W_2^2(\nu_1,\nu_2) \le 2H(\nu_1|\gamma)  +2H(\nu_2|\gamma),
\end{equation}
whenever $\nu_1$ is centered and $\nu_2$ is arbitrary. 
Fathi derived \eqref{eq:Fathi} from a functional version of Blaschke-Santal\'o's inequality.

The aim of this section is to further explore the relationships between transport-entropy inequalities and functional forms of Blaschke-Santal\'o inequality given in Theorem \ref{thm:Lehec}. We will in particular derive from the latter some optimal transport-entropy inequalities for spherically invariant probability models that go beyond the Gaussian case.

\subsection{General costs}
Utilizing Theorem \ref{thm:Lehec} gives us two different families of transport-entropy inequalities for a large class of spherically invariant probability measures. 
\begin{thm}\label{thm:Talagrand-noneven-mu}
  Let $\rho :\R_+\to(0,\infty)$ be a continuous non-increasing function such that $\int \rho(|x|^2)\,dx<+\infty$, and $t\mapsto -\log\rho(e^t)$ is convex on $\R.$  Let $\mu_\rho$ be the probability measure with density proportional to $\rho(\abs x^2)$.
\begin{itemize}
\item[$(i)$]For all $\nu_1,\nu_2\in \mc P(\R^n)$ we have
  \begin{equation}\label{eq:tal1}
\mc T_{\tilde{\omega}_\rho}(\nu_1,\nu_2) \le    H(\nu_1|  \mu_\rho) + H(\nu_2|  \mu_\rho),
  \end{equation}
  where the optimal transport cost $\mc T_{\tilde{\omega}_\rho}$ is defined with respect to the cost function $\tilde{\omega}_\rho$ given by
  \begin{equation*}
    \tilde{\omega}_\rho(x,y)=\log\pare{\frac{\rho(|x\cdot y|)^2}{\rho(\abs x^2)\rho(\abs y^2)}},\qquad x,y \in \R^n.
  \end{equation*}
  \item[$(ii)$]   For all $\nu_1,\nu_2\in \mc P(\R^n)$ with $\nu_1$ and $\nu_2$ symmetric, we have
  \begin{equation}\label{eq:tal2}
 \mathcal{T}_{\omega_\rho}(\nu_1,\nu_2) \leq    H(\nu_1|  \mu_\rho) + H(\nu_2|  \mu_\rho),
  \end{equation}
  where the optimal transport cost $ \mathcal{T}_{\omega_\rho}$ is defined with respect to the cost function $\omega_\rho$ given by
  \begin{equation*}
    \omega_\rho(x,y)=\left\{\begin{array}{ll}  \log\pare{\frac{\rho(x\cdot y)^2}{\rho(\abs x^2)\rho(\abs y^2)}}
&  \text{ if } x\cdot y \geq 0  \\ +\infty & \text{ otherwise}  \end{array}\right. ,\qquad x,y \in \R^n.\end{equation*}

  \end{itemize}
  Furthermore, there is equality in inequalities~\eqref{eq:tal1} and \eqref{eq:tal2} when $\nu_1=\nu_2=\mu_\rho$.
\end{thm}

Before turning to the proof of Theorem \ref{thm:Talagrand-noneven-mu}, let us do some comments. If \eqref{eq:tal1} holds for all couples $\nu_1,\nu_2$ without restriction, note that the cost $\tilde{\omega}_\rho$ is not very standard. For instance, if $\rho_0(t)=e^{-t/2}$ for which $\mu_\rho = \gamma$ is the standard Gaussian, one gets $\tilde{\omega}_{\rho_0}(x,y) = \frac{1}{2}(|x|-|y|)^2$, $x,y\in \R^n$ instead of the usual quadratic cost $\frac{1}{2}|x-y|^2$. The cost $\omega_\rho$ seems better adapted to the geometry of the measure $\mu_\rho$, but the corresponding transport-entropy inequality \eqref{eq:tal2} requires symmetry assumptions on $\nu_1,\nu_2$. Taking Fathi's result \eqref{eq:Fathi} in consideration, a natural question is to ask whether these symmetry assumptions can be relaxed or not. We will see in the next two sections that the answer to this question depends on the cost function $\rho$. 

\begin{proof}
In this proof we adapt the classical dualization argument by Bobkov and G\"otze \cite{BG99} to our context. Let us first prove $(i)$. Rewriting Theorem~\ref{thm:Lehec} (even case) with respect to the functions
  \[
  F(x) = \log f(x) -\log \rho(|x|^2), \quad G(y) = \log g(y) -\log \rho(|y|^2),
  \]
  we get the following: for all bounded measurable functions $F,G$ such that $F$ is even and 
  \begin{equation}\label{eq:FG}
  F \oplus G \leq \tilde{\omega}_\rho
  \end{equation}
  it holds
  \begin{equation}\label{eq:ball2}
    \int_{\R^n}e^F\,d\mu_\rho \int_{\R^n}e^G\,d\mu_\rho \leq 1,
  \end{equation}
  where $F\oplus G(x,y) = F(x)+G(y)$, $x,y\in \R^n.$
 We now introduce two probability measures $\nu_1,\nu_2$. Then, taking the logarithm of inequality~\eqref{eq:ball2}, we find that
  \begin{equation}\label{eq:ball_dual}
   H(\nu_1|m)+H(\nu_2| m) \geq  \int_{\R^n} F\,d\nu_1 - \log  \int_{\R^n}e^F\,d\mu_\rho + \int_{\R^n} G\,d\nu_2 - \log  \int_{\R^n}e^G\,d\mu_\rho \geq \int_{\R^n} F\,d\nu_1 + \int_{\R^n} G\,d\nu_2,
  \end{equation}
where the first inequality comes from the duality formula for the relative entropy functional: if $\nu\in \mc P(\R^n)$ and $\log d\nu/dm\in L^1(\nu)$, then
  \[
  H(\nu|  m) = \sup_{f\in L^1(\nu)}\left\{ \int_{\R^n}f\,d\nu - \log\int_{\R^n}e^f\,dm\right\}.
  \]
Optimizing in \eqref{eq:ball_dual} with respect to $F$ and $G$, we thus find that
\[
    H(\nu_1|  \mu_\rho) + H(\nu_2|  \mu_\rho) \geq  \sup_{(F,G)\in S}\left\{ \int_{\R^n} F\,d\nu_1 + \int_{\R^n} G\,d\nu_2\right\}
\]
  where the supremum runs over the set $S$ of couples of bounded measurable functions $(F,G)$ with $F$ even and satisfying \eqref{eq:FG}.

Now, if $(F,G)$ is a couple of bounded measurable functions satisfying \eqref{eq:FG} (with $F$ not necessarily even), then by symmetry of $\tilde{\omega}_\rho$, the even function $\tilde F(x)=\max\{F(x),F(-x)\}$, $x\in \R^n$, is such that $(\tilde F,G)\in S$, and $\int_{\R^n} \tilde F\,d\nu_1\ge\int_{\R^n} F\,d\nu_1$, and so we may remove the assumption on evenness of $F$ and conclude that
  \begin{align*}
    \sup_{(F,G)\in S}\left\{ \int_{\R^n} F\,d\nu_1 + \int_{\R^n} G\,d\nu_2\right\} &= \sup_{(F,G) :  F\oplus G\leq\tilde{\omega}_\rho}\left\{ \int_{\R^n} F\,d\nu_1 + \int_{\R^n} G\,d\nu_2\right\}\\
    &= \mc T_{\tilde{\omega}_\rho}(\nu_1,\nu_2),
  \end{align*}
 where the second equality comes from the Kantorovich duality theorem (see e.g. \cite[Theorem 5.10]{villani_book}) which applies since the cost function $\tilde{\omega}_\rho$ is  lower semicontinuous (and even continuous) and bounded from below thanks to the log-concavity of $t\mapsto \rho(e^t)$ (it is, in fact, non-negative, a proof of which can be found in Lemma~\ref{lem:cost_prop}). This completes the proof of $(i)$.

Let us now prove $(ii)$. Reasoning exactly as before, one concludes that for any $\nu_1,\nu_2 \in \mathcal{P}(\R^n)$, it holds 
\[
    H(\nu_1|  \mu_\rho) + H(\nu_2|  \mu_\rho) \geq  \sup_{(F,G)\in \bar{S}}\left\{ \int_{\R^n} F\,d\nu_1 + \int_{\R^n} G\,d\nu_2\right\},
\]
where $\bar{S}$ is the set of couples of bounded measurable functions $(F,G)$ with $F$ even such that $F\oplus G \leq \omega_\rho$.
Let $(F,G)$ be a couple of bounded measurable functions (with $F$ non necessary even) such that $F \oplus G \leq \omega_\rho$. Since, for all $x,y \in \R^n$, $\omega_\rho(x,y) = \omega_\rho(-x,-y)$, defining $\bar{F}(x)=\frac{1}{2}(F(x)+F(-x))$ and $\bar{G}(y) = \frac{1}{2}(G(y)+G(-y))$, one gets that $(\bar{F},\bar{G}) \in \bar{S}$. If $\nu_1$ and $\nu_2$ are further assumed to be symmetric, it holds $\int \bar{F}\,d\nu_1 = \int F\,d\nu_1$ and $\int \bar{G}\,d\nu_2 = \int G\,d\nu_2$. Thus, in this case, 
\[
 \sup_{(F,G)\in \bar{S}}\left\{ \int_{\R^n} F\,d\nu_1 + \int_{\R^n} G\,d\nu_2\right\} =  \sup_{(F,G) : F\oplus G \leq \omega_\rho}\left\{ \int_{\R^n} F\,d\nu_1 + \int_{\R^n} G\,d\nu_2\right\} = \mathcal{T}_{\omega_\rho} (\nu_1,\nu_2),
\]
applying Kantorovich duality for the last equation, which completes the proof of $(ii)$.

Finally, note that $\tilde{\omega}_\rho$ and $\omega_\rho$ are both non-negative and vanish on the diagonal, so that $\mathcal{T}_{\tilde{\omega}_\rho}(\mu_\rho,\mu_\rho) = \mathcal{T}_{\omega_\rho}(\mu_\rho,\mu_\rho)=0$. There is thus equality in Inequalities~\eqref{eq:tal1} and \eqref{eq:tal2} when $\nu_1=\nu_2=\mu_\rho$.
\end{proof}

In the next subsections, we will study the consequences of Theorem \ref{thm:Talagrand-noneven-mu} for two special costs, related respectively to Barenblatt-type and Cauchy-type distributions. 

\subsection{Barenblatt-type distributions}

Let $s >0$ and denote by $B_{s} = \{ x \in \R^n : |x| < \frac{1}{\sqrt{s}}\}$ the open Euclidean ball of center $0$ and radius $\frac{1}{\sqrt{s}}$.
Consider the probability measure 
\[
\gamma_{s}(dx) = \frac{1}{Z_s} \left(1 - s |x|^2\right)^{1/(2 s)} \mathbf{1}_{B_s}(x) \,dx
\]
which is a particular case of the so-called Barenblatt profiles.
Consider the cost function $k_s : B_{s} \times B_{s} \to \R$ defined by
\[
k_s(x,y) =  \frac{1}{s}\log \left( \frac{1- s x\cdot y}{(1- s |x|^2)^{1/2}(1- s|y|^2)^{1/2}}\right),\qquad x,y \in B_{s}.
\]
For this particular cost, the conclusion of Theorem \ref{thm:Talagrand-noneven-mu} can be improved, as shown in the following result.
\begin{thm}\label{thm:Barenblatt}For all $s>0$, the probability measure $\gamma_s$ satisfies the following transport-entropy inequality: 
\[
\mathcal{T}_{k_s} (\nu_1,\nu_2) \leq H(\nu_1 |  \gamma_s) + H(\nu_2 |  \gamma_s),
\]
for all probability measures $\nu_1,\nu_2$, one of which is centered and with supports $K_1,K_2 \subset B_{s}$.
\end{thm}
This result is exactly analogous to Fathi's result  \eqref{eq:Fathi} in the Gaussian case. Moreover, note that as $s \to 0$, it holds $\gamma_s \to \gamma$ (the standard Gaussian) and one recovers \eqref{eq:Fathi}.

\proof[Proof of Theorem \ref{thm:Barenblatt}]
Applying Theorem \ref{thm:Lehec} to $\rho_s(t) = [1-s t]_+^{1/(2s)}$, $t\geq0$, yields the following: for any $s>0$ and $f : \R^n \to \R_+$ integrable, it holds 
\[
\int f(x) \,dx \inf_{z\in \mathrm{conv}\,S_f} \int \mathcal{L}_s(f_z)(y)\,dy \leq \left(\int_{B_{s}} \left(1 - s |x|^2\right)^{1/(2 s)}\,dx\right)^2 = Z_s^2,
\]
where $S_f$ denotes the support of the measure $\nu_f(dx) = f(x)\,dx$ and 
\[
 \mathcal{L}_s(g)(y) = \inf_{x : g(x)>0} \frac{[1-s x\cdot y]_+^{1/s}}{g(x)},\qquad y\in \R^n.
\]
Let $b_s(x,y) = \frac{1}{s} \log [1-s x\cdot y]_+$, $x,y \in \R^n$.
It is enough to prove that
\begin{equation}\label{eq:Tbs}
\mathcal{T}_{b_s} (\nu_1,\nu_2) \leq H(\nu_1 |  \mathrm{Leb}) + H(\nu_2 |  \mathrm{Leb}) + 2\log Z_s,
\end{equation}
for all probability measures $\nu_1,\nu_2$ with supports  $K_1,K_2 \subset B_{s}$ and such that $\nu_1$ is centered.
Note that $b_s$ is bounded and continuous on $K_1 \times K_2$. Therefore, applying Kantorovich duality theorem on $K_1\times K_2$ yields the following identity
\begin{equation}\label{eq:dualityTbs}
\mathcal{T}_{b_s}(\nu_1,\nu_2) = \sup_{\varphi \in \mathcal{C}_b(K_2)}\left\{ \int_{K_1} Q_s \varphi(x_1) \,d\nu_1(x_1) - \int_{K_2} \varphi(x_2)\,d\nu_2(x_2)\right\},
\end{equation}
where $\mathcal{C}_b(K_2)$ denotes the set of bounded continuous functions on $K_2$ and 
\[
Q_s \varphi(x_1) = \inf_{x_2 \in K_2} \{\varphi(x_2) + b_s(x_1,x_2)\},\qquad x_1 \in \R^n.
\]
Take $\varphi \in \mathcal{C}_b(K_2)$ and define $f : \R^n \to \R_+$ by $f(x_2) = e^{-\varphi(x_2)}$ if $x_2 \in K_2$ and $0$ otherwise.
Note the following relation :
\begin{equation}\label{eq:Qtbs}
e^{Q_s \varphi} = \mathcal{L}_s(f).
\end{equation}
According to what precedes, it holds 
\[
\int f(x_2)\,dx_2 \inf_{z \in \mathrm{conv} K_2} \int \mathcal{L}_s(f_z)(x_1)\,dx_1 \leq Z_s^2.
\]
Indeed, by construction the support of the measure $f(x)\,dx$ is $K_2$.
Note that the following inequality holds, for any $z\in \R^n$,
\[
\mathcal{L}_s(f_z)(y) \geq (1+s z\cdot y)_+ \mathcal{L}_s f(F_{s z}(y)),\qquad \forall y\in \R^n,
\]
where, for any $a \in \R^n \setminus \{0\}$, the map $F_{a}(y) = \frac{y}{1+z\cdot a}$, $y \in H_{a} = \{y \in \R^n : 1+z\cdot a>0\}$ is a bijection from $H_{a}$ onto $H_{-a}$ (this is Item (1) of Lemma \ref{homog}; when $f(0)=0$ there is equality but this not needed here).
So it holds 
\begin{align*}
\int \mathcal{L}_s(f_z)(x_1)\,dx_1 &\geq \int(1+s z\cdot x_1)_+^{1/s} \mathcal{L}_s f(F_{s z}(x_1))\,dx_1 \\
& = \int_{H_{s z}}(1+s z\cdot x_1)^{1/s} \mathcal{L}_s f(F_{s z}(x_1))\,dx_1 \\
& = \int_{H_{-s z}}\frac{1}{(1-s z\cdot u)^{n+1+\frac{1}{s}}} \mathcal{L}_s f(u)\,du \\
& = \int e^{Q_s \varphi(u)}\,dm_z(u),
\end{align*}
where $dm_z(u) = \frac{1}{(1-s z\cdot u)^{n+1+\frac{1}{s}}} \mathbf{1}_{H_{-s z}}(u)\,du$.
Therefore,
\[
- 2 \log Z_s \leq - \log \int_{K_2} e^{-\varphi (x_2)}\,dx_2 - \inf_{z\in \mathrm{conv}\,K_2} \log \int e^{Q_s \varphi (x_1)}\,dm_z(x_1)
\]
and so 
\begin{align*}
& - 2 \log Z_s+\int Q_s \varphi\,d\nu_1 - \int \varphi\,d\nu_2\\
&\leq  \int -\varphi\,d\nu_2- \log \int_{K_2} e^{-\varphi (x_2)}\,dx_2+ \int Q_s \varphi\,d\nu_1- \inf_{z\in \mathrm{conv}\,K_2} \log \int e^{Q_s \varphi (x_1)}\,dm_z(x_1)\\
 & \leq H(\nu_2 |  \mathrm{Leb}) + \sup_{z\in \mathrm{conv}\,K_2} H(\nu_1 |  m_z),
\end{align*}
where the last inequality follows from the bound
\[
 \int \psi\,d\nu- \log \int e^{-\psi}\,dm \leq H(\nu |  m),\qquad \forall \nu \ll m.
\]
Note that if $z \in B_{s}$, then $B_{s} \subset H_{- s z}$ and so in particular $\nu_1 \ll m_z$.

Finally, for all $z\in B_{s}$, it holds
\begin{align*}
H(\nu_1 |  m_z) & = \int_{B_{s}} \log \frac{d\nu_1}{dm_z}\,d\nu_1\\
& = H(\nu_1 |  \mathrm{Leb})  - \int_{B_{s}} \log \frac{dm_z}{dx}\,d\nu_1\\ 
& = H(\nu_1 |  \mathrm{Leb})  + (n+1+\frac{1}{s}) \int_{B_{s}} \log \left(1 - s z\cdot x_1\right)\,d\nu_1(x_1)\\ 
& \leq H(\nu_1 |  \mathrm{Leb})  + (n+1+\frac{1}{s})  \log \left(1 - s z\cdot \int x_1\,d\nu_1(x_1)\right)\\ 
& = H(\nu_1 |  \mathrm{Leb}),
\end{align*}
using the concavity of the logarithm and the fact that $\nu_1$ is centered. This completes the proof.
\endproof

\begin{rem}\label{rem:centering} Suppose that $f: \R^n \to \R^+$ is a continuous function such that $\int xf(x)\,dx=0$ and $f=0$ outside $B_s$. Denote by $K_2 = \{x \in B_s : f(x) \neq 0\}$ and $\varphi=-\log f \in \mathcal{C}_b(K_2)$. Then, using \eqref{eq:Tbs} and \eqref{eq:dualityTbs}, one gets
\[
\int Q_s\varphi\,d\nu_1 -H(\nu_1|\mathrm{Leb})+ \int -\varphi\,d\nu_2  -H(\nu_2|\mathrm{Leb}) \leq 2\log Z_s,
\]
for all $\nu_1,\nu_2$ with compact support in $B_s$ and $\nu_2$ centered.
Taking 
\[
d\nu_1(x) = \frac{e^{Q_s\varphi}(x)}{\int e^{Q_s\varphi(y)}\,dy}\,dx = \frac{\L_s f(x)}{\int \L_sf(y)\,dy}\,dx \qquad\text{and}\qquad d\nu_2(x) = \frac{e^{-\varphi}(x)}{\int e^{-\varphi(y)}\,dy}\,dx = \frac{f(x)}{\int f(y)\,dy}\,dx
\]
(thanks to \eqref{eq:Qtbs}) and noting that $\nu_2$ is centered, one gets
\[
\int f \int \L_sf \leq \left(\int \rho_s(|x|^2)\,dx\right)^2,
\]
which essentially gives back the conclusion of Theorem \ref{BS-fun-gen} in the centered case.
\end{rem}

\subsection{Cauchy-type distributions}\label{sec:Cauchy}In this section, we consider the cost function
\[
\rho_\beta(t)= \frac{1}{(1+t)^{\beta}},\qquad t\geq0
\]
for which $x\mapsto\rho_\beta(\abs x^2)$ is integrable whenever $\beta>n/2$. For $\beta>n/2$, we consider the following Cauchy type distribution
\[
d\mu_\beta(x) = \frac{1}{Z_\beta(1+\abs x^2)^\beta}dx,\quad\text{with } Z_\beta=\pi^{n/2}\frac{\Gamma(\beta-n/2)}{\Gamma(\beta)}.
\]
The following result follows immediately from Item $(ii)$ of Theorem \ref{thm:Talagrand-noneven-mu}.
\begin{cor}\label{cor:Cauchy}
For any $\beta>n/2$, the Cauchy type probability measure $\mu_\beta$ satisfies the following transport-entropy inequality: for all $\nu_1,\nu_2\in \mc P(\R^n)$ with $\nu_1$ and $\nu_2$ symmetric, we have
  \begin{equation}\label{eq:talCauchy}
 \beta \mathcal{T}_{\omega}(\nu_1,\nu_2) \leq    H(\nu_1|  \mu_\beta) + H(\nu_2|  \mu_\beta),
  \end{equation}
  where the optimal transport cost $ \mathcal{T}_{\omega}$ is defined with respect to the cost function $\omega$ given by
  \begin{equation}\label{eq:baromega}
    \omega(x,y)=\left\{\begin{array}{ll} -2\log\pare{\frac{1 +x\cdot y}{\sqrt{1+|x|^2}\sqrt{1+|y|^2}}}
&  \text{ if } x\cdot y >0  \\ +\infty & \text{ otherwise}  \end{array}\right. ,\qquad x,y \in \R^n.\end{equation}
\end{cor}
Note that a similar transport-entropy inequality holds true with respect to the cost function $\tilde{\omega}(x,y) = -2\log\pare{\frac{1 +|x\cdot y|}{\sqrt{1+|x|^2}\sqrt{1+|y|^2}}}$, $x,y\in \R^n$, without symmetry restrictions on $\nu_1,\nu_2$.

\begin{proof}
The function $t\mapsto \log (1+e^t)$ being convex on $\R$, the conclusion immediately follows from Theorem \ref{thm:Talagrand-noneven-mu} (Item $(ii)$).
\end{proof}

It turns out that sharp transport-entropy inequalities for a family of probability measures on the Euclidean unit sphere can be derived from Corollary \ref{cor:Cauchy}. To state this result, we need to introduce additional notation. Let 
\[
\Sd^n = \left\{u=(u_1,\ldots,u_{n+1}) : \sum_{i=1}^{n+1} u_i^2=1\right\}\qquad \text{and}\qquad \Sd^n_+=\Sd^n \cap \{u\in \R^{n+1} : u_{n+1}\geq0\}
\]
be respectively the $n$-dimensional Euclidean unit sphere and upper half unit sphere of $\R^{n+1}$ and denote by $\sigma$ the uniform probability measure on $\Sd^n$ and by $\sigma_+(\,\cdot\,) = 2 \sigma(\Sd^n_+ \cap \,\cdot\,)$ the normalized restriction of $\sigma$ to $\Sd^n_+$ (the dimension $n$ is omitted in the notation of $\sigma$ and $\sigma_+$). For any $\beta>n/2$, let $\sigma_{\beta,+} \in \mathcal{P}(\Sd^n_+)$ (resp. $\sigma_{\beta} \in \mathcal{P}(\Sd^n))$ be the probability measure with a density proportional to 
\[
u\mapsto|u_{n+1}|^{2\beta-(n+1)}
\]
with respect to $\sigma_+$ (resp. $\sigma$). Note that $\sigma$ and $\sigma_+$ correspond to the parameter $\beta = (n+1)/2$.

The set of Borel probability measures on $\Sd^n$ (resp. $\Sd^n_+$) will be denoted by $\mathcal{P}(\Sd^n)$ (resp. $\mathcal{P}(\Sd^n_+)$). A probability measure $\mu \in \mathcal{P}(\Sd^n)$ will be called symmetric if it is invariant under the map $\Sd^n\to \Sd^n :u \mapsto -u$. The set of all symmetric probability measures on $\Sd^n$ will be denoted by $\mathcal{P}_s(\Sd^n)$.

Finally, let  $\alpha : \Sd^n\times \Sd^n \to \R_+\cup\{+\infty\}$ be the cost function defined by 
\[
\alpha(u,v) = \left\{\begin{array}{ll}  \log \left(\frac{1}{u\cdot v}\right)  & \text{if } u\cdot v >0    \\ + \infty  & \text{otherwise}  \end{array}\right.,\qquad u,v \in \Sd^n
\]
and $\mathcal{T}_\alpha$ the associated optimal transport cost on $\mathcal{P}(\Sd^n)$. This cost function has been introduced by Oliker \cite{Olik07} (see also \cite{Ber16} and \cite{Kol20}) in connection with the so-called Aleksandrov problem in convex geometry. 

Recall the definition of the geodesic distance $d_{\Sd^n}$ on $\Sd^n$:
\[
d_{\Sd^n}(u,v) = \arccos (u\cdot v),\qquad u,v \in \Sd^n.
\]
The cost $\alpha$ can thus also be expressed as 
\begin{equation}\label{eq:alpha-geod}
\alpha(u,v) =\left\{\begin{array}{ll}  -\log \cos(d_{\Sd^n}(u,v))  & \text{if } d_{\Sd^n}(u,v)<\pi/2    \\ + \infty  & \text{otherwise}  \end{array}\right.,\qquad u,v \in \Sd^n.
\end{equation}

\begin{rem}\label{rem:Talphafinite} Characterizing couples $(\mu,\nu)$ for which the transport cost $\mathcal{T}_\alpha(\mu,\nu)$ is finite is a delicate question (discussed in particular in \cite{Ber16}; see also Remark \ref{rem:Tfinite} below). Note that, according to Lemma 3.3 of \cite{Kol20} and Remark 4.9 of \cite{Ber16}, if $\mu,\nu$ are symmetric probability measures such that $\mu$ has a positive density with respect to $\sigma$ and $\nu$ is such that $\nu(\Sd^n \cap L)=0$ for any hyperplane $L$ passing through the origin, then $\mathcal{T}_\alpha(\mu,\nu)<+\infty$.
\end{rem}

\begin{cor}\label{cor:transport-sphere}\ 
Let $\beta>n/2$.
\begin{itemize}
\item[$(i)$] For any $\nu_1,\nu_2 \in \mathcal{P}(\Sd^n_+)$ which are invariant under the map $\Sd^n_+\to \Sd^n_+ :u \mapsto (-u_1,\ldots,-u_n,u_{n+1})$, it holds 
\[
2\beta\mathcal{T}_\alpha(\nu_1,\nu_2) \leq H(\nu_1|\sigma_{\beta,+})+H(\nu_2|\sigma_{\beta,+}).
\]
\item[$(ii)$] For any $\nu_1,\nu_2 \in \mathcal{P}_s(\Sd^n)$ which are also invariant under the map $\Sd^n\to \Sd^n :u \mapsto (u_1,\ldots,u_n,-u_{n+1})$, it holds 
\[
2\beta\mathcal{T}_\alpha(\nu_1,\nu_2) \leq H(\nu_1|\sigma_\beta)+H(\nu_2|\sigma_\beta).
\]
\end{itemize}
\end{cor}

\proof Let us prove $(i)$, following the proof of \cite[Theorem 19]{Goz07}. Denote by $\mu = \mu_{(n+1)/2}$ the multivariate Cauchy distribution with density $Z^{-1}(1+\abs x^2)^{-(n+1)/2}$. 
Consider the map 
\[
T : \R^n \to \Sd^n_{++} : x \mapsto \frac{1}{(1+|x|^2)^{1/2}}(x,1),
\]
denoting by $ \Sd^n_{++} = \Sd^n \cap \{u\in \R^{n+1} : u_{n+1}>0 \}$. This transformation is bijective with inverse 
\[
T^{-1} : \Sd^{n}_{++} \to \R^n : u \mapsto \frac{1}{u_{n+1}}(u_1,\ldots,u_n),
\]
which is sometimes called \emph{gnomonic projection}. It is easy to check that $T^{-1}$ pushes forward $\sigma_+$ onto $\mu$, or equivalently that $T$ pushes forward $\mu$ onto $\sigma_+$. For any $\beta>n/2$, the probability measure $\mu_\beta$ has density 
\[
g_\beta(x) = \frac{C_\beta}{(1+|x|^2)^{\beta - \frac{n+1}{2}}}, \qquad x\in \R^n
\]
with respect to $\mu$. Therefore, the probability measure $T_\# \mu_\beta$ has density $g_\beta(T^{-1})$ with respect to $T_\#\mu = \sigma_+$. A simple calculation shows that 
\[
g_{\beta}(T^{-1}(u)) = C_\beta u_{n+1}^{2\beta - (n+)},\qquad u \in \Sd^n_+,
\]
and so $\sigma_{\beta,+} = T_\# \mu_\beta.$ 

Note the following relation between the cost functions $\omega$ (of Corollary \ref{cor:Cauchy}) and $\alpha$:
\[
\alpha(u,v) = \frac{1}{2}\omega(T^{-1}(u),T^{-1}(v)),\qquad \forall u,v \in \Sd^n_{++}.
\]
Let $\nu_1,\nu_2 \in \mathcal{P}(\Sd^n_{+})$ be invariant under the map $u \mapsto (-u_1,\ldots,-u_n,u_{n+1})$. If $H(\nu_1|\sigma_{\beta,+})=+\infty$ or $H(\nu_2|\sigma_{\beta,+})=+\infty$ there is nothing to prove. Let assume that $H(\nu_1|\sigma_{\beta,+})<+\infty$ and $H(\nu_2|\sigma_{\beta,+})<+\infty$. In particular, $\nu_1$ and $\nu_2$ do not give mass to $\Sd^n \cap \{u \in \R^{n+1} : u_{n+1}=0\}$ and can thus be seen as elements of $\mathcal{P}(\Sd^n_{++})$. Define $\nu_1':=T^{-1}_\#\nu_1$ and $\nu_2' :=T^{-1}_\#\nu_2$, which are symmetric and so, according to Corollary \ref{cor:Cauchy} applied to $\mu_\beta$, it holds
\[
\beta\mathcal{T}_{\omega}(\nu_1',\nu_2') \leq H(\nu_1' | \mu_\beta) + H(\nu_2' | \mu_\beta).
\]
If $\pi'$ is a coupling between $\nu_1'$ and $\nu_2'$ and $\pi$ is the push forward of $\pi'$ under the map $(x,y) \mapsto (T(x),T(y))$, it holds
\[
\frac{1}{2}\iint \omega(x,y)\,d\pi'(x,y) =\frac{1}{2}\iint \omega(T^{-1}(u),T^{-1}(v))\,d\pi(u,v) = \iint \alpha(u,v)\,d\pi(u,v) \geq \mathcal{T}_\alpha(\nu_1,\nu_2),
\]
since $\pi$ has  $\nu_1$ and $\nu_2$ as marginals. Therefore, $\mathcal{T}_\alpha(\nu_1,\nu_2) \leq \frac{1}{2}\mathcal{T}_{\omega}(\nu_1',\nu_2')$. Finally, a simple calculation shows that
\[
H(\nu_i' | \mu_\beta) = H(T^{-1}_\#\nu_i | T^{-1}_\#\sigma_{\beta,+}) = H(\nu_i | \sigma_{\beta,+}),
\]
which completes the proof of $(i)$.

Let us now prove $(ii)$. Let $\nu_1,\nu_2 \in \mathcal{P}(\Sd^n)$ be invariant under the maps $u\mapsto -u$ and $u\mapsto (u_1,\ldots,u_{n},-u_{n+1})$ with densities $f_1,f_2$ with respect to $\sigma_\beta$.
For $i=1,2$, it holds $\nu_i(\Sd^n_+)=1/2$. Define $d\nu_{i,+}(u) = 2f_i \mathbf{1}_{\Sd^n_+}(u)\,d\sigma_\beta(u)=f_i(u)\,d\sigma_{\beta,+}(u)$. Then it holds
\[
H(\nu_i|\sigma_\beta) = \int f_i \log f_i \,d\sigma_\beta = 2 \int_{\Sd^n_+} f_i \log f_i \,d\sigma_\beta = \int f_i \log f_i \,d\sigma_{\beta,+} = H(\nu_{i,+}|\sigma_{\beta,+}).
\]
On the other hand, if $(U,V)$ is a coupling between $\nu_{1,+}$ and $\nu_{2,+}$ and $\varepsilon$ is such that $\P(\varepsilon = \pm 1) =1/2$ and is independent of $(U,V)$, then $X=(U_1,\ldots,U_n,\varepsilon U_{n+1})$, $Y=(V_1,\ldots,V_n,\varepsilon V_{n+1})$ is a coupling between $\nu_1$ and $\nu_2$, and it holds $\E[\alpha(X,Y)] = \E[\alpha(U,V)]$. From this follows that $\mathcal{T}_\alpha(\nu_1,\nu_2) \leq \mathcal{T}_\alpha(\nu_{1,+},\nu_{2,+}).$ Thus $(ii)$ immediately follows from $(i)$, which completes the proof.
\endproof

For the probability measure $\sigma$ (corresponding to $\beta = (n+1)/2$), the conclusion of Corollary \ref{cor:transport-sphere} can be improved, as shows the following result. 

\begin{thm}\label{thm:Kolesnikovsym}For all  symmetric probability measures $\nu_1,\nu_2$ on $\mathbb{S}^n$, it holds
\begin{equation}\label{eq:Kolesnikovsym}
(n+1)\mathcal{T}_\alpha(\nu_1,\nu_2) \leq H(\nu_1|\sigma) + H(\nu_2|\sigma).
\end{equation}
\end{thm}

The preceding result is an improvement of a result by Kolesnikov \cite{Kol20} who obtained the following transport-entropy inequality on $\Sd^n$:
\begin{equation}\label{eq:Kolesnikov}
(n+1)\mathcal{T}_\alpha(\nu,\sigma) \leq H(\nu|\sigma),
\end{equation}
 for all symmetric probability $\nu \in \mathcal{P}(\Sd^n)$. The proof by Kolesnikov is based on the Monge-Ampère equation relating the density of $\nu$ to the optimal transport map $T$ transporting $\sigma$ on $\mu$. 
 The determinant of the Jacobian matrix of $T$ is controlled with the help of the classical Blaschke-Santal\'o inequality for convex bodies (see the proof of \cite[Theorem 7.3]{Kol20}). Kolesnikov also establishes links between minimizers of the functional 
 \[
 \nu_1 \mapsto H(\nu_1|\sigma)-(n+1)\mathcal{T}_\alpha(\nu_1,\nu_2),
 \]
 with $\nu_1,\nu_2$ symmetric and the log-Minkowski problem; we refer to \cite{Kol20} for further explanations and references. Remark \ref{rem:optimality} below gathers further comments on \eqref{eq:Kolesnikovsym}  and \eqref{eq:Kolesnikov}.
 
 Before turning to the proof of \eqref{eq:Kolesnikovsym}, let us comment the role of the symmetry assumption. 
 It turns out that for any constant $C>0$, the inequality 
 \[
 C\mathcal{T}_\alpha(\nu,\sigma) \leq H(\nu|\sigma), 
 \]
can not be true for all $\nu \in \mathcal{P}(\Sd^n)$. This follows immediately from the following lemma:
\begin{lem}\label{lem:nonsym}
There exists $\nu \in \mathcal{P}(\Sd^n)$ such that $\mathcal{T}_\alpha(\nu,\sigma) = +\infty$ and $H(\nu|\sigma)<+\infty$.
\end{lem}
In particular, contrary to Fathi's Inequality \eqref{eq:Fathi} for the standard Gaussian measure, \eqref{eq:Kolesnikovsym} is not true if only one of the probability measures $\nu_1,\nu_2$ is assumed to be symmetric. 
\proof[Proof of Lemma \ref{lem:nonsym}]
Let $A\subset \Sd^n$ be some spherical cap and define $d\nu = \frac{\mathbf{1}_A}{\sigma(A)}\,d\sigma$. Then $H(\nu| \sigma) = - \log \sigma (A) <+\infty$. On the other hand, if $(X,Y)$ is a coupling between $\sigma$ and $\nu$, then denoting by 
\[
A_{\pi/2} = \{y \in \Sd^n : \exists x \in A \text{ s.t. } d_{\Sd^n}(x,y) <\pi/2\},
\]
it holds
\[
\P(d(X,Y) <\frac{\pi}{2}) \leq \P(Y \in A_{\pi/2}) = \sigma(A_{\pi/2}).
\]
If $A$ is small enough, then $\sigma(A_{\pi/2}) <1$ and so $\P(d(X,Y) \geq \frac{\pi}{2}) >0$. 
Therefore, by definition of $\alpha$, $\E\left[\alpha (X,Y)\right]=+\infty$. The coupling being arbitrary, one concludes that $\mathcal{T}_\alpha (\nu,\sigma)=+\infty$. 
\endproof
\begin{rem}\label{rem:non-Fathi} Note that the preceding construction can be easily adapted to show that, for any $\beta>n/2$, \eqref{eq:talCauchy} can be false if only one of the measures $\nu_1,\nu_2$ is assumed to be symmetric.
\end{rem}

Our proof of Theorem \ref{thm:Kolesnikovsym} is based on the following Kantorovich type duality for the cost function $\alpha$. 
To state this result, let us introduce additional notation. Recall that if $C \subset \mathbb{R}^{n+1}$ is a convex body, the support function of $C$ is the function denoted by $h_C$ defined by 
\[
h_C(y) = \sup_{x \in C} x\cdot y,\qquad \forall y \in \R^{n+1}
\]
and when $C$ contains $0$ in its interior, the radial function of $C$ is the function denoted by $\rho_C$ defined by 
\[
\rho_C(x) = \sup\{ r \geq 0 : rx \in C\},\qquad \forall x\in \R^{n+1}.
\]
\begin{lem}\label{lem:Oliker}
For all probability measures $\nu_1,\nu_2$ on $\mathbb{S}^n$, it holds
\[
\mathcal{T}_\alpha(\nu_1,\nu_2) = \sup_{C} \int -\ln h_C\,d\nu_1+\int \ln \rho_C\,d\nu_2,
\]
where the supremum runs over the set of all convex bodies $C$ containing $0$ in their interiors.
Moreover, when $\nu_1$ and $\nu_2$ are symmetric, the supremum can be restricted to centrally symmetric convex bodies $C$.
\end{lem}
This duality relation has been first established by Oliker in \cite{Olik07} in his transport approach to the Alexandrov's problem on the Gauss curvature prescription of Euclidean convex sets (see also \cite{Ber16} in particular for the question of dual attainment). 
For the sake of completeness, we briefly sketch the proof of Lemma \ref{lem:Oliker}.
\proof
For any $\nu_1,\nu_2$ probability measures on $\mathbb{S}^n$, Kantorovich duality \cite[Theorem 5.10 (i)]{villani_book} yields to 
\begin{equation}\label{eq:duality-alpha}
\mathcal{T}_\alpha(\nu_1,\nu_2) = \sup_{\phi,\psi} \int \phi\,d\nu_1+\int \psi\,d\nu_2,
\end{equation}
where the supremum runs over the set of couples $(\phi,\psi)$ of bounded continuous functions on $\mathbb{S}^n$ such that 
\begin{equation}\label{eq:phipsialpha}
\phi(x)+\psi(y) \leq  \alpha(x,y),\qquad \forall x,y \in  \mathbb{S}^n.
\end{equation}
Whenever $\nu_1,\nu_2$ are symmetric, and $(\phi,\psi)$ satisfies \eqref{eq:phipsialpha}, then defining $\bar{\phi}(x) = \frac{1}{2}(\phi(x) + \phi(-x))$ and $\bar{\psi}(y) =  \frac{1}{2}(\psi(y) + \psi(-y))$, $x,y\in\mathbb{S}^n$, the couple $(\bar{\phi},\bar{\psi})$ satisfies \eqref{eq:phipsialpha} (because $\alpha(-x,-y) = \alpha(x,y)$) and is such that 
\[
\int \bar{\phi}\,d\nu_1 + \int \bar{\psi}\,d\nu_2  = \int \phi\,d\nu_1 + \int \psi\,d\nu_2. 
\]
Therefore, in this symmetric case, the supremum in \eqref{eq:duality-alpha} can be further restricted to couples of even functions $(\phi,\psi)$. 
Let us now consider the $\alpha$-transform $f^\alpha$ of a function $f : \mathbb{S}^n \to \R$ defined by
\[
f^\alpha(y) = \inf_{x\in \mathbb{S}^n}\{\alpha(x,y)-f(x)\},\qquad y\in \mathbb{S}^n.
\]
It is not difficult to check that whenever $f$ is bounded on $\mathbb{S}^n$, then $f^\alpha$ is bounded continuous on $\mathbb{S}^n$, and if $f$ is even then $f^\alpha$ is also even.
Using the well known double $\alpha$-concavification trick, the duality formula \eqref{eq:duality-alpha} can be further restricted to couples $(\phi,\psi)$ of $\alpha$-conjugate functions, that is to say such that
$\phi^\alpha = \psi$ and $\psi^\alpha = \phi$ (see \cite[Theorem 5.10 (i)]{villani_book}). Moreover, in the case where $\nu_1,\nu_2$ are symmetric, \eqref{eq:duality-alpha} can be restricted to couples $(\phi,\psi)$ of even $\alpha$-conjugate functions. With the change of functions $h= e^{-\phi}$ and $\rho=e^{\psi}$, we see that $(\phi,\psi)$ is a couple of continuous (even) $\alpha$-conjugate functions, if and only if $(h,\rho)$ is a couple of continuous (even) positive functions such that 
\[
h(x)=\sup_{y \in \mathbb{S}^n} \rho(y)x\cdot y,\quad \forall x\in\mathbb{S}^n \qquad \text{and} \qquad \frac{1}{\rho(y)} = \sup_{x \in  \mathbb{S}^n} \frac{x\cdot y}{h(x)}, \quad \forall y\in\mathbb{S}^n.
\]
It is well known that to any such couple $(h,\rho)$ uniquely corresponds a convex body $C$ containing $0$ in its interior such that $h=h_C$ and $\rho=\rho_C$; we refer to \cite[Theorem 2]{Olik07} for details. In the case, $h$ and $\rho$ are both even, then $C$ is centrally symmetric, which completes the proof. 
\endproof

\proof[Proof of Theorem \ref{thm:Kolesnikovsym}.]
Let $C$ be a centrally symmetric convex body in $\R^{n+1}$.  According to the classical Blaschke-Santal\'o inequality, it holds 
\[
|C||C^\circ| \leq |B_2^{n+1}|^2.
\]
Calculating the volume of $C$ in polar coordinates yields to
\[
|C| = (n+1)|B_2^{n+1}| \int_{\mathbb{S}^n}\left(\int_{\R^+} \mathbf{1}_C(ru) r^n \,dr\right)\,d\sigma(u) = |B_2^{n+1}| \int_{\mathbb{S}^n}\rho_C(u)^{n+1}\,d\sigma(u),
\]
where $\rho_C$ denotes the radial function of $C$.
Similarly, 
\[
|C^\circ| =  |B_2^{n+1}| \int_{\mathbb{S}^n}\rho_{C^\circ}(u)^{n+1}\,d\sigma(u) =  |B_2^{n+1}| \int_{\mathbb{S}^n}\frac{1}{h_{C}(u)^{n+1}} \,d\sigma(u),
\] 
using the well known (and easy to check) relation $\rho_{C^\circ} = 1/h_C$, where $h_C$ is the support function of $C$.
So, for every symmetric convex $C$ body in $\R^{n+1}$, it holds
\begin{equation}\label{eq:santalo-polar}
\int_{\mathbb{S}^n}\rho_C(u)^{n+1} \,d\sigma(u)\int_{\mathbb{S}^n}\frac{1}{h_{C}(u)^{n+1}} \,d\sigma(u) \leq 1.
\end{equation}
On the other hand, if $\nu_1,\nu_2$ are two symmetric probability measures on $\mathbb{S}^n$, Lemma \ref{lem:Oliker} yields 
\[
(n+1)\mathcal{T}_\alpha(\nu_1,\nu_2) = \sup_{C} \int -\ln \left(h_C^{n+1}\right)\,d\nu_1+\int \ln \left(\rho_C^{n+1}\right)\,d\nu_2,
\]
where the supremum runs over the set of all centrally symmetric convex bodies $C$ containing $0$ in their interiors. 
Reasoning exactly as in the proof of Theorem \ref{thm:Talagrand-noneven-mu}, one sees that \eqref{eq:santalo-polar} implies (and is in fact equivalent to)
\[
(n+1)\mathcal{T}_\alpha(\nu_1,\nu_2) \leq H(\nu_1|\sigma) + H(\nu_2|\sigma),
\]
for all $\nu_1,\nu_2$ symmetric. This completes the proof.
\endproof

In order to discuss Inequalities  \eqref{eq:Kolesnikovsym} and \eqref{eq:Kolesnikov}, let us recall that the uniform probability measure $\sigma$ on $\Sd^n$ satisfies the following Poincaré inequality: for any smooth function $f:\Sd^n \to \R$, 
\begin{equation}\label{eq:PoincaréSn}
\lambda_1(\Sd^n) \mathrm{Var}_\sigma(f) \leq \int |\nabla_{\Sd^n} f|^2\,d\sigma,
\end{equation}
with the sharp constant $\lambda_1(\Sd^n)=n$ (corresponding to the spectral gap of the Laplace operator on $\Sd^n$). Equality in \eqref{eq:PoincaréSn} is reached for every linear forms. Under symmetry assumptions, the constant in \eqref{eq:PoincaréSn} can be improved. More precisely, for all smooth functions $f:\Sd^n \to \R$ such that $f(-u)=f(u)$, for all $u\in \Sd^n$, it holds 
\begin{equation}\label{eq:PoincaréSnsym}
\lambda_2(\Sd^n) \mathrm{Var}_\sigma(f) \leq \int |\nabla_{\Sd^n} f|^2\,d\sigma,
\end{equation}
where $\lambda_2(\Sd^n) = 2(n+1)$ is the second non-zero eigenvalue of the Laplace operator on $\Sd^n$. Moreover, Equality in \eqref{eq:PoincaréSnsym} is reached whenever $f$ is the restriction to $\Sd^n$ of an homogeneous polynomial of degree $2$.
For the sake of completeness we recall the classical argument leading to \eqref{eq:PoincaréSnsym}.
\proof[Proof of \eqref{eq:PoincaréSnsym}]
For all $d=0,1,2\ldots$ denote by $H_d \subset L^2(\sigma)$ the space of degree $d$ homogeneous harmonic polynomials (restricted to $\Sd^n$). It is well known that 
\[
L^2(\sigma) =  \bigoplus_{d=0}^{+\infty} H_d
\]
and that for all $f \in H_d$, it holds $\Delta_{\mathbf{\Sd^n}} f = - d(d+n-1) f$.
If $f:\Sd^n \to \R$ is a smooth \emph{even} function then it can be written as $f=\sum_{k=0}^{+\infty} f_{2k}$, with $f_{2k} \in H_{2k}$, for all $k\geq0.$ Therefore, by integration by part:
\[
\int |\nabla_{\Sd^n} f|^2\,d\sigma = - \int f.\Delta_{\mathbf{\Sd^n}} f\,d\sigma =  \sum_{k=0}^{+\infty} 2k(2k+n-1) \int f_k^2\,d\sigma \geq 2(n+1) \sum_{k=1}^{+\infty}\int f_k^2\,d\sigma =2(n+1) \mathrm{Var}_\sigma(f),
\]
which proves \eqref{eq:PoincaréSnsym}. Whenever $f \in H_0\oplus H_2$, equality obviously holds. This is in particular the case if $f$ is the restriction to the sphere of a degree $2$ homogeneous polynomial. Indeed, suppose that $f = P_{| \Sd^n}$, where $P:\R^{n+1} \to \R$ is some degree $2$ homogeneous polynomial. Then there is some constant $c$ such that $\Delta_{\R^{n+1}} P = c$. The polynomial $Q$ defined by $Q(x) = P(x) - \frac{c}{2(n+1)} |x|^2$, $x \in \R^{n+1}$, is homogeneous of degree $2$ and harmonic. Moreover, it holds $f=Q_{| \Sd^n} + \frac{c}{2(n+1)}$ and so $f \in H_0\oplus H_2$.
\endproof

Recall the expression \eqref{eq:alpha-geod} which will be used in the following remark on optimality of \eqref{eq:Kolesnikovsym}.

\begin{rem}\label{rem:optimality}\ 
\begin{itemize}
\item First let us relate Kolesnikov's Inequality \eqref{eq:Kolesnikov} to existing transport-entropy inequalities on $\Sd^n$. A simple calculation shows that $-\log \cos u \geq \frac{u^2}{2}$ for all $u \in [0,\pi/2]$. Therefore, \eqref{eq:Kolesnikov} implies that for all symmetric probability measures $\nu$ on $\Sd^n$, it holds
\begin{equation}\label{eq:W2Sn1}
\frac{n+1}{2} W_2^2(\nu,\sigma) \leq H(\nu|\sigma),
\end{equation}
with $W_2$ being the usual Wasserstein distance on $\Sd^n$ (with respect to the geodesic distance $d_{\Sd^n}$). The inequality \eqref{eq:W2Sn1} is an improvement of the following classical transport-entropy inequality: 
\begin{equation}\label{eq:W2Sn2}
\frac{n}{2} W_2^2(\nu,\sigma) \leq H(\nu|\sigma),
\end{equation}
that holds for all $\nu \in \mathcal{P}(\Sd^n)$. Inequality \eqref{eq:W2Sn2} can for instance be deduced from the log-Sobolev inequality on $\Sd^n$ that holds with the optimal constant $2/n$ using the Otto-Villani theorem \cite{OV00}. The constant $n/2$ in \eqref{eq:W2Sn2} is optimal. Indeed, according to a well known general linearization argument of \cite{OV00}, \eqref{eq:W2Sn2} implies the sharp Poincar\'e inequality \eqref{eq:PoincaréSn}. 
Using the fact that the function $u \mapsto - \log \cos \sqrt{u}$ is convex and increasing on $[0,(\pi/2)^2]$, it follows from Jensen inequality that \eqref{eq:Kolesnikov} implies the following transport-entropy inequality:
\begin{equation}\label{eq:EKS}
-(n+1) \log \cos W_2(\nu,\sigma) \leq H(\nu|\sigma)
\end{equation}
for all symmetric $\nu \in \mathcal{P}(\Sd^n)$. Inequality \eqref{eq:EKS} improves the conclusion of \cite[Corollary 3.29]{EKS15} in the case of symmetric probability measures on $\Sd^n$. See Remark 7.4 of \cite{Kol20} for other transport-entropy inequalities derived from \eqref{eq:Kolesnikov}.
\item Now let us discuss the sharpness of Inequality \eqref{eq:Kolesnikovsym}. Reasoning as above, we see that \eqref{eq:Kolesnikovsym} implies the following variant of \eqref{eq:W2Sn1}:
\begin{equation}\label{eq:W2Sn3}
\frac{n+1}{2} W_2^2(\nu_1,\nu_2) \leq H(\nu_1|\sigma)+H(\nu_2|\sigma),
\end{equation}
for all symmetric probability measure $\nu_1,\nu_2$ on $\Sd^n$. Adapting the linearization argument of \cite{OV00} (see below for a sketch of proof), one can see that \eqref{eq:W2Sn3} implies the Poincaré inequality \eqref{eq:PoincaréSnsym} for smooth even functions  $f:\Sd^n\to \R$. In comparison, for the same class of functions $f$, \eqref{eq:W2Sn1} only yields to Poincar\'e inequality with the sub-optimal constant $\lambda = n+1$, so that \eqref{eq:Kolesnikovsym} is a strict improvement of \eqref{eq:Kolesnikov}. As explained above, the constant $2(n+1)$ is sharp, with equality obtained for instance for $f(u) = u_1^2$, $u\in \Sd^n$.
\end{itemize}
\end{rem}

For the sake of completeness, let us recall how to deduce the Poincaré inequality \eqref{eq:PoincaréSnsym} from \eqref{eq:W2Sn3}.
\begin{proof}[Proof of \eqref{eq:W2Sn3} $\Rightarrow$ \eqref{eq:PoincaréSnsym}]
Let $f:\Sd^n \to \R$ be a smooth and even function. Without loss of generality, one can also assume that $\int f\,d\sigma = 0$. Bounding the second order derivatives, one sees there is some constant $C>0$ such that
\[
f(v) \leq f(u) + |\nabla_{\Sd^n}f|(u) d_{\Sd^n}(u,v) + C d^2_{\Sd^n}(u,v),\qquad \forall u,v \in \Sd^n.
\]
For all $t>0$, consider $\nu_{1,t} = (1-tf)\sigma$ and $\nu_{2,t} = (1+tf)\sigma$. For all $t$ small enough, $\nu_{1,t}$ and $\nu_{2,t}$ are symmetric probability measures on $\Sd^n$.
If $\pi$ is an coupling between $\nu_{1,t}$ and $\nu_{2,t}$ for $W_2$, it holds
\begin{align*}
\int f^2\,d\sigma = \int f\,d\left(\frac{\nu_{2,t}-\nu_{1,t}}{2t}\right) & = \frac{1}{2t}\int f(v)-f(u)\,d\pi(u,v)\\
&\leq \frac{1}{2t}\int |\nabla_{\Sd^n}f|(u) d_{\Sd^n}(u,v) + C d^2_{\Sd^n}(u,v)\,d\pi(u,v)\\
& \leq \frac{1}{2t}\left(\int |\nabla_{\Sd^n}f|^2\,d\sigma\right)^{1/2} W_2(\nu_{1,t},\nu_{2,t}) + \frac{C}{2t}W_2^2(\nu_{1,t},\nu_{2,t}).
\end{align*}
According to \eqref{eq:W2Sn3}, it holds 
\[
\frac{1}{t^2}W_2^2(\nu_{1,t},\nu_{2,t}) \leq \frac{2}{n+1} \left(\frac{H(\nu_{1,t} | \mu)}{t^2} + \frac{H(\nu_{2,t} | \mu)}{t^2}\right),
\]
and a simple calculation shows that $\frac{H(\nu_{i,t} | \mu)}{t^2} \to \frac{1}{2} \int f^2\,d\sigma$. 
Therefore,
\[
\limsup_{t\to 0} \frac{1}{t^2}W_2^2(\nu_{1,t},\nu_{2,t}) \leq \frac{2}{n+1} \int f^2\,d\sigma.
\]
So passing to the limit above yields to
\[
\int f^2\,d\sigma \leq \frac{1}{2}\left(\int |\nabla_{\Sd^n}f|^2\,d\sigma\right)^{1/2}\left(\frac{2}{n+1}\int f^2\,d\sigma\right)^{1/2},
\]
which amounts to \eqref{eq:PoincaréSnsym}.
\end{proof}

In the following, we derive some simple consequences of Inequality \eqref{eq:Kolesnikovsym} in terms of measure concentration for symmetric sets of the sphere. Whenever $A,B \subset \Sd^n$, we will set
\[
d_{\Sd^n}(A,B) = \inf_{x\in A,y\in B} d_{\Sd^n}(x,y)
\]
to denote the distance between $A$ and $B$.

\begin{cor}
Suppose that $A,B \subset \Sd^n$ are two symmetric subsets of $\Sd^n$, then $d_{\Sd^n}(A,B) \leq \pi/2$ and it holds 
\begin{equation}\label{eq:AB}
\sigma(A)\sigma(B) \leq \cos^{n+1}(d_{\Sd^n}(A,B)).
\end{equation}
\end{cor}
\proof
The fact that $d_{\Sd^n}(A,B) \leq \pi/2$ is obvious. Inequality \eqref{eq:AB} is then immediately derived from the transport entropy inequality \eqref{eq:Kolesnikovsym} using a general argument by Marton which is detailed in e.g \cite[Theorem 10]{Goz07}.
\endproof
\begin{rem}
Inequality \eqref{eq:AB} is not always true for general sets $A,B$ such that $d_{\Sd^n}(A,B) \leq \pi/2$. Indeed, if $A$ and $B$ are two (small enough) spherical caps such that $d_{\Sd^n}(A,B)=\pi/2$, then Inequality \eqref{eq:AB} would imply that $\sigma(A)\sigma(B) = 0$ which is obviously false. 
\end{rem}
In particular, if $A$ is some symmetric set of $\Sd^n$ such that $\sigma(A)\geq 1/2$ and $B= \Sd^n \setminus A_r$, where $0 < r \leq \pi/2$ and $A_r = \{y \in \Sd^n : d_{\Sd^n}(y,A) < r\}$ is the $r$-enlargement of $A$, it holds 
\begin{equation}\label{eq:concentration1}
\sigma(\Sd^n\setminus A_r) \leq  2\cos^{n+1}(r),\qquad \forall 0 \leq r \leq \pi/2.
\end{equation}
In comparison, for a general set $A\subset \Sd^n$ such that $\sigma(A)\geq 1/2$, the classical Talagrand inequality \eqref{eq:W2Sn2} yields to
\begin{equation}\label{eq:concentration2}
\sigma(\Sd^n\setminus A_r) \leq  2e^{-\frac{nr^2}{4}},\qquad \forall 0 \leq r \leq \pi/2
\end{equation}
and, if $A$ is supposed symmetric, Inequality \eqref{eq:W2Sn3} gives
\begin{equation}\label{eq:concentration3}
\sigma(\Sd^n\setminus A_r) \leq  2e^{-\frac{(n+1)r^2}{2}},\qquad \forall 0 \leq r \leq \pi/2.
\end{equation}
Since $\cos(r) \leq e^{-r^2/2}$ for $r\leq 0\leq \pi/2$, the bound \eqref{eq:concentration1} is clearly better than bounds \eqref{eq:concentration2} and \eqref{eq:concentration3}.
On the other hand, the classical isoperimetric inequality on $\Sd^n$ implies that if a general set $A\subset \Sd^n$ is such that $\sigma(A)\geq 1/2$, then
\begin{equation}\label{eq:concentration4}
\sigma(\Sd^n\setminus A_r) \leq \psi_n(r):=\frac{1}{2s_n} \int_r^{\pi/2}\cos^{n-1}(u)\,du,\qquad \forall r\geq 0,
\end{equation}
with $s_n = \int_0^{\pi/2} \cos^{n-1}(u)\,du$ (see e.g \cite{Ledoux}), with equality if $A$ is a spherical cap of measure $1/2$.
It is not difficult to see that
\[
\frac{\cos^n (r)}{n}\leq \int_r^{\pi/2}\cos^{n-1}(u)\,du \leq \frac{1}{\sin (r)} \frac{\cos^n (r)}{n},\qquad \forall 0<r\leq \pi/2
\]
and $s_n \sim \sqrt{\frac{\pi}{2n}}$, so that for any $0<a<b<\frac{\pi}{2}$,
\[
c\frac{\cos^{n+1} (r)}{\sqrt{n}} \leq \psi_n(r) \leq \frac{c'}{\sin (a) \cos (b)} \frac{\cos^{n+1} (r)}{\sqrt{n}},\qquad \forall   r \in [a,b],
\]
where $c,c'$ are constants independent of $a,b$ and $n$. Thus for $r \in [a,b]$ the bound \eqref{eq:concentration1} is off only by a factor of order $1/\sqrt{n}$ from the optimal bound \eqref{eq:concentration4}.

\section{Linearization of transport-Entropy inequalities}

In this section, we show that the transport-entropy inequality \eqref{eq:tal2} gives back the following sharp Brascamp-Lieb type inequality due to Cordero-Erausquin and Rotem \cite{CER}.

\begin{thm}\label{th:sharp_poincare}
  Assume that $t\mapsto v_\rho(t) = -\log\rho(e^t)$ is convex and increasing. Then, for all $f\in\mc C^\infty_c(\R^n)$ even and such that $\int f\,d\mu_\rho=0$,
  \begin{equation}\label{eq:weighted_poincare}
  \int f^2\,d\mu_\rho \leq \frac{1}{2}\int H_\rho^{-1}\nabla f\cdot\nabla f\,d\mu_\rho,
  \end{equation}
  where the positive matrix $H_\rho$ is given by
  \[
  \frac12 H_\rho(y) = \frac{1}{\abs y^2}\sbra{\pare{I_n-\frac{y\otimes y}{\abs y^2}}v_\rho'(s) + \frac{y\otimes y}{\abs y^2}v_\rho''(s)} 
  \]
  where we set $s=2\log \abs y$.
\end{thm}

\begin{rem}
  This result is exactly the one obtained in \cite[Theorem 3]{CER} for the probability $\mu_\rho$. Namely, using the same notation as in \cite{CER}, if $v_\rho(s)=w(e^{s/2})$, we find
  \[
  2H_\rho(y) = \frac{w'(\abs y)}{\abs y}\pare{2I_n - \frac{y\otimes y}{\abs y^2}} +
  \frac{y\otimes y}{\abs y^2}w''(\abs y),
  \]
  which is easily seen to be the same matrix as the one appearing in \cite[Theorem 3]{CER}.
  As observed in \cite{CER}, the Poincaré inequality~\eqref{eq:weighted_poincare} admits non-trivial equality cases, and is therefore sharp.
  Note however that \cite[Theorem 3]{CER} is much stronger than Theorem \ref{th:sharp_poincare} above since it shows that the weighted Poincaré inequality~\eqref{eq:weighted_poincare} is satisfied not only by the model probability measure $\mu_\rho$ but also by any log-concave perturbation of $\mu_\rho.$ This raises the question to know if \eqref{eq:tal2} is also true for log-concave perturbations of $\mu_\rho$.
\end{rem}

Our proof, adapted from ~\cite{CE17}, relies on a well known linearization technique involving the following Hopf-Lax operator
\begin{equation}\label{eq:hopf-lax}
  R F(y) = \inf_{x\in \R^n}\{ F(x)+ \omega_\rho(x,y)\},\qquad y\in \R^n,
\end{equation}
where we recall that the cost function $\omega_\rho$ is defined by
\begin{equation}\label{eq:cost}
  \omega_\rho(x,y)=
  \begin{cases}
    \log\pare{\frac{\rho(x\cdot y)^2}{\rho(\abs x^2)\rho(\abs y^2)}} & \text{if }x\cdot y>0\\
    +\infty & \text{otherwise}
  \end{cases}.
\end{equation}

The following result collects some properties of the cost function $\omega_\rho$ and in particular relates the matrix $H_\rho$ to the behaviour of $\omega_\rho$ near the diagonal.

\begin{lem}\label{lem:cost_prop}
  Assume $\rho:\R_+^*\to\R_+^*$ is nonincreasing, and that $t\mapsto \rho(e^t)$ is log-concave. The cost function $\omega_\rho$ defined in~\eqref{eq:cost} then satisfies the following:
  \begin{enumerate}
  \item $\omega_\rho\geq 0$.
  \end{enumerate}
  If $t\mapsto \rho(e^t)$ is furthermore assumed to be \emph{strictly} log-concave, then
  \begin{enumerate}[resume]
  \item if $\rho$ is of class $\mc C^3$, then for every $y\neq 0$, there exists a symmetric definite positive matrix $H_\rho$ such that
    \[
    \omega_\rho(y+h,y) = \frac12 H_\rho h\cdot h + o(|h|^2)
    \]
    when $h\to 0$;
  \item for every compact subset $K$ and $\rho>0$, there exists a constant $\eta >0$ such that for all $x\in K, y\in\R^n$ 
    \[
    \abs{x-y} > \delta \implies \omega_\rho(x,y) \geq \eta.
    \]
  \end{enumerate}
\end{lem}
\begin{rem}
  The log-concavity of $t\mapsto \rho(e^t)$ is, in fact, equivalent to the nonnegativity of $\omega_\rho$ if $\rho$ is assumed nonincreasing.
\end{rem}
\begin{proof}
  First, note that by monotonicity, for any $x, y\in\R^n$,
  \[
  \omega_\rho(x, y) \geq \log\pare{\frac{\rho(|x||y|)^2}{\rho(\abs x^2)\rho(\abs y^2)}}.
  \]
  To prove point \emph{(1)}, it suffices to show that, for any $s,t>0$,
  \[
  \log\pare{\frac{\rho(e^{s/2}e^{t/2})^2}{\rho(e^s)\rho(e^t)}}\geq 0.
  \]
  Rewriting this inequality in terms of $v_\rho(t)=-\log(\rho(e^t))$, we find that it is equivalent to
  \[
  v_\rho\pare{\frac{s+t}{2}} \leq \frac12 v_\rho(s) + \frac12 v_\rho(t),
  \]
  which in turn is equivalent to the convexity of $v_\rho$.

  Item \emph{(2)} is a direct consequence of the computation of the second derivative of $\varphi(h)=\omega_\rho(y, y+h)$ or, in terms of the function $v_\rho$, $\varphi(h)=-2v_\rho\pare{\log(y\cdot(y+h))} + v_\rho(\log(\abs y^2)) + v_\rho(\log(\abs{y+h}^2))$. We find that
  \[
  \nabla\varphi(0)=0,\qquad\nabla^2\varphi(0) = \frac{2}{\abs y^2}\sbra{\pare{I_n-\frac{y\otimes y}{\abs y^2}}v_\rho'(s) + \frac{y\otimes y}{\abs y^2}v_\rho''(s)}\eqcolon H_\rho,
  \]
  where we wrote $\abs y^2 = e^s$ for brievety. Strict convexity implies monotonicity of $v_\rho$, so both matrices appearing in the Hessian are nonnegative. Moreover, the second matrix is positive on the line spanned by $y$, and the first matrix is positive on its orthogonal, thus their sum must be positive. For future reference, we may rewrite $H_\rho$ in terms of $\rho$ rather than $v_\rho$:
  \[
  \frac12 H_\rho = -\frac{\rho'(s)}{\rho(s)}I_n + \pare{\frac{\rho'^2(s)}{\rho^2(s)} - \frac{\rho''(s)}{\rho(s)}}(y\otimes y).
  \]
  A Taylor expansion yields the formula of item \emph{(2)}.

  The last point is an immediate (but useful enough to be stated) consequence of the strict convexity of $v_\rho$. Notice that $\omega_\rho(x,y) >0$ whenever $x\neq y$. This is true because the monotonicity and the convexity of $\rho$ are strict. The stated result is then simply the consequence of continuity, if $x$ and $y$ are taken in some compact sets. However, we want a uniform estimate when $y$ is any point in $\R^n$, which is a bit more than we can say with just continuity.
  Fix $R>0$. So far, we proved that the property is true for all $x, y$ such that $\abs x < R$ and $\abs y< 2R$. If $\abs y \geq 2R$, then
  \begin{align*}
    \omega_\rho(x,y) &\geq \log\pare{\frac{\rho(|x||y|)^2}{\rho(\abs x^2)\rho(\abs y^2)}} \\
    &= -2v_\rho\pare{\frac{s+t}{2}} +  v_\rho(s) + v_\rho(t)
  \end{align*}
  if we once again write $\abs x^2=e^s$ and $\abs y^2=e^t$. Since $v_\rho$ is convex, $v_\rho'$ is nondecreasing, and we find that
  \begin{equation*}
    \omega_\rho(x,y) \geq -2v_\rho\pare{\frac{\log(R^2)+\log(4R^2)}{2}} +  v_\rho(\log(R^2)) + v_\rho(\log(4R^2)) > 0. 
  \end{equation*}
  Combining this estimate at infinity with the local one we had due to continuity, we may conclude.
\end{proof}

The next result establishes some Hamilton-Jacobi type (in)equation for $R(\varepsilon f)$ as $\varepsilon\to 0$.

\begin{lem}\label{lem:hopf-lax_lin}
  Let $f\in \mc C^\infty_c(\R^n)$, and assume that $\rho$ is strictly decreasing, and that $t\mapsto \rho(e^t)$ is log-concave. Then
  \begin{equation}
    R(\ep f) \geq \ep f - \frac12\ep^2H_\rho^{-1}\nabla f\cdot\nabla f + o(\ep^2),
  \end{equation}
  when $\ep$ goes to $0$, with
  \begin{equation}\label{eq:H}
    \frac12 H_\rho = -\frac{\rho'(s)}{\rho(s)}I_n + \pare{\frac{\rho'^2(s)}{\rho^2(s)} - \frac{\rho''(s)}{\rho(s)}}(y\otimes y), \quad s=2\log\abs y.
  \end{equation}
\end{lem}
\begin{proof}
  As is usual when linearizing such semigroups, the key is to localize the infimum. Namely, recalling~\eqref{eq:hopf-lax},
  \[
  R(\ep f)(y) = \inf_{x\in\R^n} \{\ep f(x) + \omega_\rho(x, y)\},
  \]
  if $x_\ep$ is a minimizer of this expression, we want to prove that $\abs{x_\ep-y}$ goes to $0$ uniformly in $y$ as $\ep$ goes to 0.
  Of course, we must also prove that such a $x_\ep$ exists.

  We would like the result to be independent from the variable $y$. To that end, notice that since $f$ has compact support, we may restrict the study to $y$ in a compact subset of $\R^n$. Indeed, notice that, in general $R(\ep f)(y) \leq \ep f(y)$. Assume more specificaly now that $y\in\supp(f)^c$. In that case, $R(\ep f)(y)\leq 0$. Since $\omega_\rho \geq 0$, the infimum in the Hopf-Lax semigroup can only be reached for $x=y$, or for $x\in\supp(f)$. In other words, whenever $y\in\supp(f)^c$,
  \[
  R(\ep f)(y) = \inf_{x\in\R^n}\{\ep f(x) + \omega_\rho(x,y)\} = \min(0, \inf_{x\in\supp f}\{\ep f(x) + \omega_\rho(x,y)\})
  \]
  Furthermore, according to the point \emph{iii.} of Lemma~\ref{lem:cost_prop}, there exist $\nu>0$ such that $x\in\supp(f)$ and $\abs{x-y}>1$ implies that $\omega_\rho(x,y)>\eta$. As such, if $\ep < \eta/\norm f_\infty$, $d(y,\supp(f))>\delta$ implies that $R(\ep f)(y)= 0$.

  We now restrict our study to some ball $B$ that contains $\supp(f)+B(0,1)$. Assume that $y\in B$.
  To make the calculations a little bit clearer, we rewrite~\eqref{eq:hopf-lax} as
  \[
  R(\ep f)(y) = \inf_{h\in\R^n} \{\ep f(y+h) + \omega_\rho(y+h, y)\},
  \]
  The immediate estimate $R(\ep f) \leq \ep \norm f_\infty$ means that to find the infimum, we may restrict $h$ to be in the set
  \[
  \{h\in\R^n, \ep f(y+h) + \omega_\rho(y+h, y) \leq \ep \norm f_\infty\} \subset \{h\in\R^n, \omega_\rho(y+h, y) \leq 2\ep \norm f_\infty\}
  \]
  Now, recall that for any $y\in B$,
  \[
  \omega_\rho(y+h,y) = \frac12 H_\rho h\cdot h + o(|h|^2),
  \]
  where $H_\rho$ is a continuous (positive definite) function of $y$, and the remainder term is uniform in $y$. This implies that there exists $r,\delta >0$ such that $\abs h<r$ implies
  \[
  \omega_\rho(y+h,y) \geq \delta\abs h^2.
  \]
  Owing to point (3) of Lemma~\ref{lem:cost_prop}, there also exists $\eta'>0$ such that if $\abs h > r$, then
  \[
  \omega_\rho(y+h,y) \geq \eta'.
  \]
  If $2\ep\norm f_\infty < \eta'$, then $\omega_\rho(y+h, y) \leq 2\ep \norm f_\infty$ implies that $\abs h<r$, and thus
  \begin{align*}
      R(\ep f)(y) = \inf_{\abs h<r} \{\ep f(y+h) + \omega_\rho(y+h, y)\},
  \end{align*}
  The fact that $B(0,r)$ is compact implies the existence of a minimizer $h_\ep$,
  \[
  R(\ep f)(y) = \ep f(y+h_\ep) + \omega_\rho(y+h_\ep, y).
  \]
  Since $\omega_\rho(y+h,y)\geq \eta\abs h^2$, we can already state that $\abs{h_\ep} \leq C\sqrt\ep$ for some constant $C$ independent from $y$, but we can do better. The function$f$ is Lipschitz for some constant $L>0$. Then,
  \[
  \ep f(y) - \ep L\abs{h_\ep} + \delta\abs{h_\ep}^2 \leq R(\ep f)(y) \leq \ep f(y),
  \]
  and thus $ \abs{h_\ep}\leq C'\ep$ for $C'=L/\delta>0$, which we emphasize is independent from $y$.

  Now that the minimizer $h_\ep$ is localized, the rest follows naturally.
  \begin{align*}
    R(\ep f)(y) &= \ep f(y+h_\ep) + \omega_\rho(y+h_\ep, y) \\
    & = \ep f(y) + \ep\nabla f(y)\cdot h_\ep + \frac{1}{2}H_\rho h_\ep\cdot h_\ep + o(\ep^2)\\
    & \geq \ep f(y) - \frac12\ep^2 H_\rho^{-1}\nabla f(y)\cdot \nabla f(y) + o(\ep^2),
  \end{align*}
  since $H_\rho z\cdot z\geq 0$, where $z=h_\ep+\ep H_\rho^{-1}\nabla f(y)$.
\end{proof}

We are now in position to prove Theorem \ref{th:sharp_poincare}. Let us underline that in order to retrieve the sharp constant in the final inequality, one needs to consider a two sided linearization involving $\mc T_{\omega_\rho}((1-\ep f)\mu_\rho, (1+\ep f)\mu_\rho)$, rather than $\mc T_{\omega_\rho}((1+\ep f)\mu_\rho,\mu_\rho)$.

\begin{proof}[Proof of Theorem~\ref{th:sharp_poincare}]
  Choose $f\in \mc C^\infty_c(\R^n)$ such that its integral against $\mu_\rho$ is 0, and consider, for $\ep >0$, $\nu_1= (1+\ep f)\mu_\rho$ and $\nu_2= (1-\ep f)\mu_\rho$. Linearizing the entropy is straightforward: since $(1+\ep)\ln(1+\ep) = \ep^2/2 + o(\ep^2)$, the right-hand side of inequality~\eqref{eq:tal2} is equal to
  \[
  H((1+\ep f)\mu_\rho |\mu_\rho) + H((1-\ep f)\mu_\rho |\mu_\rho) = \ep^2\int f^2\,d\mu_\rho + o(\ep^2).
  \]
  For the left hand side, note that since $R(\ep f)(y) - \ep f(x) \leq \omega_\rho(x,y)$,
  \begin{align*}
    \mc T_{\omega_\rho}(\nu_1,\nu_2) &\geq \int R(\ep f)\,d\nu_1 - \int \ep f\,d\nu_2\\
    &= \int R(\ep f)(1+\ep f)\,d\mu_\rho - \int \ep f(1-\ep f)\,d\mu_\rho 
  \end{align*}
  Lemma~\ref{lem:hopf-lax_lin} applies, and assuming that $\ep$ is sufficiently small, the remainder term is uniform and zero outside of a compact. We may integrate it to find
  \begin{align*}
    \mc T_{\omega_\rho}(\nu_1,\nu_2) &\geq \int \pare{\ep f - \frac{\ep^2}{2}H_\rho^{-1}\nabla f\cdot\nabla f}(1+\ep f)\,d\mu_\rho - \int \ep f(1-\ep f)\,d\mu_\rho + o(\ep^2)\\
    &= \ep^2\pare{-\frac12\int H_\rho^{-1}\nabla f\cdot\nabla f\,d\mu_\rho + 2\int f^2\,d\mu_\rho  + o(1)}
  \end{align*}
  Combine these two observations to find that, after dividing by $\ep^2$, letting it go to 0 leads to the claimed Poincaré inequality.  
\end{proof}

\section{Transport-entropy form of reverse Blaschke-Santal\'o inequality on the sphere}\label{sec:reverse-sphere}
The aim of this section is to draw connections between inverse Blaschke-Santal\'o inequalities and cone measures, in the spirit of \cite{Gozlan21, FGZ23}.

\subsection{A short reminder about cone measures}\label{sec:cone}
Let us first recall the definition of a cone measure.

\begin{defi}[Cone measures] Let $C \subset \R^{n+1}$ be a centrally symmetric convex body of volume $1$. The cone measure $\nu_C$ of $C$ is the pushforward of the uniform probability measure on $C$ under the map 
\[
C \to \Sd^n : x\mapsto N_C \left( \rho_C(x)x \right),
\]
with $N_C:\partial C \to \Sd^n$ the Gauss map and $\rho_C$ the radial function of $C$.
\end{defi}

A characterization of cones measures has been obtained by B\"or\"oczky, Lutwak, Yang and Zhang in \cite{BLYZ13}. It is shown there that a symmetric probability measure $\nu$ on $\Sd^n$ is the cone measure of some centrally symmetric convex body if and only if it satisfies the so-called \emph{subspace concentration condition}, which reads as follows: for every subspace $F \subset \R^{n+1}$ of dimension $1\leq k \leq n$, it holds 
\begin{equation}\label{eq:subspace}
\nu(\Sd^n \cap F) \leq \frac{k}{n+1} 
\end{equation}
and moreover, if there is equality in \eqref{eq:subspace} for some subspace $F$, then there is another subspace $G$ such that $F \cap G =\{0\}$ and $\mathrm{dim} (G) = n+1-k$ such that
\[
\nu(\Sd^n \cap G) = \frac{n+1-k}{n+1}
\]
(and so, in particular, $\nu(\Sd^n \cap (F\cup G) = 1)$). Note in particular, that any probability measure such that $\nu(\Sd^n \cap F)=0$ for any hyperplan $F$ satisfies the subspace concentration condition and is therefore the cone measure of some centrally symmetric convex body.

Denote by $\mathrm{Conv}_s(\R^{n+1})$ the set of all centrally symmetric convex bodies of $\R^{n+1}$. In order to construct a convex body $C$ such that $\nu=\nu_C$, the strategy of proof of \cite{BLYZ13} relies on minimizing the following functional:
\[
\Phi_\nu (C) = \int \log h_C\,d\nu
\]
over $\{C \in \mathrm{Conv}_s(\R^{n+1}) : |C|=1\}$. According to Theorem 6.3 of \cite{BLYZ13}, if $\nu$ satisfies the \emph{strict subspace concentration inequality}, that is if for all subspace $F$ of dimension $1\leq k \leq n$ the inequality in \eqref{eq:subspace} is strict, then there is some $C_o\in \mathrm{Conv}_s(\R^{n+1})$, with $|C_o|=1$ such that 
\[
\inf_{C \in \mathrm{Conv}_s(\R^{n+1}) : |C|=1} \Phi_\nu(C) = \Phi_\nu(C_o)
\]
and moreover $\nu = \nu_{C_o}$. 

Recall that $\mathcal{P}_s(\Sd^n)$ denotes the set of symmetric probability measures on the $n$-dimensional sphere $\Sd^n \subset \R^{n+1}$. In what follows, for all $\nu \in \mathcal{P}_s(\Sd^n)$, we will set
\[
K(\nu) = \inf_{C \in \mathrm{Conv}_s(\R^{n+1})} \{\Phi_\nu(C)\}.
\]
It turns out that the functional $K$ can be related to those considered in Section \ref{sec:TE}. For any $\nu \in \mathcal{P}_s(\Sd^n)$, consider the functional $F_\nu$ defined by 
\[
F_\nu(\eta) = \frac{1}{n+1}H(\eta | \sigma) - \mathcal{T}_\alpha(\nu,\eta),
\]
where $\mathcal{T}_\alpha$ is the transport cost introduced in Section \ref{sec:Cauchy}. The quantity $F_\nu(\eta)$ always makes sense in $\R\cup\{-\infty\}$ whenever $\eta$ belongs to \[
\mathcal{P}_s^*(\Sd^n):=\{ \eta \in \mathcal{P}_s(\Sd^n) : H(\eta|\sigma)<+\infty\}.
\]

\begin{lem}\label{lem:minimizers} Let $\nu \in \mathcal{P}_s(\Sd^n)$;  
\begin{itemize}
\item[(a)] It holds 
\[
K(\nu) = \inf_{\eta \in \mathcal{P}_s^*(\Sd^n)} F_\nu(\eta)-\frac{\log |B_2^{n+1}|}{n+1},
\]
where $|B_2^{n+1}|$ denotes the Lebesgue measure of the unit ball $B_2^{n+1} \subset \R^{n+1}$. 
\item[(b)] Moreover, if $K(\nu)$ is finite and $\Phi_\nu$ attains its infimum at some $C_o \in \mathrm{Conv}_s(\R^{n+1})$, with $|C_o|=1$, then $\nu$ is the cone measure of $C_o$ ($\nu=\nu_{C_o}$) and $F_\nu$ attains its infimum at 
\[
d\eta_{C_o} = |B_2^{n+1}|\rho_{C_o}^{n+1}\,d\sigma.
\]
\end{itemize}
\end{lem}
A proof of this lemma can be found in \cite{Kol20}, but we include a proof for the sake of completeness.
\proof
(a) Note that
\begin{equation}\label{eq:Kbis}
K(\nu) = \inf_{C\in \mathrm{Conv}_s(\R^{n+1})} \left\{\int \log h_C\,d\nu - \frac{\log \int \rho_C^{n+1}\,d\sigma}{n+1} \right\}-\frac{\log |B_2^{n+1}|}{n+1}.
\end{equation}
This follows from the fact that if $C \in \mathrm{Conv}_s(\R^{n+1})$, then
\[
\int \rho_C^{n+1}\,d\sigma = \frac{|C|}{|B_2^{n+1}|}
\]
and from the identity $h_{\lambda C} = \lambda h_C$.
According to a classical duality formula relating Log-Laplace and relative entropy functionals, for any bounded measurable $f : \Sd^n \to \R$,
\[
\log \int e^f\,d\sigma = \sup_{\eta \in \mathcal{P}_s(\Sd^n)}\left\{\int e^f\,d\eta - H(\eta | \sigma)\right\},
\]
with, as a convention $H(\eta | \sigma) = +\infty$ whenever $\eta$ is not absolutely continuous with respect to $\sigma$. Thus, applying this formula to $f = \log \rho_C^{n+1}$, one gets
\begin{align*}
K(\nu) &= \inf_{C\in \mathrm{Conv}_s(\R^{n+1})} \inf_{\eta \in \mathcal{P}_s(\Sd^n)} \left\{\int \log h_C\,d\nu - \int \log \rho_C\,d\eta+\frac{H(\eta|\sigma)}{n+1} \right\}\\
&=  \inf_{\eta \in \mathcal{P}_s(\Sd^n)} \inf_{C\in \mathrm{Conv}_s(\R^{n+1})}\left\{\int \log h_C\,d\nu - \int \log \rho_C\,d\eta+\frac{H(\eta|\sigma)}{n+1} \right\}\\
&=  \inf_{\eta \in \mathcal{P}_s(\Sd^n)} \left\{\frac{1}{n+1}H(\eta|\sigma)-\mathcal{T}_\alpha(\nu,\eta)\right\},
\end{align*}
where the last equality comes from Lemma \ref{lem:Oliker}. 

(b) Now let us examine equality cases. Suppose that $K(\nu)$ is finite and $\Phi_\nu$ attains its minimal value at $C_o \in \mathrm{Conv}_s(\R^{n+1})$, with $|C_o|=1$. The fact that $\nu$ is the cone measure of $C_o$ is given by Lemma 4.1 of \cite{BLYZ13}. Define $\eta_{C_o}$ as in the statement; since $C_o$ is of volume $1$, $\eta_{C_o}$ is a probability measure. Applying Lemma \ref{lem:Oliker}, one gets
\begin{align*}
K(\nu) &=  \int \log h_{C_o}\,d\nu - \frac{\log \int \rho_{C_o}^{n+1}\,d\sigma}{n+1} - \frac{\log |B_2^{n+1}|}{n+1}\\
& = - \left[-\int \log h_{C_o}\,d\nu + \int \log \rho_{C_o}\,d\eta_{C_o}\right] + \frac{H(\eta_{C_o}|\sigma)}{n+1}- \frac{\log |B_2^{n+1}|}{n+1}\\
& \geq F_\nu(\eta_{C_o})- \frac{\log |B_2^{n+1}|}{n+1} \geq K(\nu),
\end{align*}
and, so all inequalities above are equalities.
\endproof

\subsection{From reverse Blaschke-Santal\'o inequalities to log-Sobolev type inequalities on the sphere}
Theorem \ref{thm:Kolesnikovsym}, which is a direct consequence of Blaschke-Santal\'o inequality,  can obviously be restated in terms of a lower bound for the functional $K$: for all $\nu \in \mathcal{P}_s(\Sd^n)$, 
\[
K(\nu) \geq - \frac{1}{n+1}H(\nu|\sigma)-\frac{\log |B_2^{n+1}|}{n+1}.
\]
This shows in particular that, if $\nu \in \mathcal{P}_s^*(\Sd^n)$, the functional $F_\nu$ takes finite values. The aim of what follows is now to derive upper bounds on $K$ from inverse Blaschke-Santal\'o inequalities. For all $k \geq 1$, recall the notation
\[
c_k^S = \inf_{\mathrm{Conv}_s(\R^{k})} |C||C^\circ|,
\]
given in the Introduction.
A celebrated conjecture due to Mahler \cite{Mah39a,Mah39b}, states that the infimum above is attained for $C= [-1,1]^k$ or equivalently that $c_k^S= 4^k/(k!)$. According to a result by Bourgain and Milman \cite{bm87}, this conjecture is known to be true up to a geometric sequence. More precisely, there exists some $a>0$ such that, for all $k\geq 1$, $c_k^S \geq a^k/(k!).$ The following result connects the constant $c_k^S$ and the functional $K$ introduced above.

\begin{prop}\label{prop:KUB}
For any $\nu_1,\nu_2 \in \mathcal{P}_s(\Sd^n)$, it holds 
\begin{equation}\label{eq:KUB}
K(\nu_1)+K(\nu_2) \leq - \frac{\log\left(c_{n+1}^S\right)}{n+1}-\mathcal{T}_\alpha(\nu_1,\nu_2)
\end{equation}
\end{prop}
\begin{rem}\label{rem:Tfinite}
Note that this inequality compares two quantities in $\R\cup \{-\infty\}$. It entails in particular the following non obvious fact: if $\nu_1$ and $\nu_2$ satisfy the strict subspace concentration inequality introduced above, then, according to \cite[Theorem 6.3]{BLYZ13}, $K(\nu_1) >-\infty$ and $K(\nu_2)>-\infty$, and so $\mathcal{T}_\alpha(\nu_1,\nu_2)<+\infty$.
\end{rem}

\proof Using polar coordinates (as in the proof of Theorem \ref{thm:Kolesnikovsym}), one sees that
\[
|C| = |B_2^{n+1}|\int_{\Sd^n} \rho_C^{n+1}(u)\,d\sigma(u)
\]
and 
\[
|C^\circ| = |B_2^{n+1}|\int_{\Sd^n} \rho_{C^\circ}^{n+1}(u)\,d\sigma(u)
\]
hold for all $C \in \mathrm{Conv}_s(\R^{n+1})$.
Plugging these expressions in
\[
|C||C^\circ| \geq c_{n+1}^S
\]
yields
\[
-\frac{1}{n+1}\log \int_{\Sd^n} \rho_C^{n+1}(u)\,d\sigma(u) - \frac{1}{n+1}\log \int_{\Sd^n} \rho_{C^\circ}^{n+1}(u)\,d\sigma(u) - \frac{2 \log |B_2^{n+1}|}{n+1}  \leq  - \frac{\log\left(c_{n+1}^S\right)}{n+1}.
\]
So, if $\nu_1,\nu_2 \in \mathcal{P}_s(\Sd^n)$, one gets thanks to \eqref{eq:Kbis}
\begin{align*}
K(\nu_1)+K(\nu_2)& 
\leq\int \log h_C(u)\,d\nu_1(u) - \frac{1}{n+1}\log \int \rho_C^{n+1}(u)\,d\sigma(u) - \frac{\log |B_2^{n+1}|}{n+1}\\
&+ \int \log h_{C^\circ}(u)\,d\nu_2(u) - \frac{1}{n+1}\log \int \rho_{C^\circ}^{n+1}(u)\,d\sigma(u) - \frac{\log |B_2^{n+1}|}{n+1} \\
&\leq - \frac{\log\left(c_{n+1}^S\right)}{n+1}+\int \log h_C(u)\,d\nu_1(u)-\int \log \rho_C(u)\,d\nu_2(u).
\end{align*}
Optimizing over $C \in \mathrm{Conv}_s(\R^{n+1})$ and using Lemma \ref{lem:Oliker} completes the proof.
\endproof

The following log-Sobolev type inequality can be deduced from Proposition \ref{prop:KUB}.

\begin{thm}\label{thm:LSIimproved}
Let $\nu_1,\nu_2 \in \mathcal{P}_s(\Sd^n)$ satisfy the strict subspace concentration inequality; if $\eta_1= e^{-V_1}\sigma,\eta_2=e^{-V_2}\sigma$ are minimizers of $F_{\nu_1}$ and $F_{\nu_2}$, then it holds
\begin{align}\label{eq:LSIimproved}
\notag &H(\eta_1 | \sigma)+ H(\eta_2 | \sigma) + (n+1)\mathcal{T}_\alpha(\nu_1,\nu_2)\\
&\leq d_{n+1} + \frac{n+1}{2}\int \log\left(1+ \frac{|\nabla_{\Sd^n}V_1|^2}{(n+1)^2} \right)e^{-V_{1}}\,d\sigma 
+ \frac{n+1}{2}\int \log\left(1+ \frac{|\nabla_{\Sd^n}V_2|^2}{(n+1)^2} \right)e^{-V_{2}}\,d\sigma,
\end{align}
with $d_{n+1} = \log \left(\frac{|B_2^{n+1}|^2}{c_{n+1}^S}\right)$.
\end{thm}
\begin{rem}\
\begin{itemize}
\item With the notation of Lemma \ref{lem:minimizers}, there exist $C_1,C_2 \in \mathrm{Conv}_s(\R^{n+1})$ with unit volume such that, for $i=1,2$, $\eta_i = \eta_{C_i}$ and $V_i = -(n+1)\log \rho_{C_i} - \log |B_2^{n+1}|$. In particular, the functions $V_i$ are differentiable almost everywhere on $\Sd^n$.
\item It is well known that the uniform probability measure $\sigma$ on $\Sd^n$ satisfies the following log-Sobolev inequality: for all $d\eta = e^{-V}\,d\sigma$ with a smooth potential $V:\Sd^n \to \R$, 
\begin{equation}\label{eq:LSIsphere}
H(\eta | \sigma)\leq \frac{1}{2n}\int |\nabla_{\Sd^n}V|^2e^{-V}\,d\sigma.
\end{equation}
The constant $n$ in \eqref{eq:LSIsphere} is sharp (and corresponds to the spectral gap of the Laplace operator). 
In particular, if $\eta_1,\eta_2 \in \mathcal{P}(\Sd^n)$ have smooth densities of the form $e^{-V_i}$, $i=1,2$, then 
\begin{equation}\label{eq:LSIsphere2}
H(\eta_1 | \sigma)+H(\eta_2 | \sigma)\leq \frac{1}{2n}\int |\nabla_{\Sd^n}V_1|^2e^{-V_1}\,d\sigma+\frac{1}{2n}\int |\nabla_{\Sd^n}V_2|^2e^{-V_2}\,d\sigma.
\end{equation}
Using the inequality $\log(1+x) \leq x$, $x>-1$, one immediately sees that \eqref{eq:LSIimproved} improves \eqref{eq:LSIsphere2} in the case where $\eta_1,\eta_2$ are minimizers of $F_{\nu_1}$, $F_{\nu_2}$ with $(n+1)\mathcal{T}_\alpha(\nu_1,\nu_2) \geq d_{n+1}$. 
\item We do not know if the inequality 
\[
H(\eta | \sigma)\leq \frac{n+1}{2}\int \log\left(1+ \frac{|\nabla_{\Sd^n}V|^2}{(n+1)^2} \right)e^{-V}\,d\sigma
\]
is true for say symmetric probability measures $d\eta = e^{-V}\,d\sigma$ with a smooth potential $V:\Sd^n \to \R$ (and constant $n$ instead of $n+1$ without the evenness assumption). 
Since $\Sd^n$ satisfies the curvature-dimension criterion $\mathrm{CD}(n-1,n)$, the inequality 
\begin{equation}\label{eq:LSI-log}
H(\eta | \sigma)\leq \frac{n}{2}\log\left( 1+\frac{1}{(n-1)n}\int |\nabla_{\Sd^n}V|^2 e^{-V}\,d\sigma\right)
\end{equation}
holds true for all probability measures $\eta$ (see \cite{BGL}). 
In \cite[Theorem1.1]{DEKL}, one can also find the following variant of \eqref{eq:LSI-log}
\[
H(\eta | \sigma)\leq \frac{4}{\gamma_1^*}\log\left( 1+\frac{\gamma_1^*}{8 n}\int  |\nabla_{\Sd^n}V|^2e^{-V}\,d\sigma\right)
\]
with $\gamma_1^* = \frac{4n-1}{(n+1)^2}$. Note that, contrary to \eqref{eq:LSI-log}, this inequality gives back \eqref{eq:LSIsphere} with the sharp constant $2n.$
\end{itemize}
\end{rem}

We will need the following elementary result.
\begin{lem}\label{lem:cone2}
Suppose $C$ is a centrally symmetric convex body of volume $1$ containing $0$ and let $\nu_C$ be its cone measure. Then the map $T :\Sd^n \to \Sd^n : u\mapsto N_C(\rho_C(u)u)$ transports $
d\eta_C(x) = |B_2^{n+1}|\rho_C^{n+1}(x)\,d\sigma(x)$
onto $\nu_C$ and is optimal for $\mathcal{T}_\alpha.$ Moreover
\[
\mathcal{T}_\alpha (\eta_C,\nu_C) = \int_{\Sd^n} \log \left(\frac{|\nabla_{\R^{n+1}} \rho_C(u)|}{\rho_C(u)}\right)\,d\eta_C(u).
\]
\end{lem}
\proof
By definition of $\nu_C$, for any $f$ bounded measurable function $f$ on $\Sd^n$, using polar coordinates yields
\begin{align*}
\int_{\Sd^n} f(y)\,d\nu_C(y) = \int_{C} f(T(x))\,dx&= |B_2^{n+1}| \int_0^{+\infty} \int_{\Sd^n} f(T(u))\mathbf{1}_{\{r\leq \rho_C(u)\}} (n+1)r^{n}\,dr d\sigma(u)\\
& = \int_{\Sd^n} f(T(u))\, d\eta_C(u),
\end{align*}
and so $\nu_C$ is the pushforward of $\eta_C$ under the map $T$.
Let us show that $T$ is optimal for the transport cost $\mathcal{T}_\alpha$. Indeed, using the inequality
\begin{equation}\label{eq:rho-h}
- \log x\cdot y \geq \log \rho_C(x) - \log h_C(y),\qquad \forall x,y \in \Sd^n,
\end{equation}
one gets
\begin{align*}
\mathcal{T}_\alpha(\eta_C,\nu_C) & \geq \int \log \rho_C(x)\,d\eta_C(x) - \int \log h_C(y)\,d\nu_C(y)\\
& = \int \log \frac{\rho_C(x)}{h_C(N_C(x\rho_C(x)))}\,d\eta_C(x)\\
& = \int \log \frac{1}{x\cdot N_C(x\rho_C(x)))}\,d\eta_C(x)\\
& \geq \mathcal{T}_\alpha(\eta_C,\nu_C),
\end{align*}
where we used that for any $u \in N_C(z)$
\[
h_C(u) = z\cdot u,\qquad \forall z\in \partial C
\]
and that for $\eta_C$ almost all $x$, $N_C(x\rho_C(x))$ contains a single point.
Therefore,
\begin{equation}\label{eq:nuCmuC}
\mathcal{T}_\alpha(\eta_C,\nu_C) = \int \alpha(x,T(x))\,d\eta_C(x) = \int \log \rho_C(x)\,d\eta_C(x) - \int \log h_C(y)\,d\nu_C(y).
\end{equation}
For $\eta_C$ almost all $x \in \Sd^n$, it thus holds 
\[
h_C(T(x)) = \rho_C(x) x\cdot T(x)
\]
and so, using that $h_{C}^{-1} = \rho_{C^\circ}$ and $\rho_C^{-1} = h_{C^\circ}$, one gets
\[
h_{C^\circ}(x) =  x \cdot \rho_{C^\circ} (T(x))T(x).
\]
So, for $\eta_C$ almost all $x$, the vector $\rho_{C^\circ} (T(x)) T(x)$ is a subgradient of $h_{C^\circ}$ at $x$. The set where $h_{C^\circ}$ is differentiable being of $\eta_C$ measure $1$, one gets that
\[
\nabla_{\R^{n+1}} h_{C^\circ} (x) = \rho_{C^\circ} (T(x)) T(x),
\]
for $\eta_C$ almost all $x\in \Sd^n$.
Since $|T(x)|=1$, one gets $\rho_{C^\circ}(T(x)) = |\nabla_{\R^{n+1}} h_{C^\circ}(x)|$ and so $x\cdot T(x) = \frac{h_{C^\circ} (x)}{|\nabla_{\R^{n+1}} h_{C^\circ}(x)|}$. Using again that $h_{C^\circ} = 1/\rho_{C}$, one gets finally $x\cdot T(x) = \frac{\rho_C(x)}{|\nabla_{\R^{n+1}} \rho_{C}(x)|}$, which completes the proof.
\endproof

We are now ready to prove Theorem \ref{thm:LSIimproved}.
\proof[Proof of Theorem \ref{thm:LSIimproved}]
Let $\nu_1,\nu_2 \in \mathcal{P}_s(\Sd^n)$ satisfy the strict subspace concentration inequality, and let $\eta_1,\eta_2$ be minimizers of $F_{\nu_1}, F_{\nu_2}$, which exist according to Theorem 6.3 of \cite{BLYZ13}. 
According to Proposition \ref{prop:KUB}, it holds
\[
H(\eta_1|\sigma) + H(\eta_2|\sigma)+(n+1) \mathcal{T}_{\alpha}(\nu_1,\nu_2) \leq d_{n+1} + (n+1)\mathcal{T}_\alpha(\eta_1,\nu_1) + (n+1)\mathcal{T}_\alpha(\eta_2,\nu_2).
\]
Now, according to Lemma \ref{lem:cone2}, one gets
\[
(n+1)\mathcal{T}_\alpha(\eta_i,\nu_i) = (n+1)\int_{\Sd^n} \log \left(\frac{|\nabla_{\R^{n+1}} \rho_i(u)|}{\rho_i(u)}\right)\,d\eta_i(u),
\]
where $\eta_i = |B_2^{n+1}|\rho_i^{n+1}\,d\sigma$ and $\rho_i$ is the radial function of some unit volume $C_i \in \mathrm{Conv}_s(\R^{n+1})$. Letting $V_i= - \log |B_2^{n+1}| - (n+1) \log \rho_i$, we see that, for all $x\in \Sd^n$ at which $\rho_i$ is differentiable, $\frac{\nabla_{\R^{n+1}} \rho_i}{\rho_i}(x) = - \frac{1}{n+1} \nabla_{\R^{n+1}} V_i(x)$, and by projection on $x^\perp$
\[
\frac{\nabla_{\Sd^n} \rho_i}{\rho_i}(x) = - \frac{1}{n+1} \nabla_{\Sd^n} V_i(x).
\]
Since $\rho_i$ is $-1$ homogeneous, one gets
\[
\nabla_{\R^{n+1}} \rho_i(x) = \nabla_{\Sd^n} \rho_i(x)+ \left(\nabla \rho_i(x)\cdot x\right)x = \nabla_{\Sd^n} \rho_i(x) -\rho_i(x)x
\]
and so 
\[
\frac{|\nabla_{\R^{n+1}} \rho_i|(x)}{\rho_i(x)} = \sqrt{\frac{|\nabla_{\Sd^n} \rho_i|^2}{\rho_i^2}(x)+1} = \sqrt{\frac{|\nabla_{\Sd^n} V_i|^2(x)}{(n+1)^2}+1},
\]
at all point $x \in \Sd^n$ where $\rho_i$ is differentiable. This set of points being of full measure, this completes the proof.
\endproof

\subsection{Remarks on the log-Minkowski conjecture}
The following log-Minkowski inequality has been conjectured in \cite{BLYZ12} (in relation to an equivalent log-Brunn-Minkowski inequality).
\begin{conj}[log-Minkowski inequality]\label{conj:LogM}
For all $C,D \in \mathrm{Conv}_s(\R^{n+1})$ with unit volume, it holds
\[
\int \log\left(\frac{h_D}{h_C}\right)\,d\nu_C \geq 0,
\]
where $\nu_C$ is the cone measure of $C$.
\end{conj}
This conjectured inequality is known to be true in dimension 2 \cite{BLYZ12}, or when $C$ and $D$ have a lot of symmetries \cite{Sar15,BK22}.

Note that Conjecture \ref{conj:LogM} is equivalent to the following property: if $C \in \mathrm{Conv}_s(\R^{n+1})$ has unit volume, then $C$ minimizes $\Phi_{\nu_C}$, or equivalently (using Lemma \ref{lem:minimizers}), $\eta_C(dx) = |B_2^{n+1}|\rho_C^{n+1}(x)\,\sigma(dx)$ minimizes $F_{\nu_C}$. This remark, immediately leads to the following version of Theorem \ref{thm:LSIimproved}:

\begin{thm}\label{thm:LSIimproved2}
 If Conjecture \ref{conj:LogM} holds true in $\R^{n+1}$, then for all $C_1,C_2 \in \mathrm{Conv}_s(\R^{n+1})$ with unit volume, it holds 
\begin{align}\label{eq:LSIimproved2}
\notag &H(\eta_{C_1} | \sigma)+ H(\eta_{C_2} | \sigma) + (n+1)\mathcal{T}_\alpha(\nu_{C_1},\nu_{C_2})\\
&\leq d_{n+1} + \frac{n+1}{2}\int \log\left(1+ \frac{|\nabla_{\Sd^n}V_1|^2}{(n+1)^2} \right)e^{-V_{1}}\,d\sigma 
+ \frac{n+1}{2}\int \log\left(1+ \frac{|\nabla_{\Sd^n}V_2|^2}{(n+1)^2} \right)e^{-V_{2}}\,d\sigma,
\end{align}
where, for $i=1,2$, $d\eta_{C_i} = |B_2^{n+1}| \rho_{C_i}^{n+1}\,d\sigma := e^{-V_i}\,d\sigma$ and $d_{n+1} = \log \left(\frac{|B_2^{n+1}|^2}{c_{n+1}^S}\right)$.

\end{thm}

\proof
It suffices to apply Theorem \ref{thm:LSIimproved} to $\nu_1 = \nu_{C_1}$ and $\nu_2=\nu_{C_2}$ and to use the fact, explained above, that $\eta_{C_1},\eta_{C_2}$ are minimizers of $F_{\nu_1}$ and $F_{\nu_2}$. 
\endproof

One can take advantage of the fact that both Mahler and log-Minkowski conjectures \eqref{conj:LogM} hold true when $C,D$ are unconditional to get the following result.

\begin{thm}\label{thm:LSIimproved-unconditional}
For all unconditional convex bodies $C_1,C_2 \subset \R^{n+1}$ with unit volume, it holds 
 \begin{align}\label{eq:LSIimproved-unconditional}
 &H(\eta_{C_1} | \sigma)+ H(\eta_{C_2} | \sigma) + (n+1)\mathcal{T}_\alpha(\nu_{C_1},\nu_{C_2})\\
\notag&\leq e_{n+1} + \frac{n+1}{2}\int \log\left(1+ \frac{|\nabla_{\Sd^n}V_1|^2}{(n+1)^2} \right)e^{-V_{1}}\,d\sigma 
+ \frac{n+1}{2}\int \log\left(1+ \frac{|\nabla_{\Sd^n}V_2|^2}{(n+1)^2} \right)e^{-V_{2}}\,d\sigma,
\end{align}
where, for $i=1,2$, $d\eta_{C_i} = |B_2^{n+1}| \rho_{C_i}^{n+1}\,d\sigma := e^{-V_i}\,d\sigma$ and $e_{n+1}= \log \left(\frac{(n+1)!|B_2^{n+1}|^2}{4^{n+1}}\right)$.
\end{thm}

\begin{rem} Using that $|B_2^k| = \frac{\pi^{k/2}}{\Gamma(\frac{k}{2}+1)}$ one sees that
$e_{n+1} \sim (n+1) \log \left(\frac{\pi}{2}\right)$,
as $n \to \infty$. The sharpness of the Log-Sobolev type inequality of Theorem \ref{thm:LSIimproved-unconditional} is discussed in Remark \ref{rem:sharpness} below.
\end{rem}

\proof
For any $\nu \in \mathcal{P}_s(\Sd^n)$, define 
\[
\tilde{K}(\nu) = \inf \{\Phi_\nu(C)\},
\]
where the infimum runs over the set of \emph{unconditional} convex bodies $C \subset \R^{n+1}$ of volume $1$.
Since the inequality
\[
|C||C^\circ| \geq \frac{4^{n+1}}{(n+1)!}
\]
holds true for all unconditional convex body $C \subset \R^{n+1}$, repeating the proof of Proposition \ref{prop:KUB} leads to
\begin{equation}\label{eq:LSIimproved-unconditional}
(n+1)\tilde{K}(\nu_1)+(n+1)\tilde{K}(\nu_2) \leq e_{n+1}-(n+1)\mathcal{T}_\alpha(\nu_1,\nu_2),
\end{equation}
for all $\nu_1,\nu_2 \in \mathcal{P}_s(\Sd^n)$.
According to \cite[Corollary 1.3]{Sar15}, Conjecture \ref{conj:LogM} holds true whenever $C,D$ are unconditional. Therefore, for any unconditional convex body $C \subset \R^{n+1}$, it holds $\tilde{K}(\nu_C) = \Phi_{\nu_C}(C)$. 
Moreover,
\begin{align*}
\Phi_{\nu_C}(C) & = - \left[-\int \log h_{C}\,d\nu + \int \log \rho_{C}\,d\eta_{C}\right] + \frac{H(\eta_{C}|\sigma)}{n+1}- \frac{\log |B_2^{n+1}|}{n+1}\\
& =  - \mathcal{T}_{\alpha}(\nu_C,\eta_C) + \frac{H(\eta_{C}|\sigma)}{n+1}- \frac{\log |B_2^{n+1}|}{n+1}.
\end{align*}
Applying \eqref{eq:LSIimproved-unconditional} with, for $i=1,2$, $\nu_i = \nu_{C_i}$ and $C_i$ unconditional of volume $1$ yields 
\[
H(\eta_{C_1} | \sigma)+ H(\eta_{C_2} | \sigma) + (n+1)\mathcal{T}_\alpha(\nu_{C_1},\nu_{C_2})\leq e_{n+1} + (n+1)\mathcal{T}_{\alpha}(\nu_{C_1},\eta_{C_1})+(n+1)\mathcal{T}_{\alpha}(\nu_{C_2},\eta_{C_2}).
\]
The proof is then completed exactly as the one of Theorem \ref{thm:LSIimproved}.
\endproof
\begin{rem}
The equality $\tilde{K}(\nu) = K(\nu)$ for all unconditional probability measure $\nu$ would have enabled us to shorten the preceding proof, but we do not know how to prove it. 
\end{rem}

It turns out that the log-Sobolev type inequalities obtained in Theorems \ref{thm:LSIimproved2} and \ref{thm:LSIimproved-unconditional} imply back a reverse Blaschke-Santal\'o inequality. 

\begin{thm}\label{thm:LSItoMahler}\
\begin{itemize}
\item[(a)]If, for some constant $d>0$, the inequality
\begin{align}\label{eq:LSItoMahler}
\notag &H(\eta_{C_1} | \sigma)+ H(\eta_{C_2} | \sigma) + (n+1)\mathcal{T}_\alpha(\nu_{C_1},\nu_{C_2})\\
&\leq d + \frac{n+1}{2}\int \log\left(1+ \frac{|\nabla_{\Sd^n}V_1|^2}{(n+1)^2} \right)e^{-V_{1}}\,d\sigma 
+ \frac{n+1}{2}\int \log\left(1+ \frac{|\nabla_{\Sd^n}V_2|^2}{(n+1)^2} \right)e^{-V_{2}}\,d\sigma,
\end{align}
holds true for all $C_1,C_2 \in \mathrm{Conv}_s(\R^{n+1})$ with unit volume, then the following inequality holds true: for all $C \in \mathrm{Conv}_{s}(\R^{n+1})$,
\begin{equation}\label{eq:improvedMahler}
|C||C^\circ| \geq c \exp\left((n+1) \mathcal{T}_\alpha(\nu_{C_1},\nu_{C_2})+ \int \log h_{C_1}^{n+1}\,d\nu_{C_1} - \int \log \rho_{C_1}^{n+1}d\nu_{C_2}\right),
\end{equation}
where $c=e^{-d}|B_2^{n+1}|^2$ and $\nu_{C_1}$ and $\nu_{C_2}$ are the cone probability measures of $C_1=\frac{1}{|C|^{1/(n+1)}}C$ and $C_2=\frac{1}{|C^\circ|^{1/(n+1)}}C^\circ$.
\item[(b)]If $C \in \mathrm{Conv}_s(\R^{n+1})$ is unconditional, then \eqref{eq:improvedMahler} holds without restriction with the constant $c= \frac{4^{n+1}}{(n+1)!}$.
\end{itemize}
\end{thm}

\proof
Inequality \eqref{eq:LSItoMahler} reads,
\begin{multline*}
(n+1)\mathcal{T}_\alpha(\eta_{C_1},\nu_{C_1}) - H(\eta_{C_1}|\sigma) + (n+1)\mathcal{T}_\alpha(\eta_{C_2},\nu_{C_2}) - H(\eta_{C_2}|\sigma)
\geq \mathcal{T}_\alpha(\nu_{C_1},\nu_{C_2})-d,
\end{multline*}
for all $C_1,C_2 \in \mathrm{Conv}_s(\R^{n+1})$ with unit volume.
Let us apply this inequality to $C_1=\frac{1}{|C|^{1/(n+1)}}C$ and $C_2=\frac{1}{|C^\circ|^{1/(n+1)}}C^\circ$, where $C$ is some centrally symmetric convex body. Using \eqref{eq:nuCmuC}, one gets that
\[
(n+1)\mathcal{T}_\alpha(\eta_{C_1},\nu_{C_1}) - H(\eta_{C_1}|\sigma) = -\int \log h_{C}^{n+1}\,d\nu + \log |C| - \log |B_2^{n+1}|
\]
and 
\[
(n+1)\mathcal{T}_\alpha(\eta_{C_2},\nu_{C_2}) - H(\eta_{C_2}|\sigma) = -\int \log h_{C^\circ}^{n+1}\,d\nu + \log |C^\circ| - \log |B_2^{n+1}|.
\]
Since $-d = \log c - 2 \log |B_2^{n+1}|$, the proof of $(a)$ is complete.

According to Theorem \ref{thm:LSIimproved-unconditional}, in the unconditional case, \eqref{eq:LSItoMahler} is true with the constant $d=e_{n+1}=\log \left(\frac{(n+1)!|B_2^{n+1}|^2}{4^{n+1}}\right)$. Thus repeating the preceding arguments, we see that \eqref{eq:improvedMahler} holds for all unconditional bodies $C$ with the constant $c=\frac{4^{n+1}}{(n+1)!}$, which proves $(b)$.
\endproof

\begin{rem}\label{rem:sharpness}\ 
\begin{itemize}
\item Assuming the Log-Minkowski conjecture holds true in $\R^{n+1}$, it follows from Theorems \ref{thm:LSIimproved2} and \ref{thm:LSItoMahler} (Item $(a)$) that \eqref{eq:improvedMahler} holds with the constant $c=c^S_{n+1}$. According to Lemma \ref{lem:Oliker}, the exponential factor in the right-hand side of \eqref{eq:improvedMahler} is greater than or equal to $1$. Note that this term can be strictly greater than $1$, because there is no reason in general that the function $-\log h_{C_1}$ is a dual optimizer for the transport between $\nu_{C_1}$ and $\nu_{C_2}$. 
\item If $C_1 = c_1B_\infty^{n+1}$ and $C_2 = c_2B_1^{n+1}$ where $c_1,c_2$ are positive constants ensuring the bodies have volume $1$, then there is equality in \eqref{eq:LSIimproved-unconditional}. Indeed, $C=B_\infty^{n+1}$ is such that $|B_\infty^{n+1}| |(B_\infty^{n+1})^\circ| = \frac{4^{n+1}}{(n+1)!}$ and so Inequality \eqref{eq:improvedMahler} (with $c=\frac{4^{n+1}}{(n+1)!}$) implies that 
\begin{equation}\label{eq:future-work}
\mathcal{T}_\alpha(\nu_{C_1},\nu_{C_2})+ \int \log h_{C_1}^{n+1}\,d\nu_{C_1} - \int \log \rho_{C_1}^{n+1}d\nu_{C_2}=0.
\end{equation}
Plugging this relation in the proof above, one gets equality in \eqref{eq:LSIimproved-unconditional}. This shows that the conclusion of Theorem \ref{thm:LSIimproved-unconditional} cannot be improved in general.
\item A similar reasoning shows that more generally all couples $(C_1,C_2)$ with $C_1 = \frac{1}{|C|^{1/(n+1)}}C$, $C_2 =  \frac{1}{|C^\circ|^{1/(n+1)}}C^\circ$ with $C$ being a Hanner polytope are such that \eqref{eq:future-work} holds (and are equality cases in \eqref{eq:LSIimproved-unconditional}). Characterizing the class of convex bodies $(C_1,C_2)$ for which \eqref{eq:future-work} holds is a challenging question that will be considered elsewhere.
\end{itemize}
\end{rem}

Putting together the conclusions of Theorems \ref{thm:LSIimproved2} and \ref{thm:LSItoMahler} (Item $(a)$) finally yields the following result.
\begin{thm}\label{thm:LSIequivMahler}
 If Conjecture \ref{conj:LogM} holds true in $\R^{n+1}$, then the constant $c_{n+1}^S$ is the best constant $c>0$ (that is the greatest) in the inequality
    \begin{align*}
     &H(\eta_{C_1} | \sigma)+ H(\eta_{C_2} | \sigma) + (n+1)\mathcal{T}_\alpha(\nu_{C_1},\nu_{C_2})\\
    &\notag\quad\leq \log \left(\frac{|B_2^{n+1}|^2}{c}\right) + \frac{n+1}{2}\int \log\left(1+ \frac{|\nabla_{\Sd^n}V_1|^2}{(n+1)^2} \right)e^{-V_{1}}\,d\sigma 
    + \frac{n+1}{2}\int \log\left(1+ \frac{|\nabla_{\Sd^n}V_2|^2}{(n+1)^2} \right)e^{-V_{2}}\,d\sigma,
    \end{align*}
    where $C_1,C_2 \subset \R^{n+1}$ are arbitrary centrally symmetric convex bodies with unit volume and, for $i=1,2$, $d\eta_{C_i} = |B_2^{n+1}| \rho_{C_i}^{n+1}\,d\sigma := e^{-V_i}\,d\sigma$ and $\nu_{C_i}$ is the cone measure of $C_i$.
\end{thm}


\begin{thebibliography}{AAKM04}

\bibitem[AAFM12]{AAFM12}
S.~Artstein-Avidan, D.~Florentin, and V.~Milman.
\newblock Order isomorphisms on convex functions in windows.
\newblock In {\em Geometric Aspects of Functional Analysis: Israel Seminar
  2006--2010}, pages 61--122. Springer, 2012.

\bibitem[AAKM04]{AAKM05}
S.~Artstein-Avidan, B.~Klartag, and V.~Milman.
\newblock The {S}antal\'{o} point of a function, and a functional form of the
  {S}antal\'{o} inequality.
\newblock {\em Mathematika}, 51(1-2):33--48 (2005), 2004.

\bibitem[AAM08]{AAM08}
S.~Artstein-Avidan and V.~Milman.
\newblock The concept of duality for measure projections of convex bodies.
\newblock {\em Journal of Functional Analysis}, 254(10):2648--2666, 2008.

\bibitem[AASW23]{AASW23}
S.~Artstein-Avidan, S.~Sadovsky, and K.~Wyczesany.
\newblock A zoo of dualities.
\newblock {\em The Journal of Geometric Analysis}, 33(8):238, 2023.

\bibitem[Bal86]{ball86}
K.~Ball.
\newblock {\em Isometric problems in {$l_p$} and sections of convex sets}.
\newblock PhD thesis, Cambridge, 1986.

\bibitem[Ber16]{Ber16}
J.~Bertrand.
\newblock Prescription of {Gauss} curvature using optimal mass transport.
\newblock {\em Geom. Dedicata}, 183:81--99, 2016.

\bibitem[Ber21]{ber20b}
B.~Berndtsson.
\newblock Complex integrals and {K}uperberg's proof of the {B}ourgain-{M}ilman
  theorem.
\newblock {\em Adv. Math.}, 388:Paper No. 107927, 10, 2021.

\bibitem[Ber22]{ber20a}
B.~Berndtsson.
\newblock Bergman kernels for {P}aley-{W}iener spaces and {N}azarov's proof of
  the {B}ourgain-{M}ilman theorem.
\newblock {\em Pure Appl. Math. Q.}, 18(2):395--409, 2022.

\bibitem[BF13]{bf13}
F.~Barthe and M.~Fradelizi.
\newblock The volume product of convex bodies with many hyperplane symmetries.
\newblock {\em Amer. J. Math.}, 135(2):311--347, 2013.

\bibitem[BG99]{BG99}
S.~G. Bobkov and F.~G{\"o}tze.
\newblock Exponential integrability and transportation cost related to
  logarithmic {Sobolev} inequalities.
\newblock {\em J. Funct. Anal.}, 163(1):1--28, 1999.

\bibitem[BGL14]{BGL}
D.~Bakry, I.~Gentil, and M.~Ledoux.
\newblock {\em Analysis and geometry of {Markov} diffusion operators}, volume
  348 of {\em Grundlehren Math. Wiss.}
\newblock Cham: Springer, 2014.

\bibitem[BK22]{BK22}
K.~J. B\"{o}r\"{o}czky and P.~Kalantzopoulos.
\newblock Log-{B}runn-{M}inkowski inequality under symmetry.
\newblock {\em Trans. Amer. Math. Soc.}, 375(8):5987--6013, 2022.

\bibitem[Bla23]{blaschke_book}
W.~Blaschke.
\newblock {\em Vorlesungen \"uber Differentialgeometrie II: Affine
  Differentialgeometrie}.
\newblock Springer-Verlag, Berlin, 1923.

\bibitem[Blo14]{blo14}
Z.~Blocki.
\newblock A lower bound for the {B}ergman kernel and the {B}ourgain-{M}ilman
  inequality.
\newblock In {\em Geometric aspects of functional analysis}, volume 2116 of
  {\em Lecture Notes in Math.}, pages 53--63. Springer, Cham, 2014.

\bibitem[BLYZ12]{BLYZ12}
K.~B{\"o}r{\"o}czky, E.~Lutwak, D.~Yang, and G.~Zhang.
\newblock The log-{Brunn}-{Minkowski} inequality.
\newblock {\em Adv. Math.}, 231(3-4):1974--1997, 2012.

\bibitem[BLYZ13]{BLYZ13}
K.~B\"{o}r\"{o}czky, E.~Lutwak, D.~Yang, and G.~Zhang.
\newblock The logarithmic {M}inkowski problem.
\newblock {\em J. Amer. Math. Soc.}, 26(3):831--852, 2013.

\bibitem[BM87]{bm87}
J.~Bourgain and V.~D. Milman.
\newblock New volume ratio properties for convex symmetric bodies in {${\bf
  R}^n$}.
\newblock {\em Invent. Math.}, 88(2):319--340, 1987.

\bibitem[CE17]{CE17}
D.~Cordero-Erausquin.
\newblock Transport inequalities for log-concave measures, quantitative forms
  and applications.
\newblock {\em Can. J. Math.}, 69(3):481--501, 2017.

\bibitem[CEK15]{CEK15}
D.~Cordero-Erausquin and B.~Klartag.
\newblock Moment measures.
\newblock {\em J. Funct. Anal.}, 268(12):3834--3866, 2015.

\bibitem[CER]{CER}
D.~Cordero-Erausquin and L.~Rotem.
\newblock Improved log-concavity for rotationally invariant measures of
  symmetric convex sets.
\newblock To appear in Annals of Probability.

\bibitem[DEKL14]{DEKL}
J.~Dolbeault, M.~J. Esteban, M.~Kowalczyk, and M.~Loss.
\newblock Sharp interpolation inequalities on the sphere: new methods and
  consequences.
\newblock In {\em Partial differential equations. Theory, control and
  approximation. In honor of the scientific heritage of Jacques-Louis Lions.
  Selected papers based on the presentations at the international conference on
  partial differential equations: theory, control and approximation, Shanghai,
  China, May 28 -- June 1, 2012}, pages 225--242. Berlin: Springer, 2014.

\bibitem[EKS15]{EKS15}
M.~Erbar, K.~Kuwada, and K.-T. Sturm.
\newblock On the equivalence of the entropic curvature-dimension condition and
  {Bochner}'s inequality on metric measure spaces.
\newblock {\em Invent. Math.}, 201(3):993--1071, 2015.

\bibitem[Fat18]{fat18}
M.~Fathi.
\newblock A sharp symmetrized form of {T}alagrand's transport-entropy
  inequality for the {G}aussian measure.
\newblock {\em Electron. Commun. Probab.}, 23:Paper No. 81, 9, 2018.

\bibitem[FGJ17]{FGJ17}
J.~Fontbona, N.~Gozlan, and J.~F. Jabir.
\newblock A variational approach to some transport inequalities.
\newblock {\em Ann. Inst. Henri Poincar\'{e} Probab. Stat.}, 53(4):1719--1746,
  2017.

\bibitem[FGMR10]{fgmr10}
M.~Fradelizi, Y.~Gordon, M.~Meyer, and S.~Reisner.
\newblock The case of equality for an inverse {S}antal\'{o} functional
  inequality.
\newblock {\em Adv. Geom.}, 10(4):621--630, 2010.

\bibitem[FGZ23]{FGZ23}
M.~Fradelizi, N.~Gozlan, and S.~Zugmeyer.
\newblock Transport proofs of some functional inverse santal{\'o} inequalities.
\newblock In Rados{\l}aw Adamczak, Nathael Gozlan, Karim Lounici, and Mokshay
  Madiman, editors, {\em High Dimensional Probability IX}, pages 123--142,
  Cham, 2023. Springer International Publishing.

\bibitem[FHM{\etalchar{+}}22]{fhmrz21}
M.~Fradelizi, A.~Hubard, M.~Meyer, E.~Rold\'{a}n-Pensado, and A.~Zvavitch.
\newblock Equipartitions and {M}ahler volumes of symmetric convex bodies.
\newblock {\em Amer. J. Math.}, 144(5):1201--1219, 2022.

\bibitem[FM07]{fm07}
M.~Fradelizi and M.~Meyer.
\newblock Some functional forms of {B}laschke-{S}antal\'{o} inequality.
\newblock {\em Math. Z.}, 256(2):379--395, 2007.

\bibitem[FM08a]{fm08b}
M.~Fradelizi and M.~Meyer.
\newblock Increasing functions and inverse {S}antal\'{o} inequality for
  unconditional functions.
\newblock {\em Positivity}, 12(3):407--420, 2008.

\bibitem[FM08b]{fm08a}
M.~Fradelizi and M.~Meyer.
\newblock Some functional inverse {S}antal\'{o} inequalities.
\newblock {\em Adv. Math.}, 218(5):1430--1452, 2008.

\bibitem[GL10]{GL10}
N.~Gozlan and C.~L{\'e}onard.
\newblock Transport inequalities. {A} survey.
\newblock {\em Markov Process. Relat. Fields}, 16(4):635--736, 2010.

\bibitem[GMR88]{GMR88}
Y.~Gordon, M.~Meyer, and S.~Reisner.
\newblock Zonoids with minimal volume-product---a new proof.
\newblock {\em Proc. Amer. Math. Soc.}, 104(1):273--276, 1988.

\bibitem[Goz07]{Goz07}
N.~Gozlan.
\newblock Characterization of {Talagrand}'s like transportation-cost
  inequalities on the real line.
\newblock {\em J. Funct. Anal.}, 250(2):400--425, 2007.

\bibitem[Goz22]{Gozlan21}
N.~Gozlan.
\newblock The deficit in the {G}aussian log-{S}obolev inequality and inverse
  {S}antal\'{o} inequalities.
\newblock {\em Int. Math. Res. Not. IMRN}, (17):13396--13446, 2022.

\bibitem[GPV14]{gpv14}
A.~Giannopoulos, G.~Paouris, and B.~H. Vritsiou.
\newblock The isotropic position and the reverse {S}antal\'{o} inequality.
\newblock {\em Israel J. Math.}, 203(1):1--22, 2014.

\bibitem[Gro75]{Gro75}
L.~Gross.
\newblock Logarithmic {S}obolev inequalities.
\newblock {\em Amer. J. Math.}, 97(4):1061--1083, 1975.

\bibitem[IS20]{is20}
H.~Iriyeh and M.~Shibata.
\newblock Symmetric {M}ahler's conjecture for the volume product in the
  {$3$}-dimensional case.
\newblock {\em Duke Math. J.}, 169(6):1077--1134, 2020.

\bibitem[IS22]{is22}
H.~Iriyeh and M.~Shibata.
\newblock Minimal volume product of three dimensional convex bodies with
  various discrete symmetries.
\newblock {\em Discrete Comput. Geom.}, 68(3):738--773, 2022.

\bibitem[IW21]{iw22}
G.~Ivanov and E.~M. Werner.
\newblock Geometric representation of classes of concave functions and duality.
\newblock 2021.

\bibitem[KM05]{km05}
B.~Klartag and V.~D. Milman.
\newblock Geometry of log-concave functions and measures.
\newblock {\em Geom. Dedicata}, 112:169--182, 2005.

\bibitem[Kol20]{Kol20}
A.~V. Kolesnikov.
\newblock Mass transportation functionals on the sphere with applications to
  the logarithmic {Minkowski} problem.
\newblock {\em Mosc. Math. J.}, 20(1):67--91, 2020.

\bibitem[Kup08]{kup08}
G.~Kuperberg.
\newblock From the {M}ahler conjecture to {G}auss linking integrals.
\newblock {\em Geom. Funct. Anal.}, 18(3):870--892, 2008.

\bibitem[Led01]{Ledoux}
M.~Ledoux.
\newblock {\em The concentration of measure phenomenon}, volume~89 of {\em
  Mathematical Surveys and Monographs}.
\newblock American Mathematical Society, Providence, RI, 2001.

\bibitem[Leh09a]{lehec3}
J.~Lehec.
\newblock A direct proof of the functional {S}antal\'{o} inequality.
\newblock {\em C. R. Math. Acad. Sci. Paris}, 347(1-2):55--58, 2009.

\bibitem[Leh09b]{Lehec2}
J.~Lehec.
\newblock On the {Yao}-{Yao} partition theorem.
\newblock {\em Arch. Math.}, 92(4):366--376, 2009.

\bibitem[Leh09c]{Lehec1}
J.~Lehec.
\newblock Partitions and functional {Santal{\'o}} inequalities.
\newblock {\em Arch. Math.}, 92(1):89--94, 2009.

\bibitem[Lut91]{L91}
E.~Lutwak.
\newblock Extended affine surface area.
\newblock {\em Adv. Math.}, 85(1):39--68, 1991.

\bibitem[Mah39a]{Mah39b}
K.~Mahler.
\newblock Ein {M}inimalproblem f\"{u}r konvexe {P}olygone.
\newblock {\em Mathematica (Zutphen)}, 1939.

\bibitem[Mah39b]{Mah39a}
K.~Mahler.
\newblock Ein \"{U}bertragungsprinzip f\"{u}r konvexe {K}\"{o}rper.
\newblock {\em \v{C}asopis P\v{e}st. Mat. Fys.}, 68:93--102, 1939.

\bibitem[Mar86]{Mar86}
K.~Marton.
\newblock A simple proof of the blowing-up lemma.
\newblock {\em IEEE Trans. Inf. Theory}, 32:445--446, 1986.

\bibitem[Mar96a]{Mar96a}
K.~Marton.
\newblock Bounding {{$\bar d$}}-distance by informational divergence: {A}
  method to prove measure concentration.
\newblock {\em Ann. Probab.}, 24(2):857--866, 1996.

\bibitem[Mar96b]{Mar96b}
K.~Marton.
\newblock A measure concentration inequality for contracting {Markov} chains.
\newblock {\em Geom. Funct. Anal.}, 6(3):556--571, 1996.

\bibitem[Mey86]{mey86}
M.~Meyer.
\newblock Une caract\'{e}risation volumique de certains espaces norm\'{e}s de
  dimension finie.
\newblock {\em Israel J. Math.}, 55(3):317--326, 1986.

\bibitem[MW98]{MW98}
M.~Meyer and E.~Werner.
\newblock The {Santal{\'o}}-regions of a convex body.
\newblock {\em Trans. Am. Math. Soc.}, 350(11):4569--4591, 1998.

\bibitem[Naz12]{naz12}
F.~Nazarov.
\newblock The {H}\"{o}rmander proof of the {B}ourgain-{M}ilman theorem.
\newblock In {\em Geometric aspects of functional analysis}, volume 2050 of
  {\em Lecture Notes in Math.}, pages 335--343. Springer, Heidelberg, 2012.

\bibitem[Oli07]{Olik07}
V.~Oliker.
\newblock Embedding {{$\mathbb S^n$}} into {{$\mathbb R^{n+1}$}} with given
  integral gauss curvature and optimal mass transport on {{$\mathbb S^n$}}.
\newblock {\em Adv. Math.}, 213(2):600--620, 2007.

\bibitem[OV00]{OV00}
F.~Otto and C.~Villani.
\newblock Generalization of an inequality by {Talagrand} and links with the
  logarithmic {Sobolev} inequality.
\newblock {\em J. Funct. Anal.}, 173(2):361--400, 2000.

\bibitem[Rei86]{R86}
S.~Reisner.
\newblock Zonoids with minimal volume-product.
\newblock {\em Math. Z.}, 192(3):339--346, 1986.

\bibitem[Rei87]{R87}
S.~Reisner.
\newblock Minimal volume-product in {B}anach spaces with a {$1$}-unconditional
  basis.
\newblock {\em J. London Math. Soc. (2)}, 36(1):126--136, 1987.

\bibitem[Rot14]{Rot14}
L.~Rotem.
\newblock A sharp blaschke--santal{\'o} inequality for $\alpha$-concave
  functions.
\newblock {\em Geometriae Dedicata}, 172(1):217--228, 2014.

\bibitem[San49]{san49}
L.~A. Santal\'{o}.
\newblock An affine invariant for convex bodies of {$n$}-dimensional space.
\newblock {\em Portugal. Math.}, 8:155--161, 1949.

\bibitem[San16]{San16}
F.~Santambrogio.
\newblock Dealing with moment measures via entropy and optimal transport.
\newblock {\em J. Funct. Anal.}, 271(2):418--436, 2016.

\bibitem[Sar15]{Sar15}
C.~Saroglou.
\newblock Remarks on the conjectured log-{B}runn-{M}inkowski inequality.
\newblock {\em Geom. Dedicata}, 177:353--365, 2015.

\bibitem[Sch23]{sch23}
R.~Schneider.
\newblock Pseudo-cones.
\newblock {\em arXiv preprint arXiv:2305.00452}, 2023.

\bibitem[SR81]{sr81}
J.~Saint-Raymond.
\newblock Sur le volume des corps convexes sym\'{e}triques.
\newblock In {\em Initiation {S}eminar on {A}nalysis: {G}. {C}hoquet-{M}.
  {R}ogalski-{J}. {S}aint-{R}aymond, 20th {Y}ear: 1980/1981}, volume~46 of {\em
  Publ. Math. Univ. Pierre et Marie Curie}, pages Exp. No. 11, 25. Univ. Paris
  VI, Paris, 1981.

\bibitem[Tal96]{Tal96}
M.~Talagrand.
\newblock Transportation cost for {Gaussian} and other product measures.
\newblock {\em Geom. Funct. Anal.}, 6(3):587--600, 1996.

\bibitem[Vil09]{villani_book}
C.~Villani.
\newblock {\em Optimal transport, old and new}, volume 338 of {\em Grundlehren
  der Mathematischen Wissenschaften}.
\newblock Springer-Verlag, Berlin, 2009.

\bibitem[XLL23]{XLL23}
Y.~Xu, J.~Li, and G.~Leng.
\newblock Dualities and endomorphisms of pseudo-cones.
\newblock {\em Advances in Applied Mathematics}, 142:102434, 2023.

\end{thebibliography}

\newcommand{\etalchar}[1]{$^{#1}$}

\end{document}